\documentclass[10pt]{amsart}
\usepackage{amssymb,amstext,amsmath,amscd,amsthm,amsfonts,enumerate,graphicx,latexsym,color,stmaryrd}
\usepackage[all]{xy}

\newtheorem{thm}{Theorem}[section]
\newtheorem*{thm*}{Theorem}
\newtheorem{thmm}{Theorem}

\newtheorem{cor}[thm]{Corollary}
\newtheorem{lem}[thm]{Lemma}
\newtheorem{prop}[thm]{Proposition}
\theoremstyle{definition}
\newtheorem{conv}[thm]{Convention}
\newtheorem{dfn}[thm]{Definition}
\newtheorem*{dfn*}{Definition}
\newtheorem{rem}[thm]{Remark}
\newtheorem{ques}[thm]{Question}
\newtheorem{conj}[thm]{Conjecture}
\newtheorem*{conj*}{Conjecture}
\newtheorem{ex}[thm]{Example}
\newtheorem{nota}[thm]{Notation}
\theoremstyle{remark}

\newtheorem{claim}{Claim}
\newtheorem*{claim*}{Claim}
\renewcommand{\qedsymbol}{$\blacksquare$}
\numberwithin{equation}{thm}
\def\A{\mathcal{A}}
\def\adds{\operatorname{\mathsf{add^\Sigma}}}
\def\ann{\operatorname{Ann}}

\def\C{\mathcal{C}}
\def\c{\mathsf{c}}
\def\cc{\mathsf{cc}}

\def\cone{\operatorname{\mathsf{cone}}}
\def\Cpt{\mathbf{Cpt}}
\def\cpt{\mathsf{cpt}}
\def\D{\mathrm{D}}
\def\d{\operatorname{\mathsf{D}}}

\def\db{\operatorname{\mathsf{D^b}}}
\def\dbf{\operatorname{\mathsf{D^b_{fl}}}}

\def\df{\operatorname{\mathsf{D^{\mbox{\boldmath$-$}}_{fl}}}}
\def\dm{\operatorname{\mathsf{D^{\mbox{\boldmath$-$}}}}}
\def\E{\mathcal{E}}
\def\EE{\mathrm{E}}

\def\even{\mathrm{even}}
\def\Ext{\mathrm{Ext}}
\def\ge{\geqslant}
\def\h{\mathsf{H}}
\def\height{\operatorname{ht}}
\def\Hom{\mathrm{Hom}}
\def\hspec{\operatorname{Spec^h}}

\def\id{\mathrm{id}}
\def\ii{\mathbb{I}}
\def\im{\operatorname{Im}}
\def\inc{\mathsf{inc}}
\def\K{\mathcal{K}}
\def\k{\mathrm{K}}
\def\kb{\operatorname{\mathsf{K^b}}}
\def\km{\operatorname{\mathsf{K^{\mbox{\boldmath$-$}}}}}
\def\kmb{\operatorname{\mathsf{K^{\mbox{\boldmath$-$},b}}}}

\def\L{\mathcal{L}}
\def\le{\leqslant}
\def\ltensor{\otimes^{\mathbf{L}}}
\def\m{\mathfrak{m}}
\def\M{\mathcal{M}}
\def\Max{\operatorname{Max}}
\def\Min{\operatorname{Min}}

\def\mn{\operatorname{\mathtt{Mn}}}
\def\mx{\operatorname{\mathtt{Mx}}}
\def\N{\mathcal{N}}
\def\NN{\mathbb{N}}
\def\n{\mathfrak{n}}
\def\nil{\operatorname{nil}}
\def\odd{\mathrm{odd}}
\def\one{\mathbf{1}}
\def\P{\mathcal{P}}
\def\p{\mathfrak{p}}
\def\PP{\mathbb{P}}
\def\pp{\mathfrak{s}}
\def\proj{\operatorname{\mathsf{proj}}}
\def\Q{\mathcal{Q}}
\def\q{\mathfrak{q}}
\def\R{\mathcal{R}}

\def\Rad{\mathbf{Rad}}
\def\rad{\operatorname{rad}}
\def\red{\mathrm{red}}
\def\RHom{\mathrm{\mathbf{R}Hom}}
\def\RR{\mathrm{R}}
\def\rrad{\mathsf{rad}}

\def\s{\mathcal{S}}

\def\sp{\operatorname{\mathtt{Sp}}}
\def\spc{\operatorname{\mathtt{Spc}}}
\def\Spcl{\mathbf{Spcl}}
\def\spcl{\mathsf{spcl}}
\def\spec{\operatorname{Spec}}
\def\ssupp{\operatorname{\mathtt{Spp}}}
\def\supp{\operatorname{Supp}}
\def\sus{\mathsf{\Sigma}}
\def\T{\mathcal{T}}
\def\t{\mathrm{T}}
\def\Tame{\mathbf{Tame}}
\def\tame{\mathsf{tame}}
\def\thick{\operatorname{\mathsf{thick}}}
\def\Thom{\mathbf{Thom}}

\def\Tor{\mathrm{Tor}}

\def\ts{\operatorname{\mathtt{{}^tSpc}}}
\def\tthick{\operatorname{\mathsf{thick}^\otimes}}
\def\ttthick{\mathsf{thick}^\otimes}
\def\U{\mathcal{U}}
\def\u{\mathtt{U}}
\def\V{\operatorname{V}}
\def\X{\mathcal{X}}
\def\xx{\boldsymbol{x}}
\def\Y{\mathcal{Y}}
\def\yy{\boldsymbol{y}}
\def\Z{\mathbb{Z}}
\def\zero{\mathbf{0}}
\def\zs{\xrightarrow{{}_0}}
\def\ZZ{\mathcal{Z}}

\begin{document}
\allowdisplaybreaks
\title{Thick tensor ideals of right bounded derived categories}
\author{Hiroki Matsui}
\address{Graduate School of Mathematics, Nagoya University, Furocho, Chikusaku, Nagoya, Aichi 464-8602, Japan}
\email{m14037f@math.nagoya-u.ac.jp}
\author{Ryo Takahashi}
\address{Graduate School of Mathematics, Nagoya University, Furocho, Chikusaku, Nagoya, Aichi 464-8602, Japan}
\email{takahashi@math.nagoya-u.ac.jp}
\urladdr{http://www.math.nagoya-u.ac.jp/~takahashi/}
\thanks{Hiroki Matsui was partly supported by JSPS Grant-in-Aid for JSPS Fellows 16J01067.
Ryo Takahashi was partly supported by JSPS Grants-in-Aid for Scientific Research 24340004, 25400038, 16K05098 and 16H03923}
\begin{abstract}
Let $R$ be a commutative noetherian ring.
Denote by $\dm(R)$ the derived category of cochain complexes $X$ of finitely generated $R$-modules with $\h^i(X)=0$ for $i\gg0$.
Then $\dm(R)$ has the structure of a tensor triangulated category with tensor product $-\ltensor_R-$ and unit object $R$.
In this paper, we study thick tensor ideals of $\dm(R)$, i.e., thick subcategories closed under the tensor action by each object in $\dm(R)$, and investigate the Balmer spectrum $\spc\dm(R)$ of $\dm(R)$, i.e., the set of prime thick tensor ideals of $\dm(R)$.
First, we give a complete classification of the thick tensor ideals of $\dm(R)$ generated by bounded complexes, establishing a generalized version of the Hopkins--Neeman smash nilpotence theorem.
Then, we define a pair of maps between the Balmer spectrum $\spc\dm(R)$ and the Zariski spectrum $\spec R$, and study their topological properties.
After that, we compare several classes of thick tensor ideals of $\dm(R)$, relating them to specialization-closed subsets of $\spec R$ and Thomason subsets of $\spc\dm(R)$, and construct a counterexample to a conjecture of Balmer.
Finally, we explore thick tensor ideals of $\dm(R)$ in the case where $R$ is a discrete valuation ring.
\end{abstract}
\maketitle
\tableofcontents
\section*{Introduction}

{\em Tensor triangular geometry} is a theory established by Balmer at the beginning of this century.
Let $(\T,\otimes,\one)$ be an (essentially small) {\em tensor triangulated category}, that is, a triangulated category $\T$ equipped with symmetric tensor product $\otimes$ and unit object $\one$.
One can then define {\em prime} thick tensor ideals of $\T$, which behave similarly to prime ideals of commutative rings.
The {\em Balmer spectrum} $\spc\T$ of $\T$ is defined as the set of prime thick tensor ideals of $\T$.
This set has the structure of a topological space.
Tensor triangular geometry studies Balmer spectra and develops commutative-algebraic and algebro-geometric observations on them.
Tensor triangular geometry is related to a lot of areas of mathematics, including commutative/noncommutative algebra, commutative/noncommutative algebraic geometry, stable homotopy theory, modular representation theory, motivic theory, noncommutative topology and symplectic geometry.
Understandably, tensor triangular geometry has been attracting a great deal of attention, and Balmer gave an invited lecture \cite{B4} at the International Congress of Mathematicians (ICM) in 2010.

By virtue of a landmark theorem due to Balmer \cite{B}, the radical thick tensor ideals of $\T$ correspond to the Thomason subsets of the Balmer spectrum $\spc\T$ of $\T$.
It is thus a main subject in tensor triangular geometry to determine/describe the Balmer spectrum of a given tensor triangulated category.
Such studies have been done for these thirty years considerably widely; one can find ones at least in stable homotopy theory \cite{BS,DHS,HS}, commutative algebra \cite{H,N,stcm}, algebraic geometry \cite{B5,St,T}, modular representation theory \cite{B2,BCR,BIK,FP} and motivic theory \cite{DT,P}.

Let $R$ be a commutative noetherian ring.
Denote by $\dm(R)$ the right bounded derived category of finitely generated $R$-modules, namely, the derived category of (cochain) complexes $X$ of finitely generated $R$-modules such that $\h^i(X)=0$ for all $i\gg0$.
Then $(\dm(R),\ltensor_R,R)$ is a tensor triangulated category.
The main purpose of this paper is to investigate thick tensor ideals of the tensor triangulated category $\dm(R)$, analyzing the structure of the Balmer spectrum $\spc\dm(R)$ of $\dm(R)$.

Here, we should remark that results in the literature which we can apply for our purpose are quite limited.
For example, many people have been studying the Balmer spectra of tensor triangulated categories which arise as the compact objects of compactly generated tensor triangulated categories, but our tensor triangulated category $\dm(R)$ does not arise in this way.
Also, there are various results on the Balmer spectrum of a rigid tensor triangulated category, but again they do not apply to our case because $\dm(R)$ is not rigid (nor even closed); see Remark \ref{rigid}.
Furthermore, several properties have been found for tensor triangulated categories which are generated by their unit object as a thick subcategory, but $\dm(R)$ does not satisfy this property.
Thus, the only existing results that are available and useful for our goal are basically general fundamental results given in \cite{B}, and we need to start with establishing basic tools by ourselves.

From now on, let us explain the main results of this paper.
First of all, recall that an object $X$ of a triangulated category $\T$ is {\em compact} (resp. {\em cocompact}) if the natural morphism
\begin{align*}
&\textstyle
\bigoplus_{\lambda\in\Lambda}\Hom_\T(M,N_\lambda)\to\Hom_\T(M,\bigoplus_{\lambda\in\Lambda}N_\lambda)\\
\big(\text{resp. }
&\textstyle
\bigoplus_{\lambda\in\Lambda}\Hom_\T(N_\lambda, M)\to\Hom_\T(\prod_{\lambda\in\Lambda}N_\lambda, M)\big)
\end{align*}
is an isomorphism for every family $\{N_\lambda\}_{\lambda\in\Lambda}$ of objects of $\T$ with $\bigoplus_{\lambda\in\Lambda}N_\lambda\in\T$ (resp. $\prod_{\lambda\in\Lambda}N_\lambda\in\T$).
A thick tensor ideal of $\dm(R)$ is called {\em compactly generated} (resp. {\em cocompactly generated}) if it is generated by compact (resp. cocompact) objects of $\dm(R)$ as a thick tensor ideal.
For a subcategory $\X$ of $\dm(R)$ we denote by $\supp\X$ the union of the supports of complexes in $\X$, and for a subset $S$ of $\spec R$ we denote by $\langle S\rangle$ the thick tensor ideal of $\dm(R)$ generated by $R/\p$ with $\p\in S$.
We shall prove the following theorem.

\begin{thmm}[Proposition \ref{cpt}, Theorem \ref{main} and Corollary \ref{seven}]\label{A}
The compact (resp. cocompact) objects of $\dm(R)$ are the perfect (resp. bounded) complexes, hence all compactly generated thick tensor ideals are cocompactly generated.
The assignments $\X\mapsto\supp\X$ and $\langle W\rangle\mapsfrom W$ make mutually inverse bijections
$$
\xymatrix{
\{\text{Cocompactly generated thick $\otimes$-ideals of $\dm(R)$}\}\ \ar@<.5mm>[r]&\ \{\text{Specialization-closed subsets of $\spec R$}\}.\ar@<.5mm>[l]
}
$$
Consequently, all cocompactly generated thick tensor ideals of $\dm(R)$ are compactly generated.
\end{thmm}

\noindent
The core of this theorem is constituted by the classification of the {\em cocompactly} generated thick tensor ideals of $\dm(R)$, which is obtained by establishment of a {\em generalized smash nilpotence theorem}, extending the classical smash nilpotence theorem due to Hopkins \cite{H} and Neeman \cite{N} for the homotopy category of perfect complexes.
In view of Theorem \ref{A}, we may simply call $\X$ {\em compact} if $\X$ satisfies one (hence all) of the equivalent conditions in the theorem.
We should remark that in general we have
$$
\langle W\rangle\ne\supp^{-1}W,
$$
where $\supp^{-1}W$ consists of the complexes whose supports are contained in $W$.
Thus we call a thick tensor ideal of $\dm(R)$ {\em tame} if it has the form $\supp^{-1}W$ for some specialization-closed subset $W$ of $\spec R$.

Next, we relate the Balmer spectrum $\spc\dm(R)$ of $\dm(R)$ to the Zariski spectrum $\spec R$ of $R$, i.e., the set of prime ideals of $R$.
More precisely, we introduce a pair of order-reversing maps
$$
\xymatrix{
\s:\ \spec R\ \ar@<.5mm>[rr]&&\ \spc\dm(R)\ :\pp\ar@<.5mm>[ll]
}
$$
and investigate their topological properties.
These maps are defined as follows: let $\p\in\spec R$ and $\P\in\spc\dm(R)$.
Then $\s(\p)$ consists of the complexes $X\in\dm(R)$ with $X_\p=0$, and $\pp(\P)$ is the unique maximal element of ideals $I$ of $R$ with $R/I\notin\P$ with respect to the inclusion relation.
Our main result in this direction is the following theorem.
Denote by $\ts\dm(R)$ the set of tame prime thick tensor ideals of $\dm(R)$, and by $\mx\dm(R)$ (resp. $\mn\dm(R)$) the maximal (resp. minimal) elements of $\spc\dm(R)$ with respect to the inclusion relation.
For each full subcategory $\X$ of $\dm(R)$, let $\X^\tame$ stand for the smallest tame thick tensor ideal of $\dm(R)$ containing $\X$.

\begin{thmm}[Theorems \ref{s}, \ref{kakan}, \ref{homeos}, \ref{minmx}, \ref{sm} and Corollary \ref{tp}]\label{B}
The following statements hold.
\begin{enumerate}[\rm(1)]
\item
One has $\pp\cdot\s=1$ and $\s\cdot\pp=\supp^{-1}\supp=()^\tame$.
In particular, $\dim(\spc\dm(R))\ge\dim R$.
\item
The image of $\s$ coincides with $\ts\dm(R)$, and it is dense in $\spc\dm(R)$.
\item
The map $\pp$ is continuous, and its restriction $\pp':\ts\dm(R)\to\spec R$ is a continuous bijection.
\item
The map $\s':\spec R\to\ts\dm(R)$ induced by $\s$ is an open and closed bijection.
\item
The map $\Min R\to\mx\dm(R)$ induced by $\s$ is a homeomorphism.
\item
The map $\Max R\to\mn\dm(R)$ induced by $\s$ is a homeomorphism  if $R$ is semilocal.
\item
One has:\quad
$\text{$\s$ is continuous}\ \Leftrightarrow\ 
\text{$\s'$ is homeomorphic}\ \Leftrightarrow\ 
\text{$\pp'$ is homeomorphic}\ \Leftrightarrow\ 
\text{$\spec R$ is finite}.$
\end{enumerate}
\end{thmm}

The celebrated classification theorem due to Balmer \cite{B} asserts that taking the Balmer support $\ssupp$ makes a one-to-one correspondence between the set $\Rad$ of radical thick tensor ideals of $\dm(R)$ and the set $\Thom$ of Thomason subsets of $\spc\dm(R)$:
$$
\xymatrix{
\ssupp:\ \Rad\ \ar@<.5mm>[rr]&&\ \Thom\ :\ssupp^{-1}\ar@<.5mm>[ll].
}
$$
Our next goal is to complete this one-to-one correspondence to the following commutative diagram, giving complete classifications of compact and tame thick tensor ideals of $\dm(R)$.
Denote by $\Cpt$ (resp. $\Tame$) the set of compact (resp. tame) thick tensor ideals of $\dm(R)$, and by $\Spcl(\spec)$ (resp. $\Spcl(\ts)$) the set of specialization-closed subsets of $\spec R$ (resp. $\ts\dm(R)$).

\begin{thmm}[Theorems \ref{d1}, \ref{d2}]\label{C}
There is a diagram
$$
\xymatrix{
\Rad\ar@<.5mm>[rrrr]^\ssupp\ar@<.5mm>[dd]^{{()}_\cpt} &&&& \Thom\ar@<.5mm>[llll]^{\ssupp^{-1}}\ar@<.5mm>[dd]^{\s^{-1}}\ar@<.5mm>[rrrrdd]^{{()}_\spcl} &&&& \\
\\
\Cpt\ar@<.5mm>[uu]^{()^\rrad}\ar@<.5mm>[rrrr]^\supp\ar@<.5mm>[rrrrdd]^{{()}^\tame} &&&& \Spcl(\spec)\ar@<.5mm>[uu]^{\overline\s}\ar@<.5mm>[llll]^{\langle\rangle}\ar@<.5mm>[rrrr]^\s\ar@<.5mm>[dd]^{\supp^{-1}} &&&& \Spcl(\ts)\ar@<.5mm>[lllluu]^{()^\spcl}\ar@<.5mm>[llll]^\pp\ar@<.5mm>[lllldd]^{\sp^{-1}} \\
\\
&&&& \Tame\ar@<.5mm>[lllluu]^{{()}_\cpt}\ar@<.5mm>[rrrruu]^\sp\ar@<.5mm>[uu]^\supp &&&&
}
$$
where the pairs of maps $A=(()^\rrad,()_\cpt)$, $B=(\overline\s,\s^{-1})$, $C=(()^\spcl,()_\spcl)$ are section-retraction pairs (as sets), and all the other pairs consist of mutually inverse bijections.
The diagram with the sections (resp. retractions) and bijections is commutative.
\end{thmm}

\noindent
We do not give here the definitions of the maps appearing above (we do this in Section \ref{various}); what we want to emphasize now is that those maps are given explicitly.

Moreover, we prove that some/any of the three section-retraction pairs $A,B,C$ in the above theorem are bijections if and only if $R$ is artinian, which is incorporated into the following theorem.

\begin{thmm}[Theorem \ref{11}]\label{D}
The following are equivalent.
\begin{enumerate}[\rm(1)]
\item
$R$ is artinian.
\item
Every thick tensor ideal of $\dm(R)$ is compact, tame and radical.
\item
Every radical thick tensor ideal of $\dm(R)$ is tame.
\item
The pair of maps $(\s,\pp)$ consists of mutually inverse homeomorphisms.
\item
Some/all of the maps $\s,\pp$ are bijective.
\item
Some/all of the pairs $A,B,C$ consist of mutually inverse bijections.
\end{enumerate}
\end{thmm}

\noindent
This theorem says that in the case of artinian rings everything is clear.
An essential role is played in the proof of this theorem by a certain complex in $\dm(R)$ constructed from shifted Koszul complexes.

Let $(\T,\otimes,\one)$ be a tensor triangulated category.
Balmer \cite{B3} constructs a continuous map
$$
\xymatrix{
\rho_\T^\bullet:\ \spc\T\ \ar[rr]&&\ \hspec\RR_\T^\bullet,
}
$$
where $\RR_\T^\bullet=\Hom_\T(\one,\sus^\bullet\one)$ is a graded-commutative ring.
Balmer \cite{B4} conjectures that the map $\rho_\T^\bullet$ is (locally) injective when $\T$ is an {\em algebraic} triangulated category, that is, a triangulated category arisen as the stable category of a Frobenius exact category.
Our $\dm(R)$ is evidently an algebraic triangulated category, but does not satisfy this conjecture under a quite mild assumption:

\begin{thmm}[Corollary \ref{cntex}]\label{E}
Assume that $\dim R>0$ and that $R$ is either a domain or a local ring.
Then the map $\rho_{\dm(R)}^\bullet$ is not locally injective.
Hence, Balmer's conjecture does not hold for $\dm(R)$.
\end{thmm}

\noindent
In fact, the assumption of the theorem gives an element $x\in R$ with $\height(x)>0$.
Then we can find a non-tame prime thick tensor ideal $\P$ of $\dm(R)$ associated with $x$ at which $\rho_{\dm(R)}^\bullet$ is not locally injective.

Finally, we explore thick tensor ideals of $\dm(R)$ in the case where $R$ is a discrete valuation ring, because this should be the simplest unclear case, now that everything is clarified by Theorem \ref{D} in the case of artinian rings.
We show the following theorem, which says that even if $R$ is such a good ring, the structure of the Balmer spectrum of $\dm(R)$ is rather complicated.
(Here, $\ell\ell(-)$ stands for the Loewy length.)

\begin{thmm}[Propositions \ref{L_0}, \ref{a} and Theorems \ref{nonnoeth}, \ref{poly}]\label{F}
Let $(R,xR)$ be a discrete valuation ring, and let $n\ge0$ be an integer.
Let $\P_n$ be the full subcategory of $\dm(R)$ consisting of complexes $X$ with finite length homologies such that there exists an integer $t\ge0$ with $\ell\ell(\h^{-i}X)\le ti^n$ for all $i\gg0$.
Then:
\begin{enumerate}[\rm(1)]
\item
$\P_n$ coincides with the smallest thick tensor ideal of $\dm(R)$ containing the complex
$$
\textstyle\bigoplus_{i>0}(R/x^{i^n}R)[i]=(\cdots\zs R/x^{3^n}R\zs R/x^{2^n}R\zs R/x^{1^n}R\to0).
$$
\item
$\P_n$ is a prime thick tensor ideal of $\dm(R)$ which is not tame.
If $n\ge1$, then $\P_n$ is not compact.
\item
One has $\P_0\subsetneq\P_1\subsetneq\P_2\subsetneq\cdots$.
Hence $\spc\dm(R)$ has infinite Krull dimension.
\end{enumerate}
\end{thmm}

The paper is organized as follows.
Section \ref{fund} is devoted to giving several basic definitions and studying fundamental properties that are used in later sections.
In Section \ref{sect:cpt}, we study compactly and cocompactly generated thick tensor ideals of $\dm(R)$, and classify them completely.
The generalized smash nilpotence theorem and Theorem \ref{A} are proved in this section.
In Section \ref{sect:bal}, we define the maps $\s$ and $\pp$ between $\spec R$ and $\spc\dm(R)$, and prove part of Theorem \ref{B}.
In Section \ref{tsbs}, we study topological properties of the maps $\s,\pp$ and the Balmer spectrum $\spc\dm(R)$.
We complete in this section the proof of Theorem \ref{B}.
In Section \ref{various}, we compare compact, tame and radical thick tensor ideals of $\dm(R)$, relating them to specialization-closed subsets of $\spec R$ and $\ts\dm(R)$ and Thomason subsets of $\spc\dm(R)$.
Theorem \ref{C} is proved in this section.
In Section \ref{balconj}, we consider when the section-retraction pairs in Theorem \ref{C} are one-to-one correspondences, and deal with the conjecture of Balmer for $\dm(R)$.
We show Theorems \ref{D}, \ref{E} in this section.
The final Section \ref{sect:dvr} concentrates on investigation of the case of discrete valuation rings.
Several properties that are specific to this case are found out, and Theorem \ref{F} is proved in this section.
\section*{Acknowledgments}
Some of this work was introduced by Ryo Takahashi in his series of three invited plenary lectures at the 17th Workshop and International Conference on Representations of Algebras (ICRA) held at Syracuse University in August, 2016.
It was really a great honor for him, and he is very grateful to the Organizing and Advisory Committees for giving him such a wonderful opportunity.
Also, the authors very much thank the anonymous referees for their careful reading, valuable comments and helpful suggestions.
\section{Fundamental materials}\label{fund}

In this section, we give several basic definitions and study fundamental properties, which will be used in later sections.
We begin with our convention.

\begin{conv}
Throughout the paper, unless otherwise specified, $R$ is a commutative noetherian ring, and all subcategories are nonempty and full.
We put $I^0=R$ and $x^0=1$ for an ideal $I$ of $R$ and an element $x\in R$.
We denote by $\spec R$ (resp. $\Max R$, $\Min R$) the set of prime (resp. maximal, minimal prime) ideals of $R$.
For an ideal $I$ of $R$, we denote by $\V(I)$ the set of prime ideals of $R$ containing $I$, and set $\D(I)=\V(I)^\complement=\spec R\setminus\V(I)$.
When $I$ is generated by a single element $x$, we simply write $\V(x)$ and $\D(x)$.
For a prime ideal $\p$ of $R$, the residue field of $R_\p$ is denoted by $\kappa(\p)$, i.e., $\kappa(\p)=R_\p/\p R_\p$.
For a sequence $\xx=x_1,\dots,x_n$ of elements of $R$, the {\em Koszul complex} of $R$ with respect to $\xx$ is denoted by $\k(\xx,R)$.
For an additive category $\C$ we denote by $\zero$ the {\em zero subcategory} of $\C$, that is, the full subcategory consisting of objects isomorphic to the zero object.
For objects $X,Y$ of $\C$, we mean by $X\lessdot Y$ (or $Y\gtrdot X$) that $X$ is a direct summand of $Y$ in $\C$.
We often omit subscripts, superscripts and parentheses, if there is no danger of confusion.
\end{conv}

Let $\T$ be a triangulated category.
A {\em thick} subcategory of $\T$ is by definition a triangulated subcategory closed under direct summands; in other words, it is a subcategory closed under direct summands, shifts and cones.
For a subcategory $\X$ of $\T$ we denote by $\thick\X$ the {\em thick closure} of $\X$, that is, the smallest thick subcategory of $\T$ containing $\X$.

Now we recall the definitions of a tensor triangulated category and a thick tensor ideal.

\begin{dfn}
\begin{enumerate}[(1)]
\item
We say that $(\T,\otimes,\one)$ is a {\em tensor triangulated category} if $\T$ is a triangulated category equipped with a symmetric monoidal structure which is compatible with the triangulated structure of $\T$; see \cite[Appendix A]{HPS} for the precise definition.
In particular, $-\otimes-$ is exact in each variable.

\item
Let $(\T,\otimes,\one)$ be a tensor triangulated category.
A subcategory $\X$ of $\T$ is said to be a {\em thick tensor ideal} provided that $\X$ is a thick subcategory of $\T$ and for any $T\in\T$ and $X\in\X$ one has $T\otimes X\in\X$.
We often abbreviate ``tensor ideal" to ``$\otimes$-ideal".
For a subcategory $\C$ of $\T$, we define the {\em thick $\otimes$-ideal closure} of $\C$ to be the smallest thick $\otimes$-ideal of $\T$ containing $\C$, and denote it by $\tthick\C$.
\end{enumerate}
\end{dfn}

We denote by $\dm(R)$ (resp. $\db(R)$) the derived category of (cochain) complexes $X$ of finitely generated $R$-modules with $\h^i(X)=0$ for all $i\gg0$ (resp. $|i|\gg0$).
We denote by $\df(R)$ (resp. $\dbf(R)$) the subcategory of $\dm(R)$ (resp. $\db(R)$) consisting of complexes $X$ whose homologies have finite length as $R$-modules.
By $\km(R)$ (resp. $\kb(\proj R)$) we denote the homotopy category of complexes $P$ of finitely generated projective $R$-modules with $P^i=0$ for all $i\gg0$ (resp. $|i|\gg0$).
By $\kmb(R)$ the subcategory of $\km(R)$ consisting of complexes $P$ with $\h^i(P)=0$ for all $i\ll0$.
Note that there are chains
$$
\dbf(R)\subseteq\db(R)\subseteq\dm(R),\quad
\dbf(R)\subseteq\df(R)\subseteq\dm(R),\quad
\kb(\proj R)\subseteq\kmb(R)\subseteq\km(R)
$$
of thick subcategories and triangle equivalences
$$
\dm(R)\cong\km(R),\quad
\db(R)\cong\kmb(R).
$$
We will often identify $\dm(R),\db(R)$ with $\km(R),\kmb(R)$ respectively, via these equivalences.
Note that $(\kb(\proj R),\otimes_R,R)$ and $(\dm(R),\ltensor_R,R)$ are essentially small tensor triangulated categories.
(In general, if $\C$ is an essentially small additive category, then so is the category of complexes of objects in $\C$, and so is the homotopy category.)

\begin{rem}\label{rigid}
The tensor triangulated category $\dm(R)$ is never rigid.
More strongly, it is never closed.
In fact, assume that there is a functor $F:\dm(R)\times\dm(R)\to\dm(R)$ such that $\Hom_{\dm(R)}(X\ltensor_RY,Z)\cong\Hom_{\dm(R)}(Y,F(X,Z))$ for all $X,Y,Z\in\dm(R)$.
We have $\Hom_{\dm(R)}(X\ltensor_RY,Z)=\Hom_{\d(R)}(X\ltensor_RY,Z)\cong\Hom_{\d(R)}(Y,\RHom_R(X,Z))$, where $\d(R)$ is the unbounded derived category of $R$-modules.
Letting $Y=R[-i]$ for $i\in\Z$, we obtain $\h^i(F(X,Z))\cong\Ext_R^i(X,Z)$.
Since $F(X,Z)$ is in $\dm(R)$, we have $\h^i(F(X,Z))=0$ for $i\gg0$.
Hence $\Ext_R^{\gg0}(X,Z)=0$ for all $X,Z\in\dm(R)$.
This is a contradiction.
\end{rem}

Here we compute some thick closures and thick $\otimes$-ideal closures.

\begin{prop}\label{si2}
There are equalities:
\begin{enumerate}[\rm(1)]
\item
$\ttthick_{\dm(R)}\,R=\dm(R)$.
\item
$\thick_{\dm(R)}R=\thick_{\db(R)}R=\thick_{\kb(\proj R)}R=\ttthick_{\kb(\proj R)}\,R=\kb(\proj R)$.
\item
$\thick_{\dm(R)}k=\thick_{\db(R)}k=\dbf(R)$, if $R$ is local with residue field $k$.
\end{enumerate}
\end{prop}

\begin{proof}
The following hold in general, which are easy to check.
\begin{enumerate}[(a)]
\item
Let $\T$ be a triangulated category, $\U$ a thick subcategory and $U\in\U$.
Then $\thick_\U U=\thick_\T U$.
\item
Let $(\T,\otimes,\one)$ be a tensor triangulated category.
Then $\tthick\one=\T$.
\end{enumerate}
The assertion is shown by these two statements.
\end{proof}

From now on, we deal with the supports of objects and subcategories of $\dm(R)$.
Recall that the {\em support} of an $R$-module $M$ is defined as the set of prime ideals $\p$ of $R$ such that the $R_\p$-module $M_\p$ is nonzero, which is denoted by $\supp_RM$.

\begin{prop}\label{supp}
Let $X$ be a complex in $\dm(R)$.
Then the following three sets are equal.
\begin{enumerate}[\rm(1)]
\item
$\bigcup_{i\in\Z}\supp_R\h^i(X)$,
\item
$\{\p\in\spec R\mid X_\p\not\cong0\text{ in }\dm(R_\p)\}$,
\item
$\{\p\in\spec R\mid \kappa(\p)\ltensor_RX\not\cong0\text{ in }\dm(R_\p)\}$.
\end{enumerate}
\end{prop}

\begin{proof}
It is clear that the first and second sets coincide.
For a prime ideal $\p$ of $R$ one has $\kappa(\p)\ltensor_RX\cong\kappa(\p)\ltensor_{R_\p}X_\p$.
It is seen by \cite[Corollary (A.4.16)]{C} that the second and third sets coincide.
\end{proof}

\begin{dfn}
The set in Proposition \ref{supp} is called the {\em support} of $X$ and denoted by $\supp_RX$.
For a subcategory $\C$ of $\dm(R)$, we set $\supp\C=\bigcup_{C\in\C}\supp C$, and call this the {\em support} of $\C$.
For a subset $S$ of $\spec R$, we denote by $\supp^{-1}S$ the subcategory of $\dm(R)$ consisting of complexes whose supports are contained in $S$.
\end{dfn}

\begin{rem}\label{musi}
The fact that the second and third sets in Proposition \ref{supp} coincide will often play an important role in this paper.
Note that these two sets are different if $X$ is a complex outside $\dm(R)$.
For example, let $(R,\m,k)$ be a local ring of positive Krull dimension.
Take any nonmaximal prime ideal $P$, and let $X$ be the injective hull $\EE(R/P)$ of the $R$-module $R/P$.
Then $k\ltensor_RX=0$, while $X_\m\ne0$.
\end{rem}

\begin{rem}\label{r}
For $X\in\dm(R)$ one has $\supp X=\emptyset$ if and only if $X=0$.
In other words, it holds that $\supp^{-1}\emptyset=\zero$.
(If we define the support of $X$ as the third set in Proposition \ref{supp}, then the assumption that $X$ belongs to $\dm(R)$ is essential, as the example given in Remark \ref{musi} shows.)
\end{rem}

In the following lemma and proposition, we state several basic properties of $\supp$ and $\supp^{-1}$ defined above.
Both results will often be used later.

\begin{lem}\label{si0}
The following statements hold.
\begin{enumerate}[\rm(1)]
\item
$\supp(X[n])=\supp X$ for all $X\in\dm(R)$ and $n\in\Z$.
\item
If $X$ is a direct summand of $Y$ in $\dm(R)$, then $\supp X\subseteq\supp Y$.
\item
If $X\to Y\to Z\to X[1]$ is an exact triangle in $\dm(R)$, then $\supp A\subseteq\supp B\cup\supp C$ for all $\{A,B,C\}=\{X,Y,Z\}$.
\item
$\supp(X\ltensor_RY)=\supp X\cap\supp Y$ for all $X,Y\in\dm(R)$.
\end{enumerate}
\end{lem}

\begin{proof}
The assertions (1), (2) and (3) are straightforward by definition.
For each prime ideal $\p$ of $R$ there is an isomorphism $(X\ltensor_RY)_\p\cong X_\p\ltensor_{R_\p}Y_\p$.
Hence $(X\ltensor_RY)_\p=0$ if and only if either $X_\p=0$ or $Y_\p=0$ by \cite[Corollary (A.4.16)]{C}.
This shows the assertion (4).
\end{proof}

Let $X$ be a topological space.
A subset $A$ of $X$ is called {\em specialization-closed} provided that for each point $a\in A$ the closure $\overline{\{a\}}$ of $\{a\}$ in $X$ is contained in $A$.
Hence a subset $S$ of $\spec R$ is specialization-closed if and only if for each $\p\in S$ one has $\V(\p)\subseteq S$.
Note that $A$ is specialization-closed if and only if $A$ is a (possibly infinite) union of closed subsets of $X$.
Therefore a union of specialization-closed subsets is again specialization-closed, and thus one can define the largest specialization-closed subset $A_\spcl$ of $X$ contained in $A$, which will be called the {\em $\spcl$-closure} of $A$ in Section \ref{various}.

\begin{prop}\label{si}
\begin{enumerate}[\rm(1)]
\item
Let $S$ be a subset of $\spec R$.
Then there are equalities $\supp^{-1}S=\supp^{-1}(S_\spcl)$ and $\supp(\supp^{-1}S)=S_\spcl$.
Moreover, $\supp^{-1}S$ is a thick $\otimes$-ideal of $\dm(R)$.
\item
Let $\X$ be any subcategory of $\dm(R)$.
Then $\supp\X$ is a specialization-closed subset of $\spec R$, and one has $\supp\X=\supp(\tthick\X)$.
\item
It holds that $\df(R)=\supp^{-1}(\Max R)$.
In particular, $\df(R)$ is a thick $\otimes$-ideal of $\dm(R)$.
\end{enumerate}
\end{prop}

\begin{proof}
(1) We put $W=S_\spcl$.
Let $X$ be a complex in $\dm(R)$.
Since $\supp X$ is specialization-closed, it is contained in $S$ if and only if it is contained in $W$.
Hence $\supp^{-1}S=\supp^{-1}W$.
Evidently, $W$ contains $\supp(\supp^{-1}W)$, while we have $\p\in\supp R/\p=\V(\p)\subseteq W$ for $\p\in W$.
Hence $\supp(\supp^{-1}W)=W$, and thus $\supp(\supp^{-1}S)=W$.
It is seen from Lemma \ref{si0} that $\supp^{-1}S$ is a thick $\otimes$-ideal of $\dm(R)$.

(2) We have $\supp\X=\bigcup_{X\in\X}\supp X=\bigcup_{X\in\X}\bigcup_{i\in\Z}\supp\h^iX$ by Proposition \ref{supp}.
Since $\h^iX$ is a finitely generated $R$-module, $\supp\h^iX$ is closed.
Hence $\supp\X$ is specialization-closed.
A prime ideal $\p$ of $R$ is not in $\supp\X$ if and only if $\X$ is contained in $\supp^{-1}(\{\p\}^\complement)$, if and only if $\tthick\X$ is contained in $\supp^{-1}(\{\p\}^\complement)$, if and only if $\p$ does not belong to $\supp(\tthick\X)$.
It follows from (1) that $\supp^{-1}(\{\p\}^\complement)$ is a thick $\otimes$-ideal of $\dm(R)$, which shows the second equivalence.
The other two equivalences are obvious.

(3) The equality is straightforward, and the last assertion is shown by (1).
\end{proof}

\section{Classification of compact thick tensor ideals}\label{sect:cpt}

In this section, we prove a generalized version of the smash nilpotence theorem due to Hopkins \cite{H} and Neeman \cite{N}, and using this we give a complete classification of cocompact thick tensor ideals of $\dm(R)$.

We begin with recalling the definitions of compact and cocompact objects.
Let $\T$ be a triangulated category.
We say that an object $M\in\T$ is {\em compact} (resp. {\em cocompact}) if the natural morphism
\begin{align*}
&\textstyle\bigoplus_{\lambda\in\Lambda}\Hom_\T(M,N_\lambda)\to\Hom_\T(M,\bigoplus_{\lambda\in\Lambda}N_\lambda)\\
\ (\text{resp. }&\textstyle\bigoplus_{\lambda\in\Lambda}\Hom_\T(N_\lambda, M)\to\Hom_\T(\prod_{\lambda\in\Lambda}N_\lambda, M))
\end{align*}
is an isomorphism for every family $\{N_\lambda\}_{\lambda\in\Lambda}$ of objects of $\T$ with $\bigoplus_{\lambda\in\Lambda}N_\lambda\in\T$ (resp. $\prod_{\lambda\in\Lambda}N_\lambda\in\T$).
We denote by $\T^\c$ (resp. $\T^\cc$) the subcategory of $\T$ consisting of compact (resp. cocompact) objects.
For $\T=\dm(R)$ we have explicit descriptions of the compact objects and cocompact objects:

\begin{prop}\label{cpt}
One has $\dm(R)^\c=\kb(\proj R)$ and $\dm(R)^\cc=\db(R)$.
\end{prop}

\begin{proof}
The second statement follows from \cite[Theorem 18]{OS}.
The first one can be shown in the same way as the proof of the fact that the compact objects of the unbounded derived category of all $R$-modules coincides with $\kb(\proj R)$.
For the convenience of the reader, we give a proof.

First of all, $R$ is compact since each homology functor $\h^i$ commutes with direct sums.
Since the compact objects form a thick subcategory, one has $\kb(\proj R) \subseteq \dm(R)^\c$.
Next, let $X\in\dm(R)$ be a compact object.
Replacing $X$ with its projective resolution, we may assume $X \in \km(R)$.
Consider the chain map
$$
\xymatrix{
X\ar[d]^{f_n} & = & (\cdots\ar[r] & X^{n-1}\ar[r]^{d^{n-1}}\ar[d] & X^n\ar[r]^{d^{n}}\ar[d]^{f_n^n} & X^{n+1}\ar[r]\ar[d] & \cdots)\phantom{,} \\
C^n[-n] & = & (\cdots\ar[r] & 0\ar[r] & C^n\ar[r] & 0\ar[r] & \cdots),
}
$$
where $C^n$ is the cokernel of $d^{n-1}$, and $f_n^n:X^n \to C^n$ is a natural surjection.
Put $Y=\bigoplus_{n \in \Z}C^n[-n]$.
A chain map $f:X \to Y$ is induced by $\{f_n \}_{n \in \Z}$.
As $X\in\km(R)$ is compact in $\dm(R)$, we have isomorphisms
$$
\textstyle\Hom_{\K}(X, Y) \cong \Hom_{\dm(R)}(X, Y) \cong \bigoplus_{n \in \Z}\Hom_{\dm(R)}(X, C^n[-n]) \cong \bigoplus_{n \in \Z}\Hom_{\K}(X, C^n[-n]),
$$
where $\K$ is the homotopy category of $R$-modules.
The composition of these isomorphisms sends $f$ to $(f_n)_{n \in \Z}$, which implies that there exists $t\in\Z$ such that $f_n=0$ in $\K$ for all $n\le t$.
Hence, there is an $R$-linear map $g: X^{n+1} \to C^n$ such that $g \circ d^n = f_n^n$.
Let $\overline{d^n}:C^n \to X^{n+1}$ be the map induced by $d^n$.
We have $g \overline{d^n} f_n^n = g d^n= f_n^n$, and obtain $g \overline{d^n}=1$ as $f_n^n$ is a surjection.
Thus, $C^n$ is a direct summand of $X^{n+1}$, and thereby projective.
Also, $\h^nX$ is isomorphic to the kernel of $\overline{d^n}$, which vanishes since $\overline{d^n}$ is a split monomorphism.
Consequently, the truncated complex $X':=(0 \to C^t \xrightarrow{\overline{d^t}} X^{t+1} \xrightarrow{d^{t+1}} X^{t+2} \xrightarrow{d^{t+2}}\cdots)$, which is quasi-isomorphic to $X$, is in $\kb(\proj R)$.
We now conclude that $X$ belongs to $\kb(\proj R)$.
\end{proof}

Next, we make the definitions of the annihilators of morphisms and objects in $\dm(R)$.

\begin{dfn}
\begin{enumerate}[(1)]
\item
Let $f:X\to Y$ be a morphism in $\dm(R)$.
We define the {\em annihilator} of $f$ as the set of elements $a\in R$ such that $a f=0$ in $\dm(R)$, and denote it by $\ann_R(f)$.
This is an ideal of $R$.
\item
The {\em annihilator} of an object $X\in\dm(R)$ is defined as the annihilator of the identity morphism $\id_X$, and denoted by $\ann_R(X)$.
This is the set of elements $a\in R$ such that $(X \xrightarrow{a} X)= 0$ in $\dm(R)$.
\end{enumerate}
\end{dfn}

Here are some properties of annihilators.

\begin{prop}\label{annihi}
\begin{enumerate}[\rm(1)]
\item
Let $f: X \to Y$ be a morphism in $\dm(R)$ and $\p$ a prime ideal of $R$.
\begin{enumerate}[\rm(a)]
\item
The ideal $\ann_R(f)$ is the kernel of the map $\eta_f:R \to \Hom_{\dm(R)}(X, Y)$ given by $a \mapsto a f$.
\item
If the natural map $\tau_{X,Y,\p}:\Hom_{\dm(R)}(X, Y)_\p \to \Hom_{\dm(R_\p)}(X_\p, Y_\p)$ is an isomorphism, then there is an equality $\ann_R(f)_\p = \ann_{R_\p}(f_\p)$.
\end{enumerate}
\item
For any $X \in \dm(R)$ one has $\V(\ann X) \supseteq \supp X$.
The equality holds if $\tau_{X,X,\p}$ is an isomorphism for all $\p\in\spec R$. 
In particular, for $X\in\db(R)$ one has $\V(\ann X) = \supp X$.
\item
Let $\xx=x_1, \ldots, x_n$ be a sequence of elements of $R$.
Then it holds that $\ann\k(\xx, R)=\xx R$.
In particular, there is an equality $\supp\k(\xx,R)=\V(\xx)$, and $\k(\xx,R)$ belongs to $\supp^{-1}\V(\xx)$.
\end{enumerate}
\end{prop}

\begin{proof}
(1)
The assertion (a) is obvious, while (b) follows from (a) and the commutative diagram
$$
\xymatrix{
R_\p\ar[rrr]^-{(\eta_{f})_\p}\ar@{=}[d] &&& \Hom_{\dm(R)}(X, Y)_\p\ar[d]_{\cong}^{\tau_{X,Y,\p}} \\
R_\p\ar[rrr]^-{\eta_{f_\p}} &&& \Hom_{\dm(R_\p)}(X_\p, Y_\p).
}
$$

(2)
The first assertion is easy to show.
Suppose that $\tau_{X,X,\p}$ is an isomorphism for all $\p\in\spec R$.
By (1) one has $(\ann_RX)_\p=\ann_{R_\p}X_\p$.
We have $X_\p \neq 0$ if and only if $(\ann_R X)_\p\ne R_\p$, if and only if $\p \in \V(\ann_R X)$.
This shows $\V(\ann_RX)=\supp_RX$.
As for the last assertion, use \cite[Lemma 5.2(b)]{AF}.

(3)
The second statement follows from the first one and (2).
Therefore it suffices to show the equality $\ann\k(\xx,R)=\xx R$.
It follows from \cite[Proposition 1.6.5]{BH} that $\ann\k(\xx,R)$ contains $\xx R$.
Conversely, pick $a \in \ann\k(\xx, R)$.
Then the multiplication map $a:\k(\xx, R) \to \k(\xx, R)$ is null-homotopic, and there is a homotopy $\{s_i : \k_{i-1}(\xx, R) \to \k_i(\xx, R) \}$ from $a$ to $0$.
In particular, we have $a=d_1 s_{1}$, where $d_1$ is the first differential of $\k(\xx, R)$.
Writing $d_1 = (x_1, \ldots, x_n): R^n \to R$ and $s_1={}^t\!(a_1, \ldots, a_n): R \to R^n$, we get $a =  (x_1, \ldots, x_n)\, {}^t\!(a_1, \ldots, a_n) = a_1 x_1 + \cdots + a_n x_n\in\xx R$.
Consequently, we obtain $\ann\k(\xx,R)=\xx R$.
\end{proof}

To state our next results, we need to introduce some notation.

\begin{dfn}
Let $\T$ be a triangulated category.
\begin{enumerate}[(1)]
\item
For two subcategories $\C_1,\C_2$ of $\T$, we denote by $\C_1 * \C_2$ the subcategory of $\T$ consisting of objects $M$ such that there is an exact triangle $C_1 \to M \to C_2 \rightsquigarrow$ with $C_i\in\C_i$ for $i=1,2$.
\item
For a subcategory $\C$ of $\T$, we denote by $\adds\C$ the smallest subcategory of $\T$ that contains $\C$ and is closed under finite direct sums, direct summands and shifts.
Inductively we define $\thick_\T^1(\C)=\adds\C$ and $\thick_\T^r(\C)=\adds(\thick_\T^{r-1}(\C)*\adds\C)$ for $r>1$.
This is sometimes called the $r$-th {\em thickening} of $\C$.
When $\C$ consists of a single object $X$, we simply denote it by $\thick_\T^r(X)$.
\item
For a morphism $f:X\to Y$ in $\T$ and an integer $n\ge1$, we denote by $f^{\otimes n}$ the $n$-fold tensor product $\underbrace{f\otimes\cdots\otimes f}_n$.
Note that for $\T=\dm(R)$ we mean by $f^{\otimes n}$ the morphism $\underbrace{f\ltensor_R\cdots\ltensor_Rf}_n$.
\end{enumerate}
\end{dfn}

We establish two lemmas, which will be used to show the generalized smash nilpotence theorem.
The first one concerns general tensor triangulated categories, while the second one is specific to our $\dm(R)$.

\begin{lem}\label{cl}
Let $\T$ be a tensor triangulated category.
\begin{enumerate}[\rm(1)]
\item
Let $\X,\Y$ be subcategories of $\T$.
Let $f:M\to M'$ and $g:N\to N'$ be morphisms in $\T$.
If $f\otimes\X=0$ and $g\otimes\Y=0$, then $f\otimes g\otimes(\X*\Y)=0$.
\item
Let $\phi:A \to B$ be a morphism in $\T$, and let $C$ be an object of $\T$.
If $\phi\otimes C =0$, then $\phi^{\otimes n} \otimes\thick_\T^n(C)=0$ for all integers $n>0$.
\end{enumerate}
\end{lem}

\begin{proof}
As (2) is shown by induction on $n$ and (1), so let us show (1).
Let $X\to E\to Y\rightsquigarrow$ be an exact triangle in $\T$ with $X\in\X$ and $Y\in\Y$.
Then $f\otimes X=0$ and $g\otimes Y=0$ by assumption.
There is a diagram
$$
\xymatrix{
M\otimes N\otimes X\ar[rrr]\ar[d]_{M\otimes g\otimes X} &&& M\otimes N\otimes E\ar[rrr]\ar[d]_{M\otimes g\otimes E}\ar@{}[rrrd]|\circlearrowleft &&& M\otimes N\otimes Y\ar@{~>}[r]\ar[d]_{M\otimes g\otimes Y}^0 &\\
M\otimes N'\otimes X\ar[rrr]\ar[d]_{f\otimes N'\otimes X}^0\ar[rrrd]_0 &&& M\otimes N'\otimes E\ar[rrr]\ar[d]_{f\otimes N'\otimes E}\ar@{}[lld]_\circlearrowleft\ar@{}[rd]^\circlearrowleft &&& M\otimes N'\otimes Y\ar@{~>}[r]\ar[d]_{f\otimes N'\otimes Y}\ar@{.>}[llld]^h &\\
M'\otimes N'\otimes X\ar[rrr]\ar@{}[rru]_\circlearrowleft &&& M'\otimes N'\otimes E\ar[rrr] &&& M'\otimes N'\otimes Y\ar@{~>}[r] &
}
$$
in $\T$ whose rows are exact triangles, and we obtain a morphism $h$ as in it.
It is observed from this diagram that $f\otimes g\otimes E=(f\otimes N'\otimes E)\circ(M\otimes g\otimes E)$ is a zero morphism.
\end{proof}

\begin{lem}\label{lift}
\begin{enumerate}[\rm(1)]
\item
Let $f:X\to Y$ be a morphism in $\dm(R)$.
Let $\xx=x_1,\dots,x_n$ be a sequence of elements of $R$.
If $f\ltensor_RR/(\xx)=0$ in $\dm(R)$, then $f^{\otimes 2^n}\ltensor_R\k(\xx, R)=0$ in $\dm(R)$.
\item
Let $\xx =x_1, \ldots, x_n$ be a sequence of elements of $R$, and let $e>0$ be an integer.
Then $\k(\xx^e, R)$ belongs to $\thick_{\km(R)}^{ne}(\k(\xx, R))$, where $\xx^e=x_1^e, \ldots, x_n^e$.
\end{enumerate}
\end{lem}

\begin{proof}
(1) We use induction on $n$.
Let $n=1$ and set $x=x_1$.
There are exact sequences $0\to(0:x)\to R\to(x)\to0$ and $0\to(x)\to R\to R/(x)\to0$.
Applying the octahedral axiom to $(R\to(x)\to R)=(R\xrightarrow{x}R)$ gives an exact triangle $(0:x)[1] \to \k(x,R) \to R/(x) \rightsquigarrow$ in $\dm(R)$.
We have $f\ltensor_R R/(x)=0$, and $f\ltensor_R(0:x)[1]=(f\ltensor_RR/(x))\ltensor_{R/(x)}(0:x)[1]=0$.
Lemma \ref{cl}(1) yields $f^{\otimes2}\ltensor_R\k(x,R)=0$.

Let $n\ge2$.
We have $0=f\ltensor_RR/(\xx)=(f\ltensor_RR/(x_1))\ltensor_{R/(x_1)}R/(\xx)$.
The induction hypothesis gives
$$
0=(f\ltensor_RR/(x_1))^{\otimes 2^{n-1}}\ltensor_{R/(x_1)}\k(x_2,\dots,x_n,R/(x_1))=(f^{\otimes 2^{n-1}}\ltensor_R\k(x_2,\dots,x_n,R))\ltensor_RR/(x_1).
$$
The induction basis shows $0=(f^{\otimes 2^{n-1}}\ltensor_R\k(x_2,\dots,x_n,R))^{\otimes2}\ltensor_R\k(x_1,R)=f^{\otimes 2^n}\ltensor_R\k(x_2,\dots,x_n,\xx,R)$.
Note that $\k(\xx,R)$ is a direct summand of $\k(x_2,\dots,x_n,\xx,R)$; see \cite[Proposition 1.6.21]{BH}.
We thus obtain the desired equality $f^{\otimes 2^n}\ltensor_R\k(\xx,R)=0$.

(2) Again, we use induction on $n$.
Consider the case $n=1$.
Put $x=x_1$.
Applying the octahedral axiom to $(R\xrightarrow{x^{e-1}}R\xrightarrow{x}R)=(R\xrightarrow{x^e}R)$, we get an exact triangle $\k(x^{e-1}, R) \to \k(x^e, R) \to \k(x, R) \rightsquigarrow$.
Induction on $e$ shows $\k(x^e, R)\in\thick^e\k(x,R)$.
Let $n\ge2$.
By the induction hypothesis, $\k(x_1^e,\dots,x_{n-1}^e,R)$ belongs to $\thick^{(n-1)e}\k(x_1,\dots,x_{n-1},R)$.
Applying the exact functor $-\otimes\k(x_n^e,R)$, we see that $\k(\xx^e,R)$ belongs to $\thick^{(n-1)e}\k(x_1,\dots,x_{n-1},x_n^e,R)$.
Applying the exact functor $\k(x_1,\dots,x_{n-1},R)\otimes-$ to the containment $\k(x_n^e,R)\in\thick^e\k(x_n,R)$ gives rise to $\k(x_1,\dots,x_{n-1},x_n^e,R)\in\thick^e\k(\xx,R)$.
Therefore $\k(\xx^e,R)$ belongs to $\thick^{ne}\k(\xx,R)$.
\end{proof}

We now achieve the goal of generalizing the Hopkins--Neeman smash nilpotence theorem.

\begin{thm}[Generalized Smash Nilpotence]\label{gsn}
Let $f:X\to Y$ be a morphism in $\km(R)$ with $Y\in\kb(\proj R)$.
Suppose that $f\otimes\kappa(\p)=0$ for all $\p\in\spec R$.
Then $f^{\otimes t}=0$ for some $t>0$.
\end{thm}

\begin{proof}
We have an ascending chain $\ann_R(f) \subseteq \ann_R(f^{\otimes 2}) \subseteq \ann_R(f^{\otimes 3}) \subseteq \cdots$ of ideals of $R$.
Since $R$ is noetherian, there is an integer $c$ such that $\ann_R(f^{\otimes c})=\ann_R(f^{\otimes i})$ for all $i>c$.
Replacing $f$ by $f^{\otimes c}$, we may assume that $\ann_R(f) = \ann_R(f^{\otimes i})$ for all $i>0$.
Note that $\ann_R(f)=R$ if and only if $f=0$.

We assume $\ann_R(f) \neq R$, and shall derive a contradiction.
Take a minimal prime ideal $\p$ of $\ann_R(f)$.
Then localization at $\p$ reduces to the following situation:
\begin{quote}
$(R, \m, k)$ is a local ring, $\ann_R(f)$ is an $\m$-primary ideal, $f \otimes_R k = 0$ and $\ann_R(f) = \ann_R(f^{\otimes i})$ for all $i>0$.
\end{quote}
Indeed, since $Y$ is in $\kb(\proj R)$, it follows from \cite[Lemma 5.2(b)]{AF} that the map $\tau_{X,Y,\p}$ is an isomorphism, and Proposition \ref{annihi}(1) yields $\ann_{R_\p}(f_\p)=\ann_R(f)_\p$, which is a $\p R_\p$-primary ideal of $R_\p$.
Also, we have $\ann_{R_\p}(f_\p)= \ann_R(f)_\p= \ann_R(f^{\otimes i})_\p =\ann_{R_\p}((f^{\otimes i})_\p)=\ann_{R_\p}((f_\p)^{\otimes i})$ for all $i>0$.
Furthermore, it holds that $f_\p \otimes_{R_\p} \kappa(\p) =f \otimes_R \kappa(\p) =0$ by the assumption of the theorem.

For each nonnegative integer $n$, consider the following two statements.
\begin{enumerate}
\item[$F(n)$:]
Let $(R, \m, k)$ be a reduced local ring with $\dim R \le n$. 
Let $f:X\to Y$ be a morphism in $\km(R)$ with $Y\in\kb(\proj R)$.
If $\ann_R(f)$ is $\m$-primary and $f \otimes_R k =0$, then $f^{\otimes t} =0$ for some $t>0$.
\item[$G(n)$:]
Let $(R, \m, k)$ be a local ring with $\dim R \le n$.
Let $f:X\to Y$ be a morphism in $\km(R)$ with $Y\in\kb(\proj R)$.
If $\ann_R(f)$ is $\m$-primary and $f \otimes_R k =0$, then $f^{\otimes t} =0$ for some $t>0$.
\end{enumerate}
If the statement $G(n)$ holds true for all $n\ge0$, we have $\ann_R(f) = \ann_R(f^{\otimes t}) =R$, which gives a desired contradiction.
Note that the statement $F(0)$ always holds true since a $0$-dimensional reduced local ring is a field.
It is thus enough to show the implications $F(n) \Rightarrow G(n) \Rightarrow F(n+1)$.

$F(n) \Rightarrow G(n)$:
We consider the reduced ring $R_\red=R/\nil R$, where $\nil R$ stands for the nilradical of $R$.
The ideal $\ann_{R_\red}(f \otimes_R R_\red)$ of $R_\red$ is $\m R_\red$-primary since it contains $(\ann_Rf)R_\red$.
We have $(f \otimes_{R} R_\red) \otimes_{R_\red}k=f \otimes_R k = 0$.
Thus $R_\red$ and $f\otimes_RR_\red$ satisfy the assumption $F(n)$, and we find an integer $t > 0$ such that $f^{\otimes t} \otimes_R R_\red= (f \otimes_R R_\red)^{\otimes t}=0$.
Using Lemma \ref{lift}(1), we get $f^{\otimes tu} \otimes_R \k(\xx, R)$=0, where $\xx=x_1, \cdots, x_n$ is a system of generators of $\nil R$ and $u=2^n$.
Choose an integer $e>0$ such that $x_i^e=0$ for all $1\le i\le n$.
Then $R$ is a direct summand of $\k(\xx^e, R)$ by \cite[Proposition 1.6.21]{BH}, whence $R$ is in $\thick^{ne}K(\xx, R)$ by Lemma \ref{lift}(2).
Finally, Lemma \ref{cl}(2) gives rise to the equality $f^{\otimes netu}=0$.

$G(n) \Rightarrow F(n+1)$: 
We may assume $\dim R = n+1 >0$.
Since $R$ is reduced and $\ann_R(f)$ is $\m$-primary, we can choose an $R$-regular element $x\in\ann_R(f)$.
Then the local ring $R/(x)$ has dimension $n$, the ideal $\ann_{R/(x)}(f \otimes_R R/(x))$ of $R/(x)$ is $\m/(x)$-primary and $(f \otimes_R R/(x))\otimes_{R/(x)}k=0$.
Hence $R/(x)$ and $f\otimes_RR/(x)$ satisfy the assumption of $G(n)$, and there is an integer $t>0$ such that $(f \otimes_R R/(x))^{\otimes t}=0$.
The short exact sequence $0 \to R \xrightarrow{x} R \to R/(x) \to 0$ induces an exact triangle $R/(x)[-1] \to R \xrightarrow{x} R \rightsquigarrow$ in $\dm(R)$.
Tensoring $Y$ with this gives an exact triangle $Y \otimes_R R/(x)[-1] \xrightarrow{g} Y \xrightarrow{x} Y \rightsquigarrow$ in $\dm(R)$.
As $xf=0$, there is a morphism $h:X \to Y \otimes_R R/(x)[-1]$ with $f=gh$.
Now $f^{\otimes t+1}$ is decomposed as follows:
$$
X^{\otimes t+1} \xrightarrow{h \otimes X^{\otimes t}} (Y \otimes_R R/(x)[-1]) \otimes_R X^{\otimes t} \xrightarrow{(Y \otimes R/(x)[-1]) \otimes f^{\otimes t}} (Y \otimes_R R/(x)[-1]) \otimes_R Y^{\otimes t} \xrightarrow{g \otimes Y^{\otimes t}} Y^{\otimes t+1}.
$$
The middle morphism is identified with $Y[-1] \otimes_R (f \otimes_R R/(x))^{\otimes t}$, which is zero.
Thus, $f^{\otimes t+1}=0$.
\end{proof}

\begin{rem}
\begin{enumerate}[(1)]
\item
Theorem \ref{gsn} extends the smash nilpotence theorem due to Hopkins \cite[Theorem 10]{H} and Neeman \cite[Theorem 1.1]{N}, where $X$ is also assumed to belong to $\kb(\proj R)$, so that $f:X\to Y$ is a morphism in $\kb(\proj R)$.
Under this assumption one can reduce to the case where $X=R$, which plays a key role in the proof of the original Hopkins--Neeman smash nilpotence theorem.
\item
The proof of Theorem \ref{gsn} has a similar frame to that of the original Hopkins--Neeman smash nilpotence theorem, but we should notice that various delicate modifications are actually made there.
Indeed, Proposition \ref{annihi}, Lemmas \ref{cl} and \ref{lift} are all established to prove Theorem \ref{gsn}, which are not necessary to prove the original smash nilpotence theorem.
\item
The assumption in Theorem \ref{gsn} that $Y$ belongs to $\kb(\proj R)$ is used only to have $\ann_{R_\p}(f_\p)=\ann_R(f)_\p$.
\end{enumerate}
\end{rem}

Our next goal is to classify cocompactly generated thick tensor ideals of $\dm(R)$.
To this end, we begin with deducing the following proposition concerning generation of thick tensor ideals of $\dm(R)$, which will play an essential role throughout the rest of the paper.

\begin{prop}\label{key}
Let $X$ be an object of $\dm(R)$, and let $\Y$ be a subcategory of $\dm(R)$.
If $\V(\ann X)\subseteq\supp\Y$, then $X\in\tthick\Y$.
\end{prop}

\begin{proof}
Clearly, we may assume $X\ne0$.
We prove the proposition by replacing $\dm(R)$ with $\km(R)$.
There are a finite number of prime ideals $\p_1,\dots,\p_n$ of $R$ such that $\V(\ann X)=\bigcup_{i=1}^n\V(\p_i)$.
Since each $\p_i$ is in the support of $\Y$, we find an object $Y_i\in\Y$ with $\p_i\in\supp Y_i$.
All $\p_i$ are in the support of $Y:=Y_1\oplus\cdots\oplus Y_n\in\km(R)$.
Choose an integer $t$ with $\p_1, \cdots, \p_n \in\bigcup_{i>t}\supp\h^i(Y)$, and let $Y' = (\cdots \to 0 \to 0 \to Y^t \to Y^{t+1} \to \cdots) \in \kb(\proj R)$ be the truncated complex of $Y$.
Then $\V(\ann X)$ is contained in $\supp Y'$.
Let $f: Y' \to Y$ be the natural morphism, and let $\phi:R \to \Hom_R(Y', Y)$ be the composition of the homothety morphism $R \to \Hom_R(Y, Y)$ and $\Hom_R(f, Y): \Hom_R(Y, Y) \to \Hom_R(Y', Y)$.
There is an exact triangle $Z \xrightarrow{\psi} R \xrightarrow{\phi} \Hom_R(Y', Y)  \rightsquigarrow$ in $\km(R)$.
We establish two claims.
\begin{claim}\label{c1}
Let $\Phi:R\to C$ be a nonzero morphism in $\km(R)$.
If $R$ is a field, then $\Phi$ is a split monomorphism.
\end{claim}
\begin{proof}
Since $C$ is isomorphic to $\h(C)$ in $\km(R)$, we may assume that the differentials of $C$ are zero.
As $z:=\Phi^0(1)$ is nonzero, we can construct a chain map $\Psi:C\to R$ with $\Psi^0(z)=1$ and $\Psi^i=0$ for all $i\ne0$.
It then holds that $\Psi\Phi=1$.
\renewcommand{\qedsymbol}{$\square$}
\end{proof}
\begin{claim}\label{c2}
The morphism $\phi \otimes_R  \kappa(\p)$ in $\km(\kappa(\p))$ is a split monomorphism for each $\p\in\V(\ann X)$.
\end{claim}
\begin{proof}[Proof of Claim]
Set $S=\bigcup_{i>t}\supp\h^i(Y)$; note that this contains $\V(\ann X)$.
We prove the stronger statement that $\phi \otimes \kappa(\p)$ is a split monomorphism for each $\p\in S$.
Since $Y'$ is a perfect complex, there are natural isomorphisms $\Hom_R(Y', Y) \otimes \kappa(\p)\cong \Hom_R(Y', Y \otimes \kappa(\p))\cong \Hom_{\kappa(\p)}(Y' \otimes \kappa(\p) , Y \otimes \kappa(\p))$, which says that $\phi \otimes  \kappa(\p)$ is identified with the natural morphism $\kappa(\p)\to\Hom_{\kappa(\p)}(Y' \otimes\kappa(\p) , Y \otimes \kappa(\p))$.
This induces a map $\h^0(\phi\otimes\kappa(\p)):\kappa(\p)\to\Hom_{\km(\kappa(\p))}(Y' \otimes\kappa(\p) , Y \otimes \kappa(\p))$, sending $1$ to $f\otimes\kappa(\p)$.
If $f\otimes\kappa(\p)=0$ in $\km(\kappa(\p))$, then we see that $\h^{>t}(Y\otimes\kappa(\p))=0$, contradicting the fact that $\p\in S$.
Thus $\h^0(\phi\otimes\kappa(\p))$ is nonzero, and so is $\phi\otimes\kappa(\p)$.
Applying Claim \ref{c1} completes the proof.
\renewcommand{\qedsymbol}{$\square$}
\end{proof}

Claim \ref{c2} implies $\psi \otimes_R \kappa(\p)=0$ for all $\p \in \V(\ann X)$.
Using Theorem \ref{gsn} for the morphism $\psi \otimes_R (R/\ann X)$ in $\km(R/\ann X)$, we have $\psi^{\otimes m} \otimes_R (R/\ann X)=0$ for some $m>0$. 
Lemma \ref{lift}(1) shows
\begin{equation}\label{ds}
0=\psi^{\otimes u} \otimes_R  \k(\xx, R):Z^{\otimes u}\otimes\k(\xx,R)\to\k(\xx,R),
\end{equation}
where $\xx=x_1,\dots,x_r$ is a system of generators of the ideal $\ann X$, and $u=2^rm$.

For each $i>0$, let $W_i$ be the cone of the morphism $\psi^{\otimes i}:Z^{\otimes i} \to R$.
Applying the octahedral axiom to the composition $\psi\circ(\psi^{\otimes i}\otimes Z)=\psi^{\otimes i+1}$, we get an exact triangle $W_i\otimes Z \to W_{i+1} \to W_1 \rightsquigarrow$ in $\km(R)$.
As $W_1\cong\Hom_R(Y',Y)$ and $Y'\in\kb(\proj R)$, we see that $W_1$ is in $\thick Y$.
Using the triangle, we inductively observe that $W_i$ belongs to $\tthick Y$ for all $i>0$, and so does $W_u\otimes\k(\xx,R)$.
It follows from \eqref{ds} that $\k(\xx,R)$ is a direct summand of $W_u\otimes\k(\xx,R)$, and therefore $\k(\xx,R)$ belongs to $\tthick Y$.

There is an exact triangle $R\xrightarrow{x_i}R\to\k(x_i,R)\rightsquigarrow$ in $\km(R)$ for each $1\le i\le r$.
Tensoring $X$ with this and using the fact that each $x_i$ kills $X$, we see that $X$ is a direct summand of $X\otimes\k(\xx,R)$.
Consequently, $X$ belongs to $\tthick Y$.
By construction $Y$ is in $\thick\Y$, and hence $X$ belongs to $\tthick\Y$.
\end{proof}

\begin{rem}
\begin{enumerate}[(1)]
\item
Proposition \ref{key} extends Neeman's result \cite[Lemma 1.2]{N}, where both $X$ and $\Y$ are contained in $\kb(\proj R)$ (and $\Y$ is assumed to consist of a single object).
\item
Proposition \ref{key} is no longer true if we replace $\V(\ann X)$ with $\supp X$, or if we replace $\supp\Y$ with $\V(\ann\Y)$.
This will be explained in Remarks \ref{repl}(1) and \ref{repl2}.
\end{enumerate}
\end{rem}

The following result is a consequence of Proposition \ref{key}, which will often be used later.

\begin{cor}\label{3'}
Let $\X$ be a thick $\otimes$-ideal of $\dm(R)$.
Let $I$ be an ideal of $R$ and $\xx=x_1,\dots,x_n$ a system of generators of $I$.
Then there are equivalences:
$$\V(I)\subseteq\supp\X\ \Leftrightarrow\ R/I\in\X\ \Leftrightarrow\ \k(\xx,R)\in\X.
$$
\end{cor}

\begin{proof}
Proposition \ref{annihi}(3) implies $\supp R/I=\V(\ann R/I)=\V(I)=\V(\ann\k(\xx,R))=\supp\k(\xx,R)$.
The assertion is shown by combining this with Proposition \ref{key}.
\end{proof}

Now we can give a complete classification of the cocompactly generated thick tensor ideals of $\dm(R)$, using Proposition \ref{key}.
For each subset $S$ of $\spec R$, we set $\langle S\rangle=\tthick\{R/\p\mid\p\in S\}$.

\begin{thm}\label{main}
The assignments $\X\mapsto\supp\X$ and $\langle W\rangle\mapsfrom W$ make mutually inverse bijections
$$
\{\text{Cocompactly generated thick $\otimes$-ideals of $\dm(R)$}\}\rightleftarrows\{\text{Specialization-closed subsets of $\spec R$}\}.
$$
\end{thm}

\begin{proof}
Proposition \ref{si}(2) shows that the map $\X\mapsto\supp\X$ is well-defined and that for a specialization-closed subset $W$ of $\spec R$ the equality $W=\supp\langle W\rangle$ holds.
It remains to show that for any cocompactly generated thick $\otimes$-ideal $\X$ of $\dm(R)$ one has $\X = \langle\supp \X\rangle$.
Proposition \ref{key} implies that $\X$ contains $\langle\supp \X\rangle$.
Since $\X$ is cocompactly generated, there is a subcategory $\C$ of $\db(R)$ with $\X= \tthick \C$ by Proposition \ref{cpt}.
Thus, it suffices to prove that each $M\in\C$ belongs to $\langle\supp \X\rangle$.
The complex $M$ belongs to $\thick\h(M)$ as $M\in\db(R)$, and the finitely generated module $\h(M)$ has a finite filtration each of whose subquotients has the form $R/\p$ with $\p\in\supp\h(M)$.
Hence $M$ is in $\langle\supp M\rangle$, and we are done.
\end{proof}

Let us give several applications of our Theorem \ref{main}.

\begin{cor}\label{strcpt}
\begin{enumerate}[\rm(1)]
\item
Let $\C$ be a subcategory of $\db(R)$.
Then $\ttthick_{\dm(R)}\,\C$ consists of the complexes $X\in\dm(R)$ with $\V(\ann X)\subseteq\supp\C$.
In particular, those complexes $X$ form a thick $\otimes$-ideal of $\dm(R)$.
\item
Let $I$ be an ideal of $R$.
Then $\ttthick_{\dm(R)}(R/I)$ consists of the complexes $X\in\dm(R)$ with $I\subseteq\sqrt{\ann X}$.
\item
Let $W$ be a specialization-closed subset of $\spec R$.
Then $\langle W\rangle$ consists of the complexes $X\in\dm(R)$ such that $\V(\ann X)\subseteq W$.
\item
Let $\X,\Y$ be thick subcategories in $\db(R)$.
Then $\tthick{\X} = \tthick{\Y}$ if and only if $\supp{\X}=\supp{\Y}$.
\end{enumerate}
\end{cor}

\begin{proof}
(1) Let $\X$ be the subcategory of $\dm(R)$ consisting of objects $X\in\dm(R)$ with $\V(\ann X)\subseteq\supp\C$.
Proposition \ref{key} says that $\tthick\C$ contains $\X$.
Propositions \ref{si}(2), \ref{cpt} and Theorem \ref{main} yield $\tthick\C=\langle\supp(\tthick\C)\rangle=\langle\supp\C\rangle$.
For each $\p\in\supp\C$, the set $\V(\ann R/\p)=\V(\p)$ is contained in $\supp\C$, whence $R/\p$ is in $\X$.
Hence $\tthick\C$ is contained in $\X$, and we get the equality $\tthick\C=\X$.

(2) Applying (1) to $\C =\{R/I \}$, we immediately obtain the assertion.

(3) Setting $\C=\{R/\p\mid\p\in W\}\subseteq\db(R)$, we have $\supp\C=W$.
The assertion follows from (1).

(4) Let $\C$ be either $\X$ or $\Y$.
By Proposition \ref{cpt} the thick $\otimes$-ideal $\tthick\C$ is cocompactly generated, and $\supp(\tthick\C)=\supp\C$ by Proposition \ref{si}(2).
The assertion now follows from Theorem \ref{main}.
\end{proof}

We obtain the following one-to-one correspondence by combining our Theorem \ref{main} with the celebrated Hopkins--Neeman classification theorem \cite[Theorem 1.5]{N}.

\begin{cor}\label{dmperf}
The assignments $\X\mapsto\X\cap\kb(\proj R)$ and $\tthick\Y\mapsfrom\Y$ make mutually inverse bijections
$$
\{\text{Cocompactly generated thick $\otimes$-ideals of $\dm(R)$}\}\rightleftarrows\{\text{Thick subcategories of $\kb(\proj R)$}\}.
$$
In particular, all cocompactly generated thick $\otimes$-ideals of $\dm(R)$ are compactly generated. 
\end{cor}

\begin{proof}
It is directly verified (resp. follows from Proposition \ref{cpt}) that the assignment $\X\mapsto\X\cap\kb(\proj R)$ (resp. $\tthick\Y\mapsfrom\Y$) makes a well-defined map.
It follows from \cite[Theorem 1.5]{N} that
\begin{enumerate}[(\#)]
\item
the assignments $\X \mapsto \supp \X$ and $W \mapsto \supp_{\kb(\proj R)}^{-1}(W):=\supp^{-1}W\cap\kb(\proj R)$ make mutually inverse bijections between the thick subcategories of $\kb(\proj R)$ and the specialization-closed subsets of $\spec R$.
\end{enumerate}
In view of Theorem \ref{main} and (\#), we have only to show that
\begin{enumerate}[(a)]
\item
$\supp_{\kb(\proj R)}^{-1}(\supp \X) = \X \cap \kb(\proj R)$ for any cocompactly generated thick $\otimes$-ideal $\X$ of $\dm(R)$, and
\item
$\langle\supp \Y \rangle = \tthick\Y$ for any thick subcategory $\Y$ of $\kb(\proj R)$.
\end{enumerate}

Using Propositions \ref{cpt} and \ref{si}(2), we see that $\langle\supp \Y \rangle$ and $\tthick\Y$ are cocompactly generated thick $\otimes$-ideals of $\dm(R)$ whose supports are equal to $\supp \Y$.
Now Theorem \ref{main} shows the statement (b).

Clearly, $\supp(\X \cap \kb(\proj R))$ is contained in $\supp \X$.
Take a prime ideal $\p\in\supp \X$, and let $\xx$ be a system of generators of $\p$.
Then $\V(\ann \k(\xx, R)) = \supp \k(\xx, R) = \V(\p) \subseteq \supp \X$ by Proposition \ref{annihi}(3), and $\k(\xx, R) \in \X \cap \kb(\proj R)$ by Proposition \ref{key}.
It follows that $\p \in \supp \k(\xx, R) \subseteq \supp(\X \cap \kb(\proj R))$.
Thus we get $\supp(\X \cap \kb(\proj R))= \supp \X$, and obtain $\supp_{\kb(\proj R)}^{-1}(\supp \X) =\supp_{\kb(\proj R)}^{-1}(\supp(\X \cap \kb(\proj R)))= \X \cap \kb(\proj R)$, where the last equality is shown by (\#).
Now the statement (a) is proved.
\end{proof}

\begin{rem}
Corollary \ref{dmperf} in particular gives a classification of the {\em compactly} generated thick $\otimes$-ideals of $\dm(R)$.
This itself can also be deduced as follows:
Let $\X,\Y$ be thick subcategories of $\kb(\proj R)$ with $\supp(\tthick\X)=\supp(\tthick\Y)$.
Then $\supp\X=\supp\Y$ by Proposition \ref{si}(2), and the Hopkins--Neeman theorem yields $\X=\Y$.
Hence $\tthick\X=\tthick\Y$.

The essential benefit that Corollary \ref{dmperf} produces is the classification of the {\em cocompactly} generated thick $\otimes$-ideals of $\dm(R)$.
This should not follow from the Hopkins--Neeman theorem or other known results, but require the arguments established in this section so far (especially, the Generalized Smash Nilpotence Theorem \ref{gsn}).
A {\em compactly} generated thick tensor ideal of $\dm(R)$ is clearly {\em cocompactly} generated by Proposition \ref{cpt}, but the converse (shown in Corollary \ref{dmperf}) should be rather non-trivial.
\end{rem}

In view of Corollary \ref{dmperf} and Proposition \ref{cpt}, we obtain the following result and definition.

\begin{cor}\label{seven}
The following four conditions are equivalent for a thick $\otimes$-ideal $\X$ of $\dm(R)$.
\begin{tabbing}
\qquad$\bullet$ $\X$ is \phantom{co}compactly generated.\qquad\=$\bullet$ $\X$ is generated by objects in $\kb(\proj R)$.\\
\qquad$\bullet$ $\X$ is cocompactly generated.\>$\bullet$ $\X$ is generated by objects in $\db(R)$.
\end{tabbing}
\end{cor}

\begin{dfn}
Let $\X$ be a thick $\otimes$-ideal of $\dm(R)$.
We say that $\X$ is {\em compact} if it satisfies one (hence all) of the equivalent conditions in Corollary \ref{4}.
\end{dfn}

Next, for two thick $\otimes$-ideals $\X,\Y$ of $\dm(R)$ we define the thick $\otimes$-ideals $\X \wedge \Y$ and $\X \vee \Y$ by:
$$
\X \wedge \Y = \tthick\{X \ltensor_R Y \mid X \in \X,\, Y \in \Y \},\qquad
\X \vee \Y = \tthick(\X \cup \Y).
$$
These two operations yield a lattice structure in the compact thick $\otimes$-ideals of $\dm(R)$:

\begin{prop}\label{lat}
\begin{enumerate}[\rm(1)]
\item
Let $A$ and $B$ be specialization-closed subsets of $\spec R$.
One then has equalities
$$
\langle A\rangle \wedge\langle B\rangle = \langle A \cap B \rangle,\qquad
\langle A\rangle \vee \langle B\rangle = \langle A \cup B\rangle.
$$
\item
The set of compact thick $\otimes$-ideals of $\dm(R)$ is a lattice with meet $\wedge$ and join $\vee$.
\end{enumerate}
\end{prop}

\begin{proof}
(1) It is evident that the second equality holds.
Let us show the first one.

We claim that for two subcategories $\M,\N$ of $\dm(R)$ it holds that
$$
(\tthick\M)\wedge(\tthick\N)=\tthick\{M\ltensor_RN \mid M \in \M,\, N \in \N\}.
$$
In fact, clearly $(\tthick\M)\wedge(\tthick\N)$ contains $\C:=\tthick\{M\ltensor_RN \mid M \in \M,\, N \in \N\}$.
For each $N\in\N$, the subcategory of $\dm(R)$ consisting of objects $X$ with $X\ltensor_RN\in\C$ is a thick $\otimes$-ideal containing $\M$, so contains $\tthick\M$.
Let $X$ be an object in $\tthick\M$.
Then $X\ltensor_RN$ belongs to $\C$ for all $N\in\N$.
The subcategory of $\dm(R)$ consisting of objects $Y$ with $X\ltensor_RY\in\C$ is a thick $\otimes$-ideal containing $\N$, so contains $\tthick\N$.
Hence $X\ltensor_RY$ is in $\C$ for all $X\in\tthick\M$ and $Y\in\tthick\N$, and the claim follows.

Using the claim, we see that $\langle A\rangle\wedge\langle B\rangle=\tthick\{R/\p\ltensor_RR/\q\mid\p\in A,\,\q\in B\}$.
Therefore
\begin{align*}
\supp(\langle A\rangle\wedge\langle B\rangle)
&=\supp\{R/\p\ltensor_RR/\q\mid\p\in A,\,\q\in B\}\\
&=\textstyle\bigcup_{\p\in A,\,\q\in B}\supp(R/\p\ltensor_RR/\q)
=\textstyle\bigcup_{\p\in A,\,\q\in B}(\V(\p)\cap\V(\q))
=A\cap B
=\supp\langle A\cap B\rangle
\end{align*}
by Proposition \ref{si}(2), Lemma \ref{si0}(4) and the assumption that $A,B$ are specialization-closed.
Theorem \ref{main} implies that $\langle A\rangle\wedge\langle B\rangle=\langle A\cap B\rangle$.

(2) Let $\X,\Y$ be compact thick $\otimes$-ideals of $\dm(R)$.
Theorem \ref{main} implies that $\X=\langle\supp\X\rangle$ and $\Y=\langle\supp\Y\rangle$, and $\supp\X$ and $\supp\Y$ are specialization-closed.
It follows from (1) that $\X\wedge\Y=\langle\supp\X\cap\supp\Y\rangle$ and $\X\vee\Y=\langle\supp\X\cup\supp\Y\rangle$, which are compact.
It is seen by definition that any thick $\otimes$-ideal containing both $\X$ and $\Y$ contains $\X\vee\Y$.
Let $\ZZ$ be a compact thick $\otimes$-ideal contained in both $\X$ and $\Y$.
By Theorem \ref{main} again we get $\ZZ=\langle\supp\ZZ\rangle$.
Since $\supp\ZZ$ is contained in $\supp\X\cap\supp\Y$, we have that $\ZZ$ is contained in $\X\wedge\Y$.
These arguments prove the assertion.
\end{proof}

Note that the specialization-closed subsets of $\spec R$ form a lattice with meet $\cap$ and join $\cup$.
As an immediate consequence of this fact and Proposition \ref{lat}(2), we obtain a refinement of Theorem \ref{main}:

\begin{thm}\label{main2}
The assignments $\X\mapsto\supp\X$ and $\langle W\rangle\mapsfrom W$ induce a lattice isomorphism
$$
\{\text{Compact thick $\otimes$-ideals of $\dm(R)$}\}\cong\{\text{Specialization-closed subsets of $\spec R$}\}.
$$
\end{thm}

Restricting to the artinian case, we get a complete classification of thick tensor ideals of $\dm(R)$.

\begin{cor}\label{art}
Let $R$ be an artinian ring.
Then the following statements are true.
\begin{enumerate}[\rm(1)]
\item
All the thick $\otimes$-ideals of $\dm(R)$ are compact.
\item
The assignments $\X\mapsto\supp\X$ and $\langle S\rangle\mapsfrom S$ induce a lattice isomorphism
$$
\{\text{Thick $\otimes$-ideals of $\dm(R)$}\}\cong\{\text{Subsets of $\spec R$}\}.
$$
\end{enumerate}
\end{cor}

\begin{proof}
(1) Take any thick $\otimes$-ideal $\X$ of $\dm(R)$.
We want to show $\X=\langle\supp\X\rangle$.
Corollary \ref{3'} implies that $\X$ contains $\langle\supp\X\rangle$.
To show the opposite inclusion, we may assume that $\X$ consists of a single object $X$.
Let $\m_1,\dots,\m_s,\m_{s+1},\dots,\m_n$ be the maximal ideals of $R$ with $\supp X=\{\m_1,\dots,\m_s\}$.
Find an integer $t>0$ with $(\m_1\cdots\m_n)^t=0$.
The Chinese remainder theorem yields an isomorphism $R\cong R/\m_1^t\oplus\cdots\oplus R/\m_n^t$ of $R$-modules.
Tensoring $X$, we obtain an isomorphism $X\cong(X\ltensor_RR/\m_1^t)\oplus\cdots\oplus(X\ltensor_RR/\m_n^t)$.
Lemma \ref{si0}(4) gives $\supp(X\ltensor_RR/\m_i^t)=\supp X\cap\{\m_i\}$, which is an empty set for $s+1\le i\le n$.
For such an $i$ we have $X\ltensor_RR/\m_i^t=0$ by Remark \ref{r}, and get $X\cong(X\ltensor_RR/\m_1^t)\oplus\cdots\oplus(X\ltensor_RR/\m_s^t)$.
It follows that $X$ is in $\tthick\{R/\m_1^t,\dots,R/\m_s^t\}$, which is the same as $\langle\supp X\rangle$ by Corollary \ref{strcpt}.

(2) Since all prime ideals of $R$ are maximal, every subset of $\spec R$ is specialization-closed.
(A more general statement will be given in Lemma \ref{finsp}.)
The assertion follows from (1) and Theorem \ref{main2}.
\end{proof}

\section{Correspondence between the Balmer and Zariski spectra}\label{sect:bal}

In this section, we construct a pair of maps between the Balmer spectrum $\spc\dm(R)$ and the Zariski spectrum $\spec R$, which will play a crucial role in later sections.
First of all, let us recall the definitions of a prime thick tensor ideal of a tensor triangulated category and its Balmer spectrum.

\begin{dfn}
Let $\T$ be an essentially small tensor triangulated category.
A thick $\otimes$-ideal $\P$ of $\T$ is called {\em prime} provided that $\P\ne\T$ and if $X\otimes Y$ is in $\P$, then so is either $X$ or $Y$.
The set of prime thick $\otimes$-ideals of $\T$ is denoted by $\spc\T$ and called the {\em Balmer spectrum} of $\T$.
\end{dfn}

Here is an example of a prime thick tensor ideal of $\dm(R)$.

\begin{ex}\label{zerop}
When $R$ is local, the zero subcategory $\zero$ of $\dm(R)$ is a prime thick $\otimes$-ideal.
In fact, it is easy to verify that $\zero$ is a thick $\otimes$-ideal of $\dm(R)$.
(This also follows from Remark \ref{r} and Proposition \ref{si}(1).)
If $X,Y$ are objects of $\dm(R)$ with $X\ltensor_RY=0$, then either $X=0$ or $Y=0$ by Lemma \ref{si0}(4).
\end{ex}

Now we introduce the following notation.

\begin{nota}
For a prime ideal $\p$ of $R$, we denote by $\s(\p)$ the subcategory of $\dm(R)$ consisting of complexes $X$ with $X_\p\cong0$ in $\dm(R_\p)$.
\end{nota}

The subcategory $\s(\p)$ is always a prime thick tensor ideal:

\begin{prop}\label{two}
Let $\p$ be a prime ideal of $R$.
Then $\s(\p)$ is a prime thick $\otimes$-ideal of $\dm(R)$ satisfying
$$
\supp\s(\p)=\{\q\in\spec R\mid\q\nsubseteq\p\}.
$$
\end{prop}

\begin{proof}
Since $\s(\p)$ does not contain $R$, it is not equal to $\dm(R)$.
Note that $\s(\p)=\supp^{-1}(\{\p\}^\complement)$.
Using Lemma \ref{si0}(4) and Proposition \ref{si}(1), we observe that $\s(\p)$ is a prime thick $\otimes$-ideal of $\dm(R)$.

Fix a prime ideal $\q$ of $R$.
If $\q$ is in $\supp\s(\p)$, then there is a complex $X\in\s(\p)$ with $\q\in\supp X$, and it follows that $X_\p=0\ne X_\q$.
If $\q$ is contained in $\p$, then we have $X_\q=(X_\p)_\q$ and get a contradiction.
Therefore $\q$ is not contained in $\p$.
Conversely, assume this.
Take a system of generators $\xx=x_1,\dots,x_n$ of $\q$, and put $K=\k(\xx,R)$.
Then we have $K_\q\ne0=K_\p$ by Proposition \ref{annihi}(3).
Hence $K$ belongs to $\s(\p)$ and $\q$ is in $\supp K$, which implies $\q\in\supp\s(\p)$.
We thus obtain the equality in the proposition.
\end{proof}

As an easy consequence of the above proposition, we get another example of a prime thick tensor ideal.

\begin{cor}\label{df}
Let $R$ be an integral domain of dimension one.
It then holds that $\df(R)=\s((0))$, where $(0)$ stands for the zero ideal of $R$.
Hence $\df(R)$ is a prime thick $\otimes$-ideal of $\dm(R)$.
\end{cor}

\begin{proof}
For a complex $X\in\dm(R)$ it holds that
$$
X\in\df(R)\ \Leftrightarrow\ 
\ell(\h^iX)<\infty\text{ for all }i\ \Leftrightarrow\ 
\h^iX_{(0)}=0\text{ for all }i\ \Leftrightarrow\ 
X_{(0)}=0\ \Leftrightarrow\ 
X\in\s((0)),
$$
where the second equivalence follows from the fact that $\spec R=\{(0)\}\cup\Max R$.
This shows $\df(R)=\s((0))$.
Proposition \ref{two} implies that $\s((0))$ is prime, which gives the last statement of the corollary.
\end{proof}

\begin{rem}
Corollary \ref{df} is no longer valid if we remove the assumption that $R$ is an integral domain.
More precisely, the assertion of the corollary is not true even if $R$ is reduced.
In fact, consider the ring $R=k[[x,y]]/(xy)$, where $k$ is a field.
Then $R$ is a $1$-dimensional reduced local ring.
It is observed by Proposition \ref{annihi}(3) that the Koszul complexes $\k(x,R),\k(y,R)$ are outside $\df(R)$, while the complex $\k(x,R)\ltensor_R\k(y,R)=\k(x,y,R)$ is in $\df(R)$.
This shows that $\df(R)$ is not prime.
\end{rem}

We have constructed from each prime ideal $\p$ of $R$ the prime thick tensor ideal $\s(\p)$ of $\dm(R)$.
Now we are concerned with the opposite direction, that is, we construct from a prime thick tensor ideal of $\dm(R)$ a prime ideal of $R$, which is done in the following proposition.

\begin{prop}\label{pp}
Let $\P$ be a prime thick $\otimes$-ideal of $\dm(R)$.
Let $K$ be the set of ideals $I$ of $R$ such that $\V(I)$ is not contained in $\supp\P$.
Then $K$ has the maximum element $P$ with respect to the inclusion relation, and $P$ is a prime ideal of $R$.
\end{prop}

\begin{proof}
We claim that for ideals $I,J$ of $R$, if $\supp\P$ contains $\V(I+J)$, then it contains either $\V(I)$ or $\V(J)$.
Indeed, let $\xx=x_1,\dots,x_a$ and $\yy=y_1,\dots,y_b$ be systems of generators of $I$ and $J$, respectively.
Corollary \ref{3'} yields that $\k(\xx,\yy,R)$ is in $\P$.
There is an isomorphism $\k(\xx,R)\ltensor_R\k(\yy,R)\cong\k(\xx,\yy,R)$ of complexes, whence $\k(\xx,R)\ltensor_R\k(\yy,R)$ belongs to $\P$.
Since $\P$ is prime, it contains either $\k(\xx,R)$ or $\k(\yy,R)$.
Thus $\supp\P$ contains either $\V(I)$ or $\V(J)$ by Corollary \ref{3'} again.

The claim says that $K$ is closed under sums of ideals of $R$.
Taking into account that $R$ is noetherian, we see that $K$ has the maximum element $P$ with respect to the inclusion relation.
There is a filtration $0=M_0\subsetneq M_1\subsetneq \cdots \subsetneq M_t=R/P$ of submodules of the $R$-module $R/P$ such that for every $1\le i\le t$ one has $M_i/M_{i-1}\cong R/\p_i$ with some $\p_i\in\supp_RR/P$, whence each $\p_i$ contains $P$.
Suppose that $P$ is not a prime ideal of $R$.
Then the $\p_i$ strictly contain $P$, and the maximality of $P$ shows that $\supp\P$ contains $\V(\p_i)$.
There is an equality $\supp_RR/P=\bigcup_{i=1}^t\supp R/\p_i$, or equivalently, $\V(P)=\bigcup_{i=1}^t\V(\p_i)$.
It follows that $\supp\P$ contains $\V(P)$, which is a contradiction.
Consequently, $P$ is a prime ideal of $R$.
\end{proof}

Thus we have got two maps in the mutually inverse directions, between $\spec R$ and $\spc\dm(R)$:

\begin{nota}
Let $\P$ be a prime thick $\otimes$-ideal of $\dm(R)$.
With the notation of Proposition \ref{pp}, we set $\ii(\P)=K$ and $\pp(\P)=P$.
In view of Proposition \ref{two}, we obtain a pair of maps
$$
\s:\spec R\rightleftarrows\spc\dm(R):\pp
$$
given by $\p\mapsto\s(\p)$ and $\P\mapsto\pp(\P)$ for $\p\in\spec R$ and $\P\in\spc\dm(R)$.
\end{nota}

Now we compare the maps $\s,\pp$, and for this recall two basic definitions from set theory.
Let $f:A\to B$ be a map of partially ordered sets.
We say that $f$ is {\em order-reversing} if $x\le y$ implies $f(x)\ge f(y)$ for all $x,y\in A$.
Also, we call $f$ an {\em order anti-embedding} if $x\le y$ is equivalent to $f(x)\ge f(y)$ for all $x,y\in A$.
Note that any order anti-embedding is an injection.
We regard $\spec R$ and $\spc\dm(R)$ as partially ordered sets with respect to the inclusion relations.
The following theorem is the main result of this section.

\begin{thm}\label{s}
The maps $\s:\spec R\rightleftarrows\spc\dm(R):\pp$ are order-reversing, and satisfy
$$
\pp\cdot\s=1,\qquad
\s\cdot\pp=\supp^{-1}\supp.
$$
Hence, $\s$ is an order anti-embedding.
\end{thm}

\begin{proof}
Let $\p,\q$ be prime ideals of $R$ with $\q\subseteq\p$.
Then Proposition \ref{two} shows that $\q$ is not in $\supp\s(\p)$.
Hence $X_\q=0$ for all $X\in\s(\p)$, which means that $\s(\p)$ is contained in $\s(\q)$.
On the other hand, let $\P,\Q$ be prime thick $\otimes$-ideals of $\dm(R)$ with $\P\subseteq\Q$.
Then $\supp\P$ is contained in $\supp\Q$, and we see from the definition of $\pp$ that $\pp(\P)$ contains $\pp(\Q)$.
Therefore, the maps $\s,\pp$ are order-reversing.

Fix a prime ideal $\p$ of $R$.
Then $\pp(\s(\p))$ is the maximum element of $\ii(\s(\p))$, which consists of ideals $I$ with $\V(I)\nsubseteq\supp\s(\p)$.
This is equivalent to saying that $I\subseteq\p$ by Proposition \ref{two}.
Hence $\pp(\s(\p))=\p$.

Let $\P$ be a prime thick $\otimes$-ideal of $\dm(R)$.
Note that a prime ideal $\p$ of $R$ belongs to $\ii(\P)$ if and only if $\p$ is not in $\supp\P$.
Let $X\in\dm(R)$ be a complex with $X_{\pp(\P)}=0$.
If $\p$ is a prime ideal of $R$ with $X_\p\ne0$, then $\p$ is not contained in $\pp(\P)$, and $\p$ must not belong to $\ii(\P)$, which means $\p\in\supp\P$.
Therefore $\supp X$ is contained in $\supp\P$, and we obtain $\s(\pp(\P))\subseteq\supp^{-1}\supp\P$.
Conversely, let $X\in\dm(R)$ be a complex with $\supp X\subseteq\supp\P$.
Since $\pp(\P)$ is in $\ii(\P)$, it does not belong to $\supp\P$.
Hence $\pp(\P)$ is not in $\supp X$, which means $X_{\pp(\P)}=0$.
We thus conclude that $\s(\pp(\P))=\supp^{-1}\supp\P$.

The last assertion is shown by using the equality $\p=\pp(\s(\p))$ for all prime ideals $\p$ of $R$.
\end{proof}

The above theorem gives rise to several corollaries, which will often be used later.
The rest of this section is devoted to stating and proving them.

\begin{cor}\label{ato2}
Let $\p$ be a prime ideal of $R$, and let $\P$ a prime thick $\otimes$-ideal of $\dm(R)$.
It holds that:
$$
\p\subseteq\pp(\P)\ \Leftrightarrow\ R/\p\notin\P\ \Leftrightarrow\ \p\notin\supp\P\ \Leftrightarrow\ \P\subseteq\s(\p).
$$
In particular, $\pp(\P)$ is the maximum element of $(\supp\P)^\complement$ with respect to the inclusion relation.
\end{cor}

\begin{proof}
The second equivalence follows from Corollary \ref{3'}, while the third one is trivial.
If $\p\notin\supp\P$, then $\p\subseteq\pp(\P)$.
If this is the case, then $\s(\p)\supseteq\s(\pp(\P))=\supp^{-1}\supp\P\supseteq\P$ by Theorem \ref{s}.
\end{proof}

\begin{cor}\label{ppsupp}
For two prime thick $\otimes$-ideals $\P,\Q$ of $\dm(R)$ one has:
$$
\pp(\P)\subseteq\pp(\Q)\ \Leftrightarrow\ \supp\P\supseteq\supp\Q,\qquad\quad
\pp(\P)=\pp(\Q)\ \Leftrightarrow\ \supp\P=\supp\Q.
$$
\end{cor}

\begin{proof}
Theorem \ref{s} and Proposition \ref{si}(1) yield the first equivalence, which implies the second one.
\end{proof}

Here we introduce two notions of thick tensor ideals, which will play main roles in the rest of this paper.

\begin{dfn}
\begin{enumerate}[(1)]
\item
For a thick $\otimes$-ideal $\X$ of $\dm(R)$ we denote by $\sqrt\X$ the {\em radical} of $\X$, that is, the subcategory of $\dm(R)$ consisting of objects $M$ such that the $n$-fold tensor product $M\ltensor_R\cdots\ltensor_RM$ belongs to $\X$ for some $n\ge1$.
\item
A thick $\otimes$-ideal $\X$ of $\dm(R)$ is called {\em radical} if $\X=\sqrt\X$.
Any prime thick $\otimes$-ideal is radical.
\item
A thick $\otimes$-ideal $\X$ of $\dm(R)$ is called {\em tame} if one can write $\X=\supp^{-1}S$ for some subset $S$ of $\spec R$.
The set of tame prime thick $\otimes$-ideals of $\dm(R)$ is denoted by $\ts\dm(R)$.
\end{enumerate}
\end{dfn}

\begin{rem}\label{tame}
For each subcategory $\X$ of $\dm(R)$ the following are equivalent.
\begin{enumerate}[\quad(1)]
\item
$\X$ is a tame thick $\otimes$-ideal of $\dm(R)$.
\item
$\X=\supp^{-1}S$ for some subset $S$ of $\spec R$.
\item
$\X=\supp^{-1}W$ for some specialization-closed subset $W$ of $\spec R$.
\end{enumerate}
This is a direct consequence of Proposition \ref{si}(1).
\end{rem}

The following corollary of Theorem \ref{s} gives an explicit description of tame prime thick tensor ideals.

\begin{cor}\label{tp}
It holds that
$$
\ts\dm(R)=\im\s=\{\s(\p)\mid\p\in\spec R\}.
$$
\end{cor}

\begin{proof}
For a prime ideal $\p$ of $R$, we have $\s(\p)=\s\pp\s(\p)=\supp^{-1}(\supp\s(\p))$ by Theorem \ref{s}, which shows that the prime thick $\otimes$-ideal $\s(\p)$ of $\dm(R)$ is tame.
On the other hand, let $\P$ be a tame prime thick $\otimes$-ideal of $\dm(R)$.
Using Theorem \ref{s} and Proposition \ref{si}, we get $\s(\pp(\P))=\supp^{-1}(\supp\P)=\P$.
\end{proof}

Here is one more application of Theorem \ref{s}, giving a criterion for a thick tensor ideal to be prime.

\begin{cor}
Let $W$ be a specialization-closed subset of $\spec R$.
The following are equivalent.
\begin{enumerate}[\rm(1)]
\item\label{cr1}
The tame thick $\otimes$-ideal $\supp^{-1}W$ of $\dm(R)$ is prime.
\item\label{cr4}
There exists a prime ideal $\p$ of $R$ such that $W=\supp\s(\p)$.
\item\label{cr3}
There exists a prime thick $\otimes$-ideal $\P$ of $\dm(R)$ such that $W=\supp\P$.
\item\label{cr2}
The set $W^\complement$ has a unique maximal element with respect to the inclusion relation.
\end{enumerate}
\end{cor}

\begin{proof}
\eqref{cr1} $\Rightarrow$ \eqref{cr4}:
By Corollary \ref{ato2}, the complement of $W = \supp (\supp^{-1} W)$ (see Proposition \ref{si}(2)) has the maximum element $\p:=\pp(\supp^{-1}W)$.
Using Theorem \ref{s}, we obtain $W = \supp\s(\p)$.

\eqref{cr4} $\Rightarrow$ \eqref{cr3}:
Take $\P = \s(\p)$, which is a prime thick $\otimes$-ideal of $\dm(R)$ by Proposition \ref{two}.

\eqref{cr3} $\Rightarrow$ \eqref{cr2}:
This implication follows from Corollary \ref{ato2}.

\eqref{cr2} $\Rightarrow$ \eqref{cr1}:
Let $\p$ be a unique maximal element of $W^\complement$.
We claim that there is an equality $W=\supp\s(\p)$.
Indeed, $\supp\s(\p)$ consists of the prime ideals $\q$ of $R$ not contained in $\p$ by Proposition \ref{two}.
Now fix a prime ideal $\q$ of $R$.
Suppose that $\q$ is in $W$.
If $\q$ is contained in $\p$, then $\p$ belongs to $W$ as $W$ is specialization-closed.
This contradicts the choice of $\p$, whence $\q$ belongs to $\supp\s(\p)$.
Conversely, if $\q$ is not in $W$, then $\q$ is in $W^\complement$, and the choice of $\p$ shows that $\q$ is contained in $\p$.
Thus the claim follows.
Applying Theorem \ref{s}, we obtain $\supp^{-1}W=\s(\p)$ and this is a prime thick $\otimes$-ideal of $\dm(R)$.
\end{proof}

\section{Topological structures of the Balmer spectrum}\label{tsbs}

In this section, we study various topological properties of the maps $\s,\pp$ defined in the previous section, and explore the structure of the Balmer spectrum $\spc\dm(R)$ as a topological space.
We begin with recalling the definition of the topology which the Balmer spectrum possesses.

\begin{dfn}
Let $\T$ be an essentially small tensor triangulated category.
\begin{enumerate}[(1)]
\item
For an object $X\in\T$ the {\em Balmer support} of $X$, denoted by $\ssupp X$, is defined as the set of prime thick $\otimes$-ideals of $\T$ not containing $X$.
We set $\u(X)=(\ssupp X)^\complement=\spc\T\setminus\ssupp X$.
\item
The set $\spc\T$ is a topological space with open basis $\{\u(X)\mid X\in\dm(R)\}$; see \cite[Definition 2.1]{B}.
\end{enumerate}
Therefore, $\spc\dm(R)$ is a topological space.
We regard $\ts\dm(R)$ as a subspace of $\spc\dm(R)$ by the relative topology.
\end{dfn}

We first consider a direct sum decomposition of the Balmer spectrum.

\begin{prop}\label{dsd}
There is a direct sum decomposition of sets
$$
\textstyle\spc\dm(R)=\coprod_{\p\in\spec R}\pp^{-1}(\p),
$$
where $\pp^{-1}(\p):=\{\P\in\spc\dm(R)\mid\pp(\P)=\p\}=\{\P\in\spc\dm(R)\mid\supp\P=\{\q\in\spec R\mid\q\nsubseteq\p\}\}$
\end{prop}

\begin{proof}
Theorem \ref{s} says that the map $\pp$ is surjective.
Using this, we easily get the direct sum decomposition.
Applying Theorem \ref{s}, Corollary \ref{ppsupp} and Proposition \ref{two}, we observe that for any $\p\in\spec R$ and $\P\in\spc\dm(R)$ one has $\pp(\P)=\p$ if and only if $\supp\P=\{\q\in\spec R\mid\q\nsubseteq\p\}$.
\end{proof}

Next we investigate the dimension of the Balmer spectrum.
The {\em (Krull) dimension} of a topological space $X$, denoted by $\dim X$, is defined to be the supremum of integers $n\ge0$ such that there exists a chain $Z_0\subsetneq Z_1\subsetneq \cdots\subsetneq Z_n$ of nonempty irreducible closed subsets of $X$.
(Recall that a subset of $X$ is called {\em irreducible} if it cannot be written as a union of two nonempty proper closed subsets.)

\begin{prop}\label{pr}
\begin{enumerate}[\rm(1)]
\item
Let $\T$ be an essentially small $\otimes$-triangulated category.
The dimension of $\spc\T$ is equal to the supremum of integers $n\ge0$ such that there is a chain $\P_0\subsetneq \P_1\subsetneq \cdots\subsetneq\P_n$ in $\spc\T$.
\item
There is an inequality
$$
\dim(\spc\dm(R))\ge\dim R.
$$
\end{enumerate}
\end{prop}

\begin{proof}
Applying \cite[Propositions 2.9 and 2.18]{B} shows (1), while (2) follows from (1) and Theorem \ref{s}.
\end{proof}

\begin{rem}
We will see that the inequality in Proposition \ref{pr}(2) sometimes becomes equality, and sometimes becomes strict inequality.
See Corollaries \ref{minmxcor2}, \ref{mugen} and Theorem \ref{nonnoeth}.
\end{rem}

Let $\P,\Q$ be prime thick $\otimes$-ideals of $\dm(R)$.
We write $\P\sim\Q$ if $\supp\P=\supp\Q$.
Then $\sim$ defines an equivalence relation on $\spc\dm(R)$.
We denote by $\spc \dm(R)/\supp$ the quotient topological space of $\spc \dm(R)$ by the equivalence relation $\sim$, so that a subset $U$ of $\spc\dm(R)/\supp$ is open if and only if $\pi^{-1}(U)$ is open in $\spc\dm(R)$, where $\pi:\spc\dm(R)\to\spc\dm(R)/\supp$ stands for the canonical surjection.
By definition, $\pi$ is a continuous map.
Denote by $\theta:\ts\dm(R)\to\spc\dm(R)$ the inclusion map, which is continuous.
Now we can state our first main result in this section.

\begin{thm}\label{kakan}
\begin{enumerate}[\rm(1)]
\item
The set $\ts\dm(R)$ is dense in $\spc\dm(R)$.
\item
The composition $\pi\theta$ is a continuous bijection.
\item
The maps $\s,\pp$ induce the bijections $\s',\widetilde\s,\pp',\widetilde\pp$ which make the diagram below commute.
$$
\xymatrix{
&& \ts\dm(R)\ar[d]^\theta\ar[rrd]^{\pp'}\\
\spec R\ar[rrd]^{\widetilde\s}\ar[rr]^\s\ar[rru]^{\s'} && \spc\dm(R)\ar[d]^\pi\ar[rr]^\pp && \spec R\\
&& \spc\dm(R)/\supp\ar[rru]^{\widetilde\pp}
}
$$
In particular, one has $\pp\s=\pp'\s'=\widetilde\pp\widetilde\s=1$.
\item
The maps $\pp,\pp',\widetilde\pp$ are continuous.
The maps $\s',\widetilde\s$ are open and closed.
\end{enumerate}
\end{thm}

\begin{proof}
First of all, recall from Corollary \ref{tp} that the image of $\s$ coincides with $\ts\dm(R)$.

(1) Let $X$ be a complex in $\dm(R)$, and suppose that $U:=\u(X)$ is nonempty.
Then $U$ contains a prime thick $\otimes$-ideal $\P$ of $\dm(R)$, and $X$ is in $\P$.
It is seen from Theorem \ref{s} that $\P$ is contained in $\s(\pp(\P))$, and hence $X$ is in $\s(\pp(\P))$.
Therefore $\s(\pp(\P))$ belongs to the intersection $U\cap\ts\dm(R)$, and we have $U\cap\ts\dm(R)\ne\emptyset$.
This shows that any nonempty open subset of $\spc\dm(R)$ meets $\ts\dm(R)$.

(2) Since $\pi$ and $\theta$ are continuous, so is $\pi\theta$.
Let $\P,\Q$ be tame prime thick $\otimes$-ideals of $\dm(R)$.
Then $\P=\s\pp(\P)$ and $\Q=\s\pp(\Q)$ by Theorem \ref{s}.
One has $\P\sim\Q$ if and only if $\pp(\P)=\pp(\Q)$ by Corollary \ref{ppsupp}, if and only if $\P=\Q$ by Theorem \ref{s} again.
This shows that the map $\pi\theta$ is well-defined and injective.
To show the surjectivity, pick a prime thick $\otimes$-ideal $\R$ of $\dm(R)$.
It is seen from Proposition \ref{si}(1) that $\R\sim\supp^{-1}\supp\R$, and the latter thick $\otimes$-ideal is tame.
Consequently, $\pi\theta$ is a bijection.

(3) Using Theorem \ref{s}, we obtain the bijection $\s'$ satisfying $\theta\s'=\s$.
Set $\widetilde\s=\pi\s$ and $\pp'=\pp\theta$.
Define the map $\widetilde\pp:\spc\dm(R)/\supp\to\spec R$ by $\widetilde\pp([\P])=\pp(\P)$ for $\P\in\spc\dm(R)$.
Corollary \ref{ppsupp} guarantees that this is well-defined, and by definition we have $\widetilde\pp\pi=\pp$.
Thus the commutative diagram in the assertion is obtained, which and Theorem \ref{s} yield $1=\pp\s=\pp'\s'=\widetilde\pp\widetilde\s$.
It follows that the map $\s'$ is bijective, and so is $\pp'$.
We have $\widetilde\s=(\pi\theta)\s'$, which is bijective by (2), and so is $\widetilde\pp$.

(4) Let $\P\in\spc\dm(R)$.
An ideal $I$ of $R$ is contained in $\pp(\P)$ if and only if $\V(I)$ is not contained in $\supp\P$, if and only if $R/I$ does not belong to $\P$ by Corollary \ref{3'}.
We obtain an equality
$$
\pp^{-1}(\V(I))=\ssupp R/I,
$$
which shows that $\pp$ is a continuous map.
Since the map $\theta$ is continuous, so is the composition $\pp'=\pp\theta$.
The equality $\pp'=\s'^{-1}$ from (3) and the continuity of $\pp'$ imply that the map $\s'$ is open and closed.

Fix an ideal $I$ of $R$.
A prime ideal $\p$ of $R$ is in $\D(I)$ if and only if $\s(\p)$ is in $\u(R/I)$.
This shows $\s(\D(I))=\u(R/I)\cap\ts\dm(R)$, and we get $\pi^{-1} \widetilde{\s}(\D(I))= \pi^{-1} \pi \s(\D(I))= \pi^{-1}\pi (\u(R/I) \cap \ts \dm(R))$.
Let $\P\in\spc\dm(R)$ and $\Q\in\ts\dm(R)$.
One has $\pi(\P)=\pi(\Q)$ if and only if $\supp\P=\supp\Q$, if and only if $\supp^{-1}\supp\P=\Q$ since $\supp^{-1}\supp\Q=\Q$ by Proposition \ref{si}.
Hence $\P$ is in $\pi^{-1}\pi (\u(R/I) \cap \ts \dm(R))$ if and only if $\supp^{-1}\supp\P$ contains $R/I$ (note here that $\supp^{-1}\supp\P$ is in $\ts\dm(R)$ by Theorem \ref{s}), if and only if $\supp\P$ contains $\V(I)$, if and only if $R/I$ belongs to $\P$ by Corollary \ref{3'}.
Thus we obtain $\pi^{-1}\widetilde\s(\D(I))=\u(R/I)$, which shows that $\widetilde\s(\D(I))$ is an open subset of $\spc\dm(R)/\supp$.
Therefore $\widetilde\s$ is an open map.
This map is also closed since it is bijective.
Combining the equality $\widetilde\pp=\widetilde\s^{-1}$ from (3) and the openness of $\widetilde\s$, we observe that $\widetilde\pp$ is a continuous map.
\end{proof}

The assertions of the above theorem naturally lead us to ask when the maps in the diagram in the theorem are homeomorphisms.
We start by establishing a lemma.

\begin{lem}\label{finsp}
The following are equivalent.
\begin{enumerate}[\rm(1)]
\item\label{f4}
The set $\spec R$ is finite.
\item\label{f2}
There are only finitely many specialization-closed subsets of $\spec R$.
\item\label{f3}
There are only finitely many closed subsets of $\spec R$.
\item\label{f1}
Every specialization-closed subset of $\spec R$ is closed.
\end{enumerate}
\end{lem}

\begin{proof}
\eqref{f4} $\Rightarrow$ \eqref{f2}:
If $\spec R$ is finite, then there are only finitely many subsets of $\spec R$.

\eqref{f2} $\Rightarrow$ \eqref{f3}:
This implication follows from the fact that any closed subset is specialization-closed.

\eqref{f3} $\Rightarrow$ \eqref{f1}:
Every specialization-closed subset is a union of closed subsets.
This is a finite union by assumption, and hence it is closed.

\eqref{f1} $\Rightarrow$ \eqref{f4}:
Since $\Max R$ is specialization-closed, it is closed by our assumption.
Hence $\Max R$ possesses only finitely many minimal elements with respect to the inclusion relation, which means that it is a finite set.
Therefore the ring $R$ is semilocal.
In particular, it has finite Krull dimension, say $d$.

Suppose that $R$ possesses infinitely many prime ideals.
Then there exists an integer $0 \le n \le d$ such that the set $S$ of prime ideals of $R$ with height $n$ is infinite.
Then the specialization-closed subset $W=\bigcup_{\p \in S}\V(\p)$ is not closed because $S$ consists of the minimal elements of $W$, which is an infinite set.
This provides a contradiction, and consequently, $R$ has only finitely many prime ideals.
\end{proof}

Now we can prove the following theorem, which answers the question stated just before the lemma.

\begin{thm}\label{homeos}
Consider the following seven conditions.\\
{\rm(1)} $\s$ is continuous.\quad
{\rm(2)} $\s'$ is homeomorphic.\quad
{\rm(3)} $\pp'$ is homeomorphic.\quad
{\rm(4)} $\widetilde\s$ is homeomorphic.\\
{\rm(5)} $\widetilde\pp$ is homeomorphic.\quad
{\rm(6)} $\pi\theta$ is homeomorphic.\quad
{\rm(7)} $\spec R$ is finite.\\
Then the following implications hold:
$$
\xymatrix{
(1)\ar@{<=>}[r] & (2)\ar@{<=>}[r] & (3)\ar@{<=>}[r] & (5+6)\ar@{<=>}[r]\ar@{=>}[ld]\ar@{=>}[rd] & (7)\\
& (4)\ar@{<=>}[r] & (5) && (6)
}
$$
\end{thm}

\begin{proof}
In this proof we tacitly use Theorem \ref{kakan}.

(2) $\Leftrightarrow$ (3):
Note that $\s'$ and $\pp'$ are mutually inverse bijections.
The equivalence follows from this.

(4) $\Leftrightarrow$ (5):
As $\widetilde\s,\widetilde\pp$ are mutually inverse bijections, we have the equivalence.

(7) $\Rightarrow$ (2):
For each $X\in\dm(R)$ we have $\s'^{-1}(\ssupp X\cap\ts\dm(R))=\{\p\in\spec R\mid\s(\p)\in\ssupp X\}=\supp X$.
As $\supp X$ is specialization-closed, it is closed by Lemma \ref{finsp}.
Hence the map $\s'$ is continuous.

(2) $\Rightarrow$ (1):
This follows from the fact that $\s$ is the composition of the continuous maps $\s'$ and $\theta$.

(1) $\Rightarrow$ (7):
It is easy to observe that for any complex $X \in \dm(R)$ one has
\begin{equation}\label{ssss}
\s^{-1}(\ssupp X)=\supp X.
\end{equation}
Since $\s$ is continuous, $\supp X$ is closed in $\spec R$ for all $X \in \dm(R)$ by \eqref{ssss}.
Suppose that $\spec R$ is an infinite set. 
Then by Lemma \ref{finsp} there is a non-closed specialization-closed subset $W$ of $\spec R$.
There are infinitely many minimal elements of $W$ with respect to the inclusion relation, and we can choose countably many pairwise distinct minimal elements $\p_1, \p_2,\p_3, \dots$ of $W$.
Consider the complex $X=\bigoplus_{i=1}^\infty(R/\p_i)[i]\in\dm(R)$.
Then $\supp X=\bigcup_{i=1}^\infty\V(\p_i)$ is not closed since it has infinitely many minimal elements.
This contradiction shows that $\spec R$ is a finite set.

(2) $\Rightarrow$ (4+6):
Since $\pi,\theta,\s'$ are all continuous, so is $\widetilde{\s}=\pi\theta\s'$.
Combining this with the fact that $\widetilde\s$ is bijective and open, we see that $\widetilde\s$ is a homeomorphism.
As $\s'$ is homeomorphic, so is $\pi \theta = \widetilde{\s}\s'^{-1}$.

(4+6) $\Rightarrow$ (2):
We have $\s'=(\pi \theta)^{-1} \widetilde\s$.
Since $\pi \theta$ and $\widetilde\s$ are homeomorphisms, so is $\s'$.
\end{proof}

Next we consider the maximal and minimal elements of $\spc\dm(R)$ with respect to the inclusion relation.

\begin{dfn}
Let $\T$ be an essentially small tensor triangulated category.
\begin{enumerate}[(1)]
\item
A thick $\otimes$-ideal $\M$ of $\T$ is said to be {\em maximal} if $\M\ne\T$ and there is no thick $\otimes$-ideal $\X$ of $\T$ with $\M\subsetneq\X\subsetneq\T$.
We denote the set of maximal thick $\otimes$-ideals of $\T$ by $\mx\T$.
According to \cite[Proposition 2.3(c)]{B}, any maximal thick $\otimes$-ideal is prime, or in other words, it holds that $\mx\T\subseteq\spc\T$.
\item
A prime thick $\otimes$-ideal $\P$ of $\T$ is said to be {\em minimal} if it is minimal in $\spc\T$ with respect to the inclusion relation.
We denote the set of minimal prime thick $\otimes$-ideals of $\T$ by $\mn\T$.
\end{enumerate}
\end{dfn}

By Proposition \ref{dsd} the Balmer spectrum of $\dm(R)$ is decomposed into the fibers by $\pp:\spc\dm(R)\to\spec R$ as a set.
Concerning the fibers of maximal ideals and minimal primes of $R$, we have the following.

\begin{prop}
Let $\m\in\Max R$ and $\p\in\Min R$.
Then
$$
\min\pp^{-1}(\m)\subseteq\mn\dm(R),\qquad
\max\pp^{-1}(\p)\subseteq\mx\dm(R).
$$
\end{prop}

\begin{proof}
Take $\P\in\min\pp^{-1}(\m)$, and let $\Q$ be a prime thick $\otimes$-ideal contained in $\P$.
Then $\m=\pp(\P) \subseteq \pp(\Q)$ by Theorem \ref{s}.
Since $\m$ is a maximal ideal, we get $\m=\pp(\P) = \pp(\Q)$, and $\Q\in\pp^{-1}(\m)$.
The minimality of $\P$ implies $\P=\Q$.
Thus the first inclusion follows.
The second inclusion is obtained similarly.
\end{proof}

To prove our next theorem, we establish a lemma and a proposition.
Recall that a topological space is called {\em $\t_1$-space} if every one-point subset is closed.

\begin{lem}\label{gent}
\begin{enumerate}[\rm(1)]
\item
The subspaces $\Max R,\,\Min R$ of $\spec R$ are $\t_1$-spaces, so every finite subset is closed.
\item
Let $\T$ be an essentially small $\otimes$-triangulated category.
The subspaces $\mx\T,\,\mn\T$ of $\spc\T$ are $\t_1$-spaces, so every finite subset is closed.
\end{enumerate}
\end{lem}

\begin{proof}
(1) Let $A$ be either $\Max R$ or $\Min R$.
For each $\p\in A$ the closure of $\{\p\}$ in $A$ is $\V(\p)\cap A$, which coincides with $\{\p\}$.
Hence $A$ is a $\t_1$-space.

(2) Let $B$ be either $\mx\T$ or $\mn\T$.
For each $\P\in B$ the closure of $\{\P\}$ in $B$ is $\{\Q\in B\mid\Q\subseteq\P\}$ by \cite[Proposition 2.9]{B}, which coincides with $\{\P\}$.
Hence $B$ is a $\t_1$-space.
\end{proof}

\begin{prop}\label{whole}
For each complex $X\in\dm(R)$ it holds that
$$
\supp X=\spec R\ \Leftrightarrow\ \tthick X=\dm(R)\ \Leftrightarrow\ \ssupp X=\spc\dm(R).
$$
\end{prop}

\begin{proof}
The second equivalence follows from \cite[Corollary 2.5]{B}.
Let us prove the first equivalence.
Proposition \ref{si}(2) implies $\supp X=\supp(\tthick X)$, which shows $(\Leftarrow)$.
As for $(\Rightarrow)$, for every $M\in\dm(R)$ we have $\V(\ann M)\subseteq\spec R=\supp X$, by which and Proposition \ref{key} we get $M\in\tthick X$.
\end{proof}

Now we can prove the following theorem.
This especially says that $\dm(R)$ is ``semilocal" in the sense that $\dm(R)$ admits only a finite number of maximal thick tensor ideals.
If $R$ is an integral domain, then $\dm(R)$ is ``local" in the sense that $\dm(R)$ has a unique maximal thick tensor ideal.

\begin{thm}\label{minmx}
The restriction of $\s$ to $\Min R$ induces a homeomorphism
$$
\s|_{\Min R}:\Min R\xrightarrow{\cong}\mx\dm(R).
$$
\end{thm}

\begin{proof}
Let us show that there is an equality
\begin{equation}\label{eq}
\mx\dm(R)=\{\s(\p)\mid\p\in\Min R\}.
\end{equation}
Let $\p_1,\dots,\p_n$ be the minimal prime ideals of $R$.

Let $\M$ be a maximal thick $\otimes$-ideal of $\dm(R)$.
Suppose that $\M$ is not contained in $\s(\p_i)$ for any $1\le i\le n$.
Then for each $i$ we find an object $M_i\in\M$ such that $(M_i)_{\p_i}$ is nonzero.
Set $M=M_1\oplus\cdots\oplus M_n$.
This object belongs to $\M$, and $M_{\p_i}$ is nonzero for all $1\le i\le n$.
Hence $\supp M$ contains all the $\p_i$, and we get $\supp M=\spec R$ because $\supp M$ is specialization-closed.
Proposition \ref{whole} yields $\tthick M=\dm(R)$, and hence we have $\M=\dm(R)$, which contradicts the definition of a maximal thick $\otimes$-ideal.
Thus, $\M$ is contained in $\s(\p_l)$ for some $1\le l\le n$.
The maximality of $\M$ implies that $\M=\s(\p_l)$.

Fix an integer $1\le i\le n$.
By \cite[Proposition 2.3(b)]{B} there exists a maximal thick $\otimes$-ideal $\M_i$ of $\dm(R)$ that contains $\s(\p_i)$.
Applying the above argument to $\M_i$, we see that $\M_i$ coincides with $\s(\p_j)$ for some $1\le j\le n$.
Hence $\s(\p_i)$ is contained in $\s(\p_j)$, and Theorem \ref{s} shows that $\p_i$ contains $\p_j$.
The fact that $\p_i,\p_j$ are minimal prime ideals of $R$ forces us to have $i=j$.
Therefore we obtain $\M_i=\s(\p_i)$, which especially says that $\s(\p_i)$ is a maximal thick $\otimes$-ideal of $\dm(R)$.

Thus, we get the equality \eqref{eq}.
This shows that the restriction of the map $\s:\spec R\to\spc\dm(R)$ to $\Min R$ gives rise to a surjection $\Min R\to\mx\dm(R)$.
As Theorem \ref{s} says that $\s$ is an injection, the map $\s|_{\Min R}$ is a bijection.
By Lemma \ref{gent} we see that $\s|_{\Min R}$ is a homeomorphism.
\end{proof}

Theorem \ref{minmx} yields the following result concerning the structure of the Balmer spectrum of $\dm(R)$.

\begin{cor}\label{minmxcor}
\begin{enumerate}[\rm(1)]
\item
There are equalities
$$
\textstyle
\spc\dm(R)=\bigcup_{\p\in\spec R}\overline{\{\s(\p)\}}=\bigcup_{\p\in\Min R}\overline{\{\s(\p)\}}.
$$
\item
The topological space $\spc\dm(R)$ is irreducible if and only if so is $\spec R$.
\end{enumerate}
\end{cor}

\begin{proof}
(1) The inclusions $\spc\dm(R)\supseteq\bigcup_{\p\in\spec R}\overline{\{\s(\p)\}}\supseteq\bigcup_{\p\in\Min R}\overline{\{\s(\p)\}}$ clearly hold.
Pick a prime thick $\otimes$-ideal $\P$ of $\dm(R)$.
Then one finds a maximal thick $\otimes$-ideal $\M$ containing $\P$ by \cite[Proposition 2.3(b)]{B}.
Theorem \ref{minmx} implies that $\M=\s(\p)$ for some minimal prime ideal $\p$ of $R$, and it follows from \cite[Proposition 2.9]{B} that $\P$ belongs to the closure $\overline{\{\s(\p)\}}$.

(2) First of all, $\spc\dm(R)$ is irreducible if and only if $\spc\dm(R)=\overline{\{\P\}}$ for some $\P\in\spc\dm(R)$.
In fact, the ``if" part is obvious, while the ``only if" part follows from \cite[Proposition 2.18]{B}.
By \cite[Proposition 2.9]{B}, the set $\overline{\{\P\}}$ consists of the prime thick $\otimes$-ideals contained in $\P$.
Hence $\spc\dm(R)=\overline{\{\P\}}$ for some $\P\in\spc\dm(R)$ if and only if $\dm(R)$ has a unique maximal element, which is equivalent to $\spec R$ having a unique minimal element by Theorem \ref{minmx}.
This is equivalent to saying that $\spec R$ is irreducible.
\end{proof}

The following theorem is opposite to Theorem \ref{minmx}.
The third assertion says that if $R$ is local, then $\dm(R)$ is an ``integral domain" in the sense that $\zero$ is a (unique) minimal prime thick tensor ideal of $\dm(R)$.

\begin{thm}\label{sm}
\begin{enumerate}[\rm(1)]
\item
For every maximal ideal $\m$ of $R$, the subcategory $\s(\m)$ is a minimal prime thick $\otimes$-ideal of $\dm(R)$, or in other words, the restriction of $\s$ to $\Max R$ induces an injection
\begin{equation}\label{mn}
\s|_{\Max R}:\Max R\hookrightarrow \mn\dm(R).
\end{equation}
\item
The ring $R$ is semilocal if and only if $\dm(R)$ has only finitely many minimal prime thick $\otimes$-ideals.
When this is the case, the map \eqref{mn} is a homeomorphism.
\item
If $(R,\m)$ is a local ring, then $\s(\m)=\zero$ is a unique minimal prime thick $\otimes$-ideal of $\dm(R)$.
\end{enumerate}
\end{thm}

\begin{proof}
(1) Let $\P$ be a prime thick $\otimes$-ideal of $\dm(R)$ contained in $\s(\m)$.
Take any object $X\in\s(\m)$.
Then $\supp(X\ltensor_RR/\m)=\supp X\cap\{\m\}=\emptyset$ by Lemma \ref{si0}(4).
Remark \ref{r} shows $X\ltensor_RR/\m=0$, which belongs to $\P$.
As $\P$ is prime, either $X$ or $R/\m$ is in $\P$.
Since $\s(\m)$ does not contain $R/\m$, neither does $\P$.
Therefore $X$ must be in $\P$, and we obtain $\P=\s(\m)$.
This shows that the prime thick $\otimes$-ideal $\s(\m)$ is minimal.
Thus, $\s$ induces a map $\Max R\to\mn\dm(R)$.
The injectivity follows from Theorem \ref{s}.

(2) The first assertion of the theorem implies the ``if" part, and it suffices to show that if $R$ is semilocal, then \eqref{mn} is a homeomorphism.
Let us first prove the surjectivity of the map \eqref{mn}.
Take a minimal prime thick $\otimes$-ideal $\P$ of $\dm(R)$.
What we want is that there is a maximal ideal $\m$ of $R$ such that $\P=\s(\m)$.

Suppose that $\P$ does not contain $\s(\m)$ for all $\m\in\Max R$.
Write $\Max R=\{\m_1,\dots,\m_t\}$.
For each $1\le i\le t$ we find an object $X_i$ of $\dm(R)$ with $X_i\in\s(\m_i)$ and $X_i\notin\P$.
Setting $X=X_1\ltensor_R\cdots\ltensor_RX_t$, for each $i$ we have $X_{\m_i}=X_1\ltensor_R\cdots\ltensor_R(X_i)_{\m_i}\ltensor_R\cdots\ltensor_RX_t=0$.
Hence $X_\m=0$ for all $\m\in\Max R$, which implies $X_\p=0$ for all $\p\in\spec R$.
This means that $\supp X$ is empty, and Remark \ref{r} yields $X=0$.
In particular, $X=X_1\ltensor_R\cdots\ltensor_RX_t$ is in $\P$.
As $\P$ is prime, it contains some $X_u$, which is a contradiction.

Consequently, $\P$ must contain $\s(\m)$ for some $\m\in\Max R$.
The minimality of $\P$ shows that $\P=\s(\m)$.
We conclude that the map \eqref{mn} is surjective, whence it is bijective.
Since the set $\Max R$ is finite, so is $\mn\dm(R)$.
Applying Lemma \ref{gent}, we observe that \eqref{mn} is a homeomorphism.

(3) As $R$ is a local ring with maximal ideal $\m$, the equality $\s(\m)=\zero$ holds, which especially says that $\zero$ is a prime thick $\otimes$-ideal of $\dm(R)$ by Proposition \ref{two}.
If $\P$ is a minimal prime thick $\otimes$-ideal, then $\P$ contains $\zero$, and the minimality of $\P$ implies $\P=\zero$.
Thus $\zero$ is a unique minimal prime thick $\otimes$-ideal.
\end{proof}

\begin{ques}
Is the map \eqref{mn} bijective even if $R$ is not semilocal?
\end{ques}

Recall that a topological space $X$ is called {\em noetherian} if any descending chain of closed subsets of $X$ stabilizes.
Applying the above two theorems to the artinian case gives rise to the following result.

\begin{cor}\label{minmxcor2}
Let $R$ be an artinian ring.
Then the map $\s:\spec R\to\spc\dm(R)$ is a homeomorphism.
Hence the topological space $\spc\dm(R)$ is noetherian, and one has $\dim(\spc\dm(R))=\dim R=0<\infty$.
\end{cor}

\begin{proof}
Since $\spec R=\Min R=\Max R$, the assertion is deduced from Theorems \ref{minmx} and \ref{sm}(2).
\end{proof}

From here we consider when $\dm(R)$ is a local tensor triangulated category.
Let us recall the definition.

\begin{dfn}
(1) A topological space $X$ is called {\em local} if for any open cover $X=\bigcup_{i\in I}U_i$ of $X$ there exists $t\in I$ such that $X=U_t$.
In particular, any local topological space is quasi-compact.\\
(2) An essentially small tensor triangulated category $\T$ is called {\em local} if $\spc\T$ is a local topological space.
\end{dfn}

\begin{rem}
It is clear that the topological space $\spec R$ is local if and only if the ring $R$ is local.
\end{rem}

For an essentially small $\otimes$-triangulated category $\T$ the following are equivalent (\cite[Proposition 4.2]{B3}).
\begin{enumerate}[\quad(i)]
\item
$\T$ is local.
\item
$\T$ has a unique minimal prime thick $\otimes$-ideal.
\item
The radical thick $\otimes$-ideal $\sqrt\zero$ of $\T$ is prime.
\end{enumerate}
If moreover $\T$ is {\em rigid}, then the above three conditions are equivalent to:
\begin{enumerate}[\quad(i)]
\setcounter{enumi}{3}
\item
The zero subcategory $\zero$ of $\T$ is a prime thick $\otimes$-ideal.
\end{enumerate}
Also, it follows from \cite[Example 4.4]{B3} that $\kb(\proj R)$ is local if and only if so is $R$.

The following result says that the same statements hold for $\dm(R)$.
Also, we emphasize that it contains the equivalent condition (4), even though $\dm(R)$ is not rigid; see Remark \ref{rigid}.

\begin{cor}\label{4}
The following are equivalent.
\begin{enumerate}[\rm(1)]
\item
The $\otimes$-triangulated category $\dm(R)$ is local.
\item
There is a unique minimal thick $\otimes$-ideal of $\dm(R)$.
\item
The radical thick $\otimes$-ideal $\sqrt\zero$ of $\dm(R)$ is prime.
\item
The zero subcategory $\zero$ of $\dm(R)$ is a prime thick $\otimes$-ideal.
\item
The ring $R$ is local.
\end{enumerate}
\end{cor}

\begin{proof}
Combining Theorem \ref{sm} with the result given just before the corollary, we observe that (1) $\Leftrightarrow$ (2) $\Leftrightarrow$ (3) $\Leftrightarrow$ (5) $\Rightarrow$ (4) hold.
If $\zero$ is prime, then it is easy to see that $\sqrt\zero=\zero$.
Thus (4) implies (3).
\end{proof}

One can indeed obtain more precise information on the structure of $\spc\dm(R)$ than Corollary \ref{4}:

\begin{prop}
One has
$$
\spc\dm(R)=
\begin{cases}
\u(R/\m)\sqcup\{\zero\} & \text{if $(R,\m)$ is local},\\
\bigcup_{\m\in\Max R}\u(R/\m) & \text{if $R$ is non-local}.
\end{cases}
$$
If $\m,\n$ are distinct maximal ideals of $R$, then $\spc\dm(R)=\u(R/\m)\cup\u(R/\n)$.
\end{prop}

\begin{proof}
Suppose that $(R,\m,k)$ is a local ring.
Corollary \ref{4} implies that $\zero$ is prime, and $\spc\dm(R)$ contains $\u(k)\cup\{\zero\}$.
Let $\P$ be a nonzero prime thick $\otimes$-ideal of $\dm(R)$.
Then there exists an object $X\ne0$ in $\P$.
By Remark \ref{r} the support of $X$ is nonempty and specialization-closed, whence contains $\m$.
Using Lemma \ref{si0}(4), we have $\supp(X\ltensor_Rk)=\supp X\cap\supp k=\{\m\}\ne\emptyset$.
Hence $X\ltensor_Rk$ is nonzero by Remark \ref{r} again.
Since $X\ltensor_Rk$ is isomorphic to a direct sum of shifts of $k$-vector spaces, it contains $k[n]$ as a direct summand for some $n\in\Z$.
As $X\ltensor_Rk$ is in $\P$, so is $k$.
Therefore $\P$ is in $\u(k)$, and we obtain $\spc\dm(R)=\u(k)\cup\{\zero\}$.
It is obvious that $\u(k)\cap\{\zero\}=\emptyset$.
We conclude that $\spc\dm(R)=\u(k)\sqcup\{\zero\}$.

Now, let $\m$ and $\n$ be distinct maximal ideals of $R$.
Applying Lemma \ref{si0}(4), we have $\supp(R/\m\ltensor_RR/\n)=\{\m\}\cap\{\n\}=\emptyset$, and hence $R/\m\ltensor_RR/\n=0$ by Remark \ref{r}.
Therefore we obtain $\u(R/\m)\cup\u(R/\n)=\u(R/\m\ltensor_RR/\n)=\u(0)=\spc\dm(R)$, where the first equality follows from \cite[Lemma 2.6(e)]{B}.
Thus the last assertion of the proposition follows, which shows the first assertion in the non-local case.
\end{proof}

So far we have investigated the irreducible and local properties of $\spc\dm(R)$.
In general, there is no implication between the local and irreducible properties of $\spc\dm(R)$:

\begin{rem}
If $R$ is a local ring possessing at least two minimal prime ideals, then $\spc \dm(R)$ is local by Corollary \ref{4}, but not irreducible by Corollary \ref{minmxcor}(2).
Similarly, if $R$ is a nonlocal ring with unique minimal prime ideal, then $\spc \dm(R)$ is irreducible but not local.
\end{rem}

\section{Relationships among thick tensor ideals and specialization-closed subsets}\label{various}

This section compares compact, tame and radical thick tensor ideals of $\dm(R)$, relating them to specialization closed subsets of $\spec R$, $\ts\dm(R)$ and Thomason subsets of $\spc\dm(R)$.
We start with some notation.

\begin{dfn}
(1) Let $\T$ be a tensor triangulated category.
Let $\PP$ be a property of thick $\otimes$-ideals of $\T$.
For a subcategory $\X$ of $\C$ we denote by $\X^\PP$ (resp. $\X_\PP$) the {\em $\PP$-closure} (resp. {\em $\PP$-interior}) of $\X$, that is to say, the smallest (resp. largest) thick $\otimes$-ideal of $\T$ which contains (resp. which is contained in) $\X$ and satisfies the property $\PP$.
We define this only when it exists.\\
(2) Let $X$ be a topological space.
Let $\PP$ be a property of subsets of $X$.
For a subset $A$ of $X$ we denote by $A^\PP$ (resp. $A_\PP$) the {\em $\PP$-closure} (resp. {\em $\PP$-interior}) of $A$, namely, the smallest (resp. largest) subset of $X$ that contains (resp. that is contained in) $A$ and satisfies $\PP$.
We define this only when it exists.
\end{dfn}

Here is a list of properties $\PP$ as in the above definition which we consider:
$$
\rrad=\text{radical},\qquad
\tame=\text{tame},\qquad
\cpt=\text{compact},\qquad
\spcl=\text{specialization-closed}.
$$

\begin{nota}
We denote by $\Rad$ (resp. $\Tame$, $\Cpt$) the set of radical (resp. tame, compact) thick $\otimes$-ideals of $\dm(R)$.
Also, $\Spcl(\spec)$ (resp. $\Spcl(\ts))$ stands for the set of specialization-closed subsets of the topological space $\spec R$ (resp. $\ts\dm(R)$).
\end{nota}

Our first purpose in this section is to give a certain commutative diagram of bijections.
To achieve this purpose, we prepare several propositions.
We state here two propositions.
The first one is shown by using Proposition \ref{si}, while the second one is nothing but Theorem \ref{main2}.

\begin{prop}\label{tss}
There is a one-to-one correspondence $\supp:\Tame\rightleftarrows\Spcl(\spec):\supp^{-1}$.
\end{prop}

\begin{prop}\label{cs}
There is a one-to-one correspondence $\supp:\Cpt\rightleftarrows\Spcl(\spec):\langle\rangle$.
\end{prop}

\begin{nota}
For an object $M$ of $\dm(R)$ we denote by $\sp M$ the set of tame prime thick $\otimes$-ideals of $\dm(R)$ not containing $M$, i.e., $\sp M=\ssupp M\cap\ts\dm(R)$.
For a subcategory $\X$ of $\dm(R)$ we set $\sp\X=\bigcup_{M\in\X}\sp M$.
For a subset $A$ of $\spc\dm(R)$ we denote by $\sp^{-1}A$ the subcategory of $\dm(R)$ consisting of objects $M$ such that $\sp M$ is contained in $A$.
\end{nota}

We make a lemma, whose second assertion is a variant of \cite[Lemma 4.8]{B}.

\begin{lem}\label{sp}
\begin{enumerate}[\rm(1)]
\item
For a subcategory $\X$ of $\dm(R)$, the subset $\sp\X$ of $\ts\dm(R)$ is specialization-closed.
\item
For a subset $A$ of $\ts\dm(R)$ one has $\sp^{-1}A=\bigcap_{\P\in A^\complement}\P$, where $A^\complement=\ts\dm(R)\setminus A$.
\item
Let $\{\X_\lambda\}_{\lambda\in\Lambda}$ be a collection of tame thick $\otimes$-ideals of $\dm(R)$.
Then the intersection $\bigcap_{\lambda\in\Lambda}\X_\lambda$ is also a tame thick $\otimes$-ideal of $\dm(R)$.
\end{enumerate}
\end{lem}

\begin{proof}
(1) We have $\sp\X=\bigcup_{X\in\X}\sp X$, and $\sp X=\ssupp X\cap\ts\dm(R)$ is closed in $\ts\dm(R)$ since $\ssupp X$ is closed in $\spc\dm(R)$.
Therefore $\sp\X$ is specialization-closed in $\ts\dm(R)$.

(2) An object $X$ of $\dm(R)$ belongs to $\sp^{-1}A$ if and only if $\sp X$ is contained in $A$, if and only if $A^\complement$ is contained in $(\sp X)^\complement=\{\P\in\ts\dm(R)\mid X\in\P\}$, if and only if $X$ belongs to $\bigcap_{\P\in A^\complement}\P$.

(3) For each $\lambda\in\Lambda$ there is a subset $S_\lambda$ of $\spec R$ such that $\X_\lambda=\supp^{-1}S_\lambda$.
Then it is clear that the equality $\bigcap_{\lambda\in\Lambda}\X_\lambda=\supp^{-1}(\bigcap_{\lambda\in\Lambda}S_\lambda)$ holds, which shows the assertion.
\end{proof}

Using the above lemma, we obtain a bijection induced by $\sp$.

\begin{prop}\label{ts2}
There is a one-to-one correspondence $\sp:\Tame\rightleftarrows\Spcl(\ts):\sp^{-1}$.
\end{prop}

\begin{proof}
Fix a tame thick $\otimes$-ideal $\X$ of $\dm(R)$ and a specialization-closed subset $U$ of $\ts\dm(R)$.
Lemma \ref{sp}(1) implies that $\sp\X$ is specialization-closed in $\ts\dm(R)$, that is, $\sp\X\in\Spcl(\ts)$.
Lemma \ref{sp}(2) implies that $\sp^{-1}U=\bigcap_{\P\in U^\complement}\P$, and each $\P\in U^\complement$ is a tame thick $\otimes$-ideal of $\dm(R)$.
Hence $\sp^{-1}U$ is also a tame thick $\otimes$-ideal of $\dm(R)$ by Lemma \ref{sp}(3), namely, $\sp^{-1}U\in\Tame$.

Let us show that $\sp(\sp^{-1}U)=U$.
It is evident that $\sp(\sp^{-1}U)$ is contained in $U$.
Pick any $\P\in U$.
Corollary \ref{tp} says $\P=\s(\p)$ for some prime ideal $\p$ of $R$.
Since $U$ is specialization-closed in $\ts\dm(R)$, the closure $C$ of $\s(\p)$ in $\ts\dm(R)$ is contained in $U$.
Using \cite[Proposition 2.9]{B}, we see that $C$ consists of the prime thick $\otimes$-ideals of the form $\s(\q)$, where $\q$ is a prime ideal of $R$ with $\s(\q)\subseteq\s(\p)$.
In view of Theorem \ref{s}, we have $C=\{\s(\q)\mid\q\in\V(\p)\}$, and it is easy to observe that this coincides with $\sp(R/\p)$.
Hence $R/\p$ is in $\sp^{-1}U$, and $\P=\s(\p)$ belongs to $\sp(\sp^{-1}U)$.
Now we obtain $\sp(\sp^{-1}U)=U$.

It remains to prove that $\sp^{-1}(\sp\X)=\X$.
We have $\sp^{-1}(\sp\X)=\bigcap_{\P\in(\sp\X)^\complement}\P$ by Lemma \ref{sp}(2).
Fix a prime thick $\otimes$-ideal $\P$ of $\dm(R)$.
Then $\P$ is in $(\sp\X)^\complement$ if and only if $\P$ is tame and $\P$ is not in $\sp\X$.
The former statement is equivalent to saying that $\P=\s(\p)$ for some $\p\in\spec R$ by Corollary \ref{tp}, while the latter is equivalent to saying that $\X$ is contained in $\P$.
Hence $\sp^{-1}(\sp\X)=\bigcap_{\p\in\spec R,\,\X\subseteq\s(\p)}\s(\p)$.
Thus an object $Y$ of $\dm(R)$ belongs to $\sp^{-1}(\sp\X)$ if and only if $Y$ belongs to $\s(\p)$ for all $\p\in\spec R$ with $\X\subseteq\s(\p)$, if and only if $Y_\p=0$ for all $\p\in\spec R$ with $\X_\p=0$, if and only if $\supp Y$ is contained in $\supp\X$, if and only if $Y$ belongs to $\X$ by Proposition \ref{tss}.
\end{proof}

Here we consider describing $\rrad$-closures, $\tame$-closures and $\cpt$-interiors, and their supports.

\begin{lem}\label{clint}
Let $\X$ be a subcategory of $\dm(R)$, and let $\Y$ be a thick $\otimes$-ideal of $\dm(R)$.
One has:
\begin{enumerate}[\rm(1)]
\item
$(\tthick\X)_\cpt=\langle\supp\X\rangle,\quad\X^\rrad=\sqrt{\tthick\X},\quad\X^\tame=\supp^{-1}\supp\X$,
\item
$\Y_\cpt\subseteq\Y\subseteq\Y^\rrad\subseteq\Y^\tame,\quad\supp(\Y_\cpt)=\supp\Y=\supp(\Y^\rrad)=\supp(\Y^\tame)$,
\end{enumerate}
\end{lem}

\begin{proof}
(1) It follows from \cite[Lemma 4.2]{B} (resp. Remark \ref{tame}) that $\sqrt{\tthick\X}$ (resp. $\supp^{-1}\supp\X$) is a thick $\otimes$-ideal of $\dm(R)$.
It is clear that $\sqrt{\tthick\X}$ (resp. $\supp^{-1}\supp\X$) is radical (resp. tame) and contains $\X$.
If $\C$ is a radical (resp. tame) thick $\otimes$-ideal of $\dm(R)$ containing $\X$, then we have $\sqrt{\tthick\X}\subseteq\sqrt{\tthick\C}=\sqrt\C=\C$ (resp. $\supp^{-1}\supp\X\subseteq\supp^{-1}\supp\C=\C$ by Proposition \ref{tss}).
Thus, we obtain the two equalities $\X^\rrad=\sqrt{\tthick\X}$ and $\X^\tame=\supp^{-1}\supp\X$.
It remains to show the equality $(\tthick\X)_\cpt=\langle\supp\X\rangle$.
Clearly, $\langle\supp\X\rangle$ is a compact thick $\otimes$-ideal of $\dm(R)$.
Applying Corollary \ref{3'}, we observe that $\langle\supp\X\rangle$ is contained in $\tthick\X$.
Let $\C$ be a compact thick $\otimes$-ideal of $\dm(R)$ contained in $\tthick\X$.
Then it follows from Proposition \ref{cs} that $\C=\langle\supp\C\rangle$, which is contained in $\langle\supp(\tthick\X)\rangle=\langle\supp\X\rangle$ by Proposition \ref{si}(2).
We now conclude $(\tthick\X)_\cpt=\langle\supp\X\rangle$.

(2) Fix a prime ideal $\p$ of $R$.
Proposition \ref{two} says that $\s(\p)$ is a prime thick $\otimes$-ideal of $\dm(R)$, whence it is radical.
Therefore $\Y_\p=0$ if and only if $(\sqrt\Y)_\p=0$.
This shows $\supp(\sqrt\Y)=\supp\Y$.
Hence $\sqrt\Y$ is contained in $\supp^{-1}\supp\Y$, meaning that $\Y^\rrad$ is contained in $\Y^\tame$ by (1).
Thus we get the inclusions $\Y_\cpt\subseteq\Y\subseteq\Y^\rrad\subseteq\Y^\tame$, which implies $\supp(\Y_\cpt)\subseteq\supp\Y\subseteq\supp(\Y^\rrad)\subseteq\supp(\Y^\tame)$.
By (1) and Proposition \ref{si} we get $\supp(\Y^\tame)=\supp\Y=\supp(\Y_\cpt)$.
The equalities in the assertion follow.
\end{proof}

The inclusion $\Y^\rrad\subseteq\Y^\tame$ in Lemma \ref{clint} in particular says:

\begin{cor}\label{rtsr}
Every tame thick $\otimes$-ideal of $\dm(R)$ is radical.
\end{cor}

We now obtain a bijection, using the above lemma.

\begin{prop}\label{ct}
There is a one-to-one correspondence $()^\tame:\Cpt\rightleftarrows\Tame:()_\cpt$.
\end{prop}

\begin{proof}
Fix a compact thick $\otimes$-ideal $\X$, and a tame thick $\otimes$-ideal $\Y$ of $\dm(R)$.
We have $(\X^\tame)_\cpt=\langle\supp(\X^\tame)\rangle=\langle\supp\X\rangle=\X$, where the first equality follows from Lemma \ref{clint}(1), the second from Lemma \ref{clint}(2), and the last from Proposition \ref{cs}.
Also, it holds that $(\Y_\cpt)^\tame=\supp^{-1}\supp(\Y_\cpt)=\supp^{-1}\supp\Y=\Y$, where the first equality follows from Lemma \ref{clint}(1), the second from Lemma \ref{clint}(2), and the last from Proposition \ref{tss}.
Thus we obtain the one-to-one correspondence in the proposition.
\end{proof}

For each subset $A$ of $\spec R$, we put $\s(A)=\{\s(\p)\mid\p\in A\}$.
For each subset $B$ of $\spc\dm(R)$, we put $\pp(B)=\{\pp(\P)\mid\P\in B\}$.
We get another bijection.

\begin{prop}\label{ss}
There is a one-to-one correspondence $\s:\Spcl(\spec)\rightleftarrows\Spcl(\ts):\pp$.
\end{prop}

\begin{proof}
First of all, applying Theorem \ref{s} and Corollary \ref{tp}, we observe that
\begin{equation}\label{pps}
\pp(\s(\p))=\p\text{ for all }\p\in\spec R
\qquad\text{and}\qquad
\s(\pp(\P))=\P\text{ for all }\P\in\ts\dm(R).
\end{equation}
Fix a specialization-closed subset $W$ of $\spec R$ and a specialization-closed subset $U$ of $\ts\dm(R)$.
It follows from \eqref{pps} that $\pp(\s(W))=W$ and $\s(\pp(U))=U$.

Pick a prime ideal $\p$ in $W$.
Let $X$ be the closure of $\{\s(\p)\}$ in $\ts\dm(R)$.
Then $X=Y\cap\ts\dm(R)$, where $Y$ is the closure of $\{\s(\p)\}$ in $\spc\dm(R)$, and hence
$$
X=\{\P\in\ts\dm(R)\mid\P\subseteq\s(\p)\}
=\{\s(\q)\mid\q\in\spec R,\ \s(\q)\subseteq\s(\p)\}
=\{\s(\q)\mid\q\in\V(\p)\}
\subseteq\s(W),
$$
where the first equality follows from \cite[Proposition 2.9]{B}, the second from Corollary \ref{tp}, and the third from Theorem \ref{s}.
The inclusion holds since $W$ is a specialization-closed subset of $\spec R$.
Therefore, $\s(W)$ is a specialization-closed subset of $\ts\dm(R)$, namely, $\s(W)\in\Spcl(\ts)$.

Pick $\P\in U$.
As $U$ is a subset of $\ts\dm(R)$, the prime thick $\otimes$-ideal $\P$ is tame.
Let $\q$ be a prime ideal of $R$ containing $\pp(\P)$.
We then get $\s(\q)\subseteq\s(\pp(\P))=\P$ by Theorem \ref{s} and \eqref{pps}, which says that $\s(\q)$ belongs to the closure of the set $\{\P\}$ in $\ts\dm(R)$ by \cite[Proposition 2.9]{B}.
The specialization-closed property of $U$ implies that $\s(\q)$ belongs to $U$.
We have $\q=\pp(\s(\q))$ by \eqref{pps}, which belongs to $\pp(U)$.
Consequently, the subset $\pp(U)$ of $\spec R$ is specialization-closed, that is, $\pp(U)\in\Spcl(\spec)$.
\end{proof}

Here we note an elementary fact on commutativity of a diagram of maps.

\begin{rem}\label{comm}
Consider the following diagram of bijections
$$
\xymatrix{
& A\ar@<1mm>[ld]^a\ar@<1mm>[rd]^{c^{-1}} &\\
B\ar@<1mm>[ru]^{a^{-1}}\ar@<1mm>[rr]^b && C\ar@<1mm>[lu]^c\ar@<1mm>[ll]^{b^{-1}}.
}
$$
One can choose infinitely many compositions of maps in the diagram, but once one of them is equal to another, this triangle with edges having any directions commutes.
To be more explicit, if $c=ba$ for instance, then the set $\{1,a,a^{-1},b,b^{-1},c,c^{-1}\}$ is closed under possible compositions.
\end{rem}

Now we can state and prove our first main result in this section.

\begin{thm}\label{d1}
There is a commutative diagram of mutually inverse bijections:
$$
\xymatrix{
&&& \Spcl(\spec)\ar@<2mm>[llld]^{\langle\rangle}\ar@<2mm>[rrrd]^\s_\cong\ar@<2mm>[dd]^{\supp^{-1}}_\cong &&& \\
\quad\qquad\Cpt\ar@<2mm>[rrru]_\cong^\supp\ar@<2mm>[rrrd]^{{()}^\tame}_\cong &&&&&& \Spcl(\ts)\ar@<2mm>[lllu]^\pp\ar@<2mm>[llld]^{\sp^{-1}}_\cong \\
&&& \Tame\ar@<2mm>[lllu]^{{()}_\cpt}\ar@<2mm>[rrru]^\sp\ar@<2mm>[uu]^\supp &&&
}
$$
\end{thm}

\begin{proof}
The five one-to-one correspondences in the diagram are shown in Propositions \ref{tss}, \ref{cs}, \ref{ts2}, \ref{ct} and \ref{ss}.
It remains to show the commutativity, and for this we take Remark \ref{comm} into account.

For a thick $\otimes$-ideal $\X$ of $\dm(R)$, we have $\supp(\X^\tame)=\supp\X$ by Lemma \ref{clint}(2), which shows that the left triangle in the diagram commutes.
It is easy to observe from Corollary \ref{tp} that
\begin{equation}\label{sss}
\text{$\sp\X=\s(\supp\X)$ for any subcategory $\X$ of $\dm(R)$}.
\end{equation}
The commutativity of the right triangle in the diagram follows from \eqref{sss}.
\end{proof}

\begin{rem}
The bijections in the diagram of Theorem \ref{d1} induce lattice structures in $\Tame$ and $\Spcl(\ts)$, so that the maps are lattice isomorphisms.
However, we do not know if there is an explicit way to define lattice structures like the one of $\Cpt$ given in Proposition \ref{lat}(2).
\end{rem}

Let $f:A\to B$ and $g:B\to A$ be maps with $gf=1$.
Then we say that $(f,g)$ is a {\em section-retraction pair}, and write $f\dashv g$.
Our next goal is to construct a certain commutative diagram of section-retraction pairs, and for this we again give several propositions.
The first one is a consequence of \cite[Theorem 4.10]{B}.

\begin{prop}\label{bal}
There is a one-to-one correspondence $\ssupp:\Rad\rightleftarrows\Thom:\ssupp^{-1}$.
\end{prop}

\begin{prop}\label{cr}
There is a section-retraction pair $()^\rrad:\Cpt\rightleftarrows\Rad:()_\cpt$.
\end{prop}

\begin{proof}
For every $\X\in\Cpt$, we have $(\X^\rrad)_\cpt=\langle\supp(\X^\rrad)\rangle=\langle\supp\X\rangle=\X_\cpt=\X$ by Lemma \ref{clint}.
\end{proof}

Let $X$ be a topological space.
A subset $T$ of $X$ is called {\em Thomason} if $T$ is a union of closed subsets of $X$ whose complements are quasi-compact.
Note that a Thomason subset is specialization-closed.

For each subset $A$ of $\spec R$, we set $\overline\s(A)=\bigcup_{\p\in A}\overline{\{\s(\p)\}}$.
For each subset $B$ of $\spc\dm(R)$, we set $\s^{-1}(B)=\{\p\in\spec R\mid\s(\p)\in B\}$.
We obtain another section-retraction pair.

\begin{prop}\label{st}
There is a section-retraction pair $\overline\s:\Spcl(\spec)\rightleftarrows\Thom:\s^{-1}$.
\end{prop}

\begin{proof}
Corollary \ref{ato2} and \cite[Proposition 2.9]{B} yield
\begin{equation}\label{spplem}
\text{$\ssupp(R/\p)=\overline{\{\s(\p)\}}$ for any prime ideal $\p$ of $R$},
\end{equation}
whence $(\overline{\{\s(\p)\}})^\complement=\u(R/\p)$, which is quasi-compact by \cite[Proposition 2.14(a)]{B}.
Hence $\overline\s(A)$ is a Thomason subset of $\spc\dm(R)$ for any subset $A$ of $\spec R$.
In particular, we get a map $\overline\s:\Spcl(\spec)\to\Thom$.

Let $T$ be a Thomason subset of $\spc\dm(R)$.
Let $\p,\q$ be prime ideals of $R$ with $\p\subseteq\q$ and $\s(\p)\in T$.
Then $\s(\q)$ belongs to $\overline{\{\s(\p)\}}$ by Theorem \ref{s} and \cite[Proposition 2.9]{B}.
Since $T$ is Thomason, it contains $\overline{\{\s(\p)\}}$.
Hence $\s(\q)$ belongs to $T$.
Thus the assignment $T\mapsto\s^{-1}(T)$ defines a map $\s^{-1}:\Thom\to\Spcl(\spec)$.

For a specialization-closed subset $W$ of $\spec R$ and a prime ideal $\p$ of $R$, one has
$$
\s(\p)\in\overline{\{\s(\q)\}}\text{ for some }\q\in W
\,\Leftrightarrow\,
\s(\p)\subseteq\s(\q)\text{ for some }\q\in W\\
\,\Leftrightarrow\,
\p\supseteq\q\text{ for some }\q\in W
\,\Leftrightarrow\,
\p\in W,
$$
where the first and second equivalences follow from \cite[Proposition 2.9]{B} and Theorem \ref{s}, and the last equivalence holds by the fact that $W$ is specialization-closed.
This yields $\s^{-1}(\overline\s(W))=W$.
\end{proof}

Now we consider describing $\spcl$-closures and $\spcl$-interiors.

\begin{prop}\label{thom}
Let $A$ be a specialization-closed subset of $\spc\dm(R)$, and let $B$ be a specialization-closed subset of $\ts\dm(R)$.
\begin{enumerate}[\rm(1)]
\item
Let $A_\spcl$ stand for the $\spcl$-interior of $A$ in $\ts\dm(R)$.
Then
$$
A_\spcl=A\cap\ts\dm(R).
$$
\item
Let $B^\spcl$ stand for the $\spcl$-closure of $B$ in $\spc\dm(R)$.
Then
$$
\textstyle
B^\spcl=\{\P\in\spc\dm(R)\mid\P^\tame\in B\}=\bigcup_{\P\in B^\spcl}\ssupp(R/\pp(\P)).
$$
In particular, $B^\spcl$ is a Thomason subset of $\spc\dm(R)$.
\end{enumerate}
\end{prop}

\begin{proof}
(1) We easily observe that $A\cap\ts\dm(R)$ is a specialization-closed subset of the topological space $\ts\dm(R)$ contained in $A$.
Also, it is obvious that if $X$ is a specialization-closed subset of $\ts\dm(R)$ contained in $A$, then $X$ is contained in $A\cap\ts\dm(R)$.
Hence $A\cap\ts\dm(R)$ coincides with $A_\spcl$.

(2) Let $C$ be the set of prime thick $\otimes$-ideals $\P$ of $\dm(R)$ with $\P^\tame\in B$.
We proceed step by step.\\
(a) Each $\Q\in B$ is tame.
Hence we have $\Q^\tame=\Q\in B$.
This shows that $C$ contains $B$.\\
(b) Let $Y$ be a specialization-closed subset of $\spc\dm(R)$ containing $B$.
Take any element $\P$ of $C$.
Then $\P^\tame$ belongs to $B$, and hence to $Y$.
Since $Y$ is specialization-closed, $\overline{\{\P^\tame\}}$ is contained in $Y$.
Hence $\P$ belongs to $Y$ by \cite[Proposition 2.9]{B}.
It follows that $C$ is contained in $Y$.\\
(c) We prove $C=\bigcup_{\P\in C}\ssupp(R/\pp(\P))$.
Combining Theorem \ref{s}, Lemma \ref{clint}(1) and \eqref{spplem} gives rise to $\ssupp(R/\pp(\P))=\overline{\{\P^\tame\}}$, and thus it is enough to verify $C=\bigcup_{\P\in C}\overline{\{\P^\tame\}}$.
By \cite[Proposition 2.9]{B} we see that $C$ is contained in $\bigcup_{\P\in C}\overline{\{\P^\tame\}}$.
Conversely, let $\P\in\C$ and $\Q\in\overline{\{\P^\tame\}}$.
Then $\P^\tame$ belongs to $B$, and $\Q$ is contained in $\P^\tame$ by \cite[Proposition 2.9]{B}, which shows that $\Q^\tame$ is contained in $\P^\tame$.
Hence $\Q^\tame$ is in $\overline{\{\P^\tame\}}\cap\ts\dm(R)$.
As $B$ is specialization-closed in $\ts\dm(R)$, it contains $\overline{\{\P^\tame\}}\cap\ts\dm(R)$, and therefore $\Q^\tame$ is in $B$.
Thus $\Q$ belongs to $C$.
We obtain $C=\bigcup_{\P\in C}\overline{\{\P^\tame\}}$.

The equality $C=\bigcup_{\P\in C}\ssupp(R/\pp(\P))$ shown in (c) especially says that $C$ is specialization-closed.
By this together with (a) and (b) we obtain $C=B^\spcl$, and it follows that $C=\bigcup_{\P\in B^\spcl}\ssupp(R/\pp(\P))$.
\end{proof}

We now obtain another section-retraction pair:

\begin{prop}\label{sst}
The operations $()^\spcl$ and $()_\spcl$ defined in Proposition \ref{thom} make a section-retraction pair $()^\spcl:\Spcl(\ts)\rightleftarrows\Thom:()_\spcl$.
\end{prop}

\begin{proof}
Let $U$ be a specialization-closed subset of $\ts\dm(R)$.
By Proposition \ref{thom}, $U^\spcl$ is a Thomason subset of $\spc\dm(R)$, and $(U^\spcl)_\spcl=U^\spcl\cap\ts\dm(R)=\{\P\in\ts\dm(R)\mid\P^\tame\in U\}=U$.
\end{proof}

We can prove our second main result in this section.

\begin{thm}\label{d2}
There is a diagram
$$
\xymatrix{
\Rad\ar@{=}[rr]^\sim\ar@<2mm>[dd]_\dashv^{{()}_\cpt} && \Thom\ar@<2mm>[dd]_\dashv^{\s^{-1}}\ar@{=}[rr] && \Thom\ar@<2mm>[dd]_\dashv^{{()}_\spcl}\\
\\
\Cpt\ar@{=}[rr]^-\sim\ar@<2mm>[uu]^{{()}^\rrad} && \Spcl(\spec)\ar@<2mm>[uu]^{\overline\s}\ar@{=}[rr]^-\sim && \Spcl(\ts)\ar@<2mm>[uu]^{{()}^\spcl}
}
$$
where the upper horizontal bijections are the one given in Proposition \ref{bal} and an equality, and the lower horizontal bijections are the ones appearing in Theorem \ref{d1}.
The diagram with vertical arrows from the bottom (resp. top) to the top (resp. bottom) is commutative.
\end{thm}

\begin{proof}
The three section-retraction pairs are obtained in Propositions \ref{cr}, \ref{st} and \ref{sst}.

We claim that for any thick $\otimes$-ideal $\X$ of $\dm(R)$ one has
\begin{equation}\label{spplem2}
\ssupp(\X^\rrad)=\ssupp\X.
\end{equation}
Indeed, Lemma \ref{clint}(1) shows $\X^\rrad=\sqrt\X$.
The inclusion $\X\subseteq\sqrt\X$ implies $\ssupp\X\subseteq\ssupp\sqrt\X$.
Let $\P$ be a prime thick $\otimes$-ideal of $\dm(R)$.
If $\X$ is contained in $\P$, then so is $\sqrt{\X}$ as $\P$ is prime.
Therefore we obtain $\ssupp\sqrt\X=\ssupp\X$, and the claim follows.

Fix a thick $\otimes$-ideal $\C$ of $\dm(R)$.
For a prime ideal $\p$ of $R$ one has $\s(\p)\in\ssupp\C$ if and only if $\C\nsubseteq\s(\p)$, if and only if $\C_\p\ne0$, if and only if $\p\in\supp\C$.
This shows $\s^{-1}(\ssupp\C)=\supp\C$.
Lemma \ref{clint}(2) gives $\supp(\C_\cpt)=\s^{-1}(\ssupp\C)$.
Next, suppose that $\C$ is compact.
Lemma \ref{clint}(1), \eqref{spplem} and \eqref{spplem2} yield
$$\ssupp(\C^\rrad)
=\ssupp\C
=\ssupp(\C_\cpt)
=\ssupp(\langle\supp\C\rangle)\\
=\ssupp\{R/\p\mid\p\in\supp\C\}=\overline\s(\supp\C).
$$
Thus we obtain the commutativity of the left square of the diagram.

Let $A$ be any subset of $\spec R$.
It is clear that $\s(A)=\{\s(\p)\mid\p\in A\}$ is contained in $\overline\s(A)$.
As $\overline\s(A)$ is a union of closed subsets of the topological space $\spc\dm(R)$, it is a specialization-closed subset of $\spc\dm(R)$.
Note that any specialization-closed subset of $\spc\dm(R)$ containing $\s(A)$ contains $\overline\s(A)$.
Hence we have $\overline\s(A)=(\s(A))^\spcl$.
Let $B$ be a specialization-closed subset of $\spc\dm(R)$.
Then $\s(\s^{-1}(B))=\{\s(\p)\mid\p\in\spec R,\,\s(\p)\in B\}=B\cap\ts\dm(R)=B_\spcl$ by Corollary \ref{tp} and Proposition \ref{thom}(1).
Now it follows that the right square of the diagram commutes.
\end{proof}

We close this section by producing another commutative diagram, coming from the above theorem.

\begin{cor}\label{d3}
There is a commutative diagram:
$$
\xymatrix{
&&&& \Rad\ar[lllldd]_{{()}_\cpt}\ar[ldd]_(.6)\supp\ar[rdd]^(.6){{()}^\tame}\ar[rrrrdd]^\sp \\
\\
\Cpt\ar@{=}[rrr]^-\sim &&& \Spcl(\spec)\ar@{=}[rr]^-\sim && \Tame\ar@{=}[rrr]^-\sim &&& \Spcl(\ts)
}
$$
Here, the three bijections are the ones appearing in Theorem \ref{d1}, and the other maps are retractions.
\end{cor}

\begin{proof}
We have the following diagram.
$$
\xymatrix{
\Rad\ar@<2mm>[d]^{()_\cpt}_\dashv\\
\Cpt\ar@<2mm>[u]^{()^\rrad}\ar@<2mm>[rr]^-\supp_-\cong && \Spcl(\spec)\ar@<2mm>[rr]^-{\supp^{-1}}_-\cong\ar@<2mm>[ll]^-{\langle\rangle} && \Tame\ar@<2mm>[rr]^-\sp_-\cong\ar@<2mm>[ll]^-\supp && \Spcl(\ts)\ar@<2mm>[ll]^-{\sp^{-1}}
}
$$
Thus it suffices to verify the equalities of compositions of maps $\supp\circ()_\cpt=\supp$, $\supp^{-1}\circ\supp=()^\tame$ and $\sp\circ()^\tame=\sp$.
This is equivalent to showing that the equalities
\begin{center}
(i) $\supp(\X_\cpt)=\supp\X$,\qquad
(ii) $\supp^{-1}\supp\X=\X^\tame$,\qquad
(iii) $\sp(\X^\tame)=\sp\X$
\end{center}
hold for each (radical) thick $\otimes$-ideal $\X$ of $\dm(R)$.
The equalities (i) and (ii) immediately follow from Lemma \ref{clint}.
We have $\sp(\X^\tame)=\sp(\supp^{-1}\supp\X)=(\sp\circ\supp^{-1})(\supp\X)=\s(\supp\X)=\sp\X$, where the first and last equalities follow from Lemma \ref{clint}(1) and \eqref{sss}.
Proposition \ref{si}(2) says that $\supp\X$ belongs to $\Spcl(\spec)$, and the third equality above is obtained by Theorem \ref{d1}.
Now the assertion (iii) follows, and the proof of the corollary is completed.
\end{proof}

\section{Distinction between thick tensor ideals, and Balmer's conjecture}\label{balconj}

In this section, we consider when the section-retraction pairs in Theorem \ref{d2} and Corollary \ref{d3} are one-to-one correspondences, and construct a counterexample to the conjecture of Balmer.
We begin with a lemma on the annihilator of an object in the thick $\otimes$-ideal closure.

\begin{lem}\label{ann}
Let $\{X_\lambda\}_{\lambda\in\Lambda}$ be a family of objects of $\dm(R)$.
For $M\in\tthick\{X_\lambda\}_{\lambda\in\Lambda}$ there are (pairwise distinct) indices $\lambda_1,\dots,\lambda_n\in\Lambda$ and integers $e_1,\dots,e_n>0$ such that $\ann M$ contains $\prod_{i=1}^n(\ann X_{\lambda_i})^{e_i}$.
\end{lem}

\begin{proof}
Let $\C$ be the subcategory of $\dm(R)$ consisting of objects $C$ such that there are $\lambda_1,\dots,\lambda_n\in\Lambda$ and $e_1,\dots,e_n>0$ such that $\ann C$ contains $\prod_{i=1}^n(\ann X_{\lambda_i})^{e_i}$.
The following statements hold in general.
\begin{itemize}
\item
If $A$ is an object of $\dm(R)$ and $B$ is a direct summand of $A$, then $\ann A\subseteq\ann B$.
\item
For each object $A\in\dm(R)$ one has $\ann(A[\pm1])=\ann A$.
\item
If $A\to B\to C\to A[1]$ is an exact triangle in $\dm(R)$, then $\ann B$ contains $\ann A\cdot\ann C$.
\item
For any objects $A,B$ of $\dm(R)$ one has $\ann(A\ltensor_RB)\supseteq\ann A$.
\end{itemize}
It follows from these that $\C$ is a thick $\otimes$-ideal of $\dm(R)$.
Since $X_\lambda$ is in $\C$ for all $\lambda\in\Lambda$, it holds that $\C$ contains $\tthick\{X_\lambda\}_{\lambda\in\Lambda}$.
The assertion of the lemma now follows.
\end{proof}

The proposition below says in particular that in the case where $R$ is a local ring $\dm(R)$ has a compact prime thick tensor ideal.
On the other hand, in the nonlocal case it is often that $\dm(R)$ has no such one.

\begin{prop}\label{no}
\begin{enumerate}[\rm(1)]
\item
If $R$ is a local ring with maximal ideal $\m$, then $\Cpt\cap\pp^{-1}(\m)=\{\zero\}\ne\emptyset$.
\item
Let $R$ be a nonlocal semilocal domain.
Then there exists no compact prime thick $\otimes$-ideal of $\dm(R)$.
In particular, one has $\P_\cpt\subsetneq\P=\P^\rrad$ for all $\P\in\spc\dm(R)$.
\end{enumerate}
\end{prop}

\begin{proof}
(1) Let $\P$ be in $\spc\dm(R)$.
Then $\P$ is in $\pp^{-1}(\m)$ if and only if $\supp\P=\{\p\in\spec R\mid\p\nsubseteq\m\}=\emptyset$ by Proposition \ref{dsd}, if and only if $\P=\zero$ by Remark \ref{r}.
Since $\zero$ is compact, we are done.

(2) Let $\m_1,\dots,\m_n$ be the (pairwise distinct) maximal ideals of $R$ with $n\ge2$.
For each $1\le i\le n$ one finds an element $x_i\in\m_i$ that does not belong to any other maximal ideals.
As $R$ is a domain of positive dimension, $x_i$ is a non-zerodivisor of $R$.
Set $C_i=\bigoplus_{t\ge0}R/x_i^{t+1}[t]$; note that this is an object of $\dm(R)$.
We have $\supp(C_1\ltensor_R\cdots\ltensor_RC_n)=\bigcap_{i=1}^n\supp C_i=\bigcap_{i=1}^n\V(x_i)=\V(x_1,\dots,x_n)=\emptyset$ by Lemma \ref{si0}(4) and the fact that $(x_1,x_2)$ is a unit ideal of $R$.
Remark \ref{r} gives $C_1\ltensor_R\cdots\ltensor_RC_n=0$.

Suppose that there exists a compact prime thick $\otimes$-ideal $\P$ of $\dm(R)$.
Then $C_1\ltensor_R\cdots\ltensor_RC_n=0$ is contained in $\P$, and so is $C_\ell$ for some $1\le\ell\le n$.
We have $\P=\langle\supp\P\rangle$ by Proposition \ref{cs}, and by Lemma \ref{ann} there exist prime ideals $\p_1,\dots,\p_r\in\supp\P$ and integers $e_1,\dots,e_r>0$ such that $\ann C_\ell$ contains $(\ann R/\p_1)^{e_1}\cdots(\ann R/\p_r)^{e_r}=\p_1^{e_1}\cdots\p_r^{e_r}$.
Since $R$ is a domain and $x_\ell$ is a non-unit of $R$, we have $\ann C_\ell=\bigcap_{t\ge0}x_\ell^{t+1}R=0$ by Krull's intersection theorem.
Therefore $\p_1^{e_1}\cdots\p_r^{e_r}=0$, and $\p_s=0$ for some $1\le s\le r$ as $R$ is a domain.
Thus the zero ideal $0$ of $R$ belongs to $\supp\P$, which implies $\supp\P=\spec R$.
We obtain $\P=\dm(R)$ by Proposition \ref{whole}, which is a contradiction.
\end{proof}

To show a main result of this section, we make two lemmas.
The first one concerns the structure of the radical and tame closures, while the second one gives an elementary characterization of artinian rings.

\begin{lem}\label{ti}
Let $\X$ be a subcategory of $\dm(R)$.
One has
$$
\textstyle\X^\rrad=\bigcap_{\X\subseteq\P\in\spc\dm(R)}\P,\quad\qquad
\X^\tame=\bigcap_{\X\subseteq\P\in\ts\dm(R)}\P.
$$
\end{lem}

\begin{proof}
Lemma \ref{clint}(1) implies $\X^\rrad=\sqrt{\tthick\X}$, which coincides with the intersection of the prime thick $\otimes$-ideals of $\dm(R)$ containing $\tthick\X$ by \cite[Lemma 4.2]{B}.
This is equal to the intersection of the prime thick $\otimes$-ideals containing $\X$, and thus the first equality holds.
As for the second equality, if $\P$ is a tame thick $\otimes$-ideal containing $\X$, then we have $\X^\tame\subseteq\P^\tame=\P$, which shows the inclusion $(\subseteq)$.
Let $M$ be an object of $\dm(R)$ belonging to all $\P\in\ts\dm(R)$ with $\X\subseteq\P$.
Corollary \ref{tp} says that $M$ is in $\s(\p)$ for all prime ideals $\p$ of $R$ with $\X\subseteq\s(\p)$.
This means that $\supp M$ is contained in $\supp\X$.
Hence $M$ is in $\supp^{-1}\supp\X$, which coincides with $\X^\tame$ by Lemma \ref{clint}(1).
Thus the second equality follows.
\end{proof}

\begin{lem}\label{smsp}
The ring $R$ is artinian if and only if for any sequence $I_1, I_2, \ldots$ of ideals of $R$ it holds that $\V(\bigcap_{n \ge 1} I_n)= \bigcup_{n \ge 1}\V(I_n)$.
\end{lem}

\begin{proof}
First of all, note that the inclusion $\V(\bigcap_{n \ge 1} I_n) \supseteq  \bigcup_{n \ge 1} \V(I_n)$ always holds.

If $R$ is artinian, then there exists an integer $m \ge 1$ such that $\bigcap_{n \ge 1} I_n = \bigcap_{j=1}^{m} I_j$.
From this we obtain $\textstyle\V(\bigcap_{n \ge 1} I_n) = \V(\bigcap_{j=1}^m I_j) = \bigcup_{j=1}^m \V(I_j) \subseteq  \bigcup_{n \ge 1} \V(I_n)$.
This shows the ``only if" part.

Let us prove the ``if" part.
Assume first that $R$ has infinitely many maximal ideals, and take a sequence $\m_1, \m_2, \dots$ of pairwise distinct maximal ideals of $R$.
By assumption, we get $\V(\bigcap_{n \ge 1}\m_n) = \bigcup_{n \ge 1} \V(\m_n)$.
Since $\V(\bigcap_{n \ge 1}\m_n)$ is a closed subset of $\spec R$, it has only finitely many minimal elements with respect to the inclusion relation.
However, $\bigcup_{n \ge 1} \V(\m_n) = \{\m_1, \m_2, \ldots\}$ has infinitely many minimal elements, which is a contradiction.
Thus, $R$ is a semilocal ring.
Let $\m_1,\dots,\m_t$ be the maximal ideals of $R$, and $J=\m_1\cap\cdots\cap\m_t$ the Jacobson radical of $R$.
Applying the assumption to the sequence $\{J^n \}_{n \ge 1}$ of ideals gives $\V(\bigcap_{n \ge 1} J^n) = \bigcup_{n \ge 1} \V(J^n)=\V(J)$.
By Krull's intersection theorem, we obtain $\bigcap_{n \ge 1} J^n=0$, whence $\V(J)= \spec R$.
Hence $\spec R = \{\m_1,\dots,\m_t\}=\Max R$, and we conclude that $R$ is artinian.
\end{proof}

Now we can prove our first main result in this section.
Roughly speaking, if our ring $R$ is artinian, then everything is explicit and behaves well, and vice versa.
Note that this result includes Corollary \ref{minmxcor2}.

\begin{thm}\label{11}
The following are equivalent.
\begin{enumerate}[\rm(1)]
\item\label{11-10}
The ring $R$ is artinian.
\item\label{11-11}
Every thick $\otimes$-ideal of $\dm(R)$ is compact, tame and radical.
\item\label{11-9}
The maps
$\ \s:\spec R\rightleftarrows\spc\dm(R):\pp\ $
are mutually inverse homeomorphisms.
\item\label{11-1}
The section-retraction pair
$\ \s:\spec R\rightleftarrows\spc\dm(R):\pp\ $
is a one-to-one correspondence.
\item\label{11-2}
The section-retraction pair
$\ ()_\cpt:\Rad\rightleftarrows\Cpt:()^\rrad\ $
is a one-to-one correspondence.
\item\label{11-3}
The section-retraction pair
$\ \s^{-1}:\Thom\rightleftarrows\Spcl(\spec):\overline\s\ $
is a one-to-one correspondence.
\item\label{11-4}
The section-retraction pair
$\ ()_\spcl:\Thom\rightleftarrows\Spcl(\ts):()^\spcl\ $
is a one-to-one correspondence.
\item\label{11-5}
The retraction
$\ \supp:\Rad\to\Spcl(\spec)\ $
is a bijection.
\item\label{11-6}
The retraction
$\ ()^\tame:\Rad\to\Tame\ $
is a bijection.
\item\label{11-7}
The retraction
$\ \sp:\Rad\to\Spcl(\ts)\ $
is a bijection.
\item\label{11-8}
The inclusion
$\ \Rad\supseteq\Tame\ $
is an equality.
\end{enumerate}
\end{thm}

\begin{proof}
Theorems \ref{s}, \ref{d1}, \ref{d2} and Corollary \ref{rtsr} imply that the pairs in \eqref{11-1}, \eqref{11-2}, \eqref{11-3}, \eqref{11-4} are section-retraction pairs, the maps in \eqref{11-5}, \eqref{11-6}, \eqref{11-7} are retractions, and one has the inclusion in \eqref{11-8}.

The equivalences \eqref{11-2} $\Leftrightarrow$ \eqref{11-3} $\Leftrightarrow$ \eqref{11-4} and \eqref{11-2} $\Leftrightarrow$ \eqref{11-5} $\Leftrightarrow$ \eqref{11-6} $\Leftrightarrow$ \eqref{11-7} follow from Theorem \ref{d2} and Corollary \ref{d3}, respectively.
It is trivial that \eqref{11-9} implies \eqref{11-1}, while \eqref{11-10} implies \eqref{11-11} by Corollaries \ref{art}, \ref{rtsr} and Proposition \ref{si}(1).
If $\spc\dm(R)=\ts\dm(R)$, then $\s=\s'$ and $\pp=\pp'$.
From Theorems \ref{kakan}(3) and \ref{homeos} we see that \eqref{11-11} implies \eqref{11-9}.
Corollary \ref{rtsr} says $\X^\tame\in\Rad$ for each $\X\in\Rad$.
Hence, if $()^\tame:\Rad\to\Tame$ is injective, then $\X=\X^\tame$ holds.
This shows that \eqref{11-6} implies \eqref{11-8}.
It is easily seen that the converse is also true, and we get the equivalence \eqref{11-6} $\Leftrightarrow$ \eqref{11-8}.
When $\s:\spec R\to\spc\dm(R)$ is surjective, we have $\spc\dm(R)=\ts\dm(R)$, and for a radical thick $\otimes$-ideal $\X$ it holds that $\X=\X^\rrad=\bigcap_{\X\subseteq\P\in\spc\dm(R)}\P=\bigcap_{\X\subseteq\P\in\ts\dm(R)}\P=\X^\tame$ by Lemma \ref{ti}, whence $\X$ is tame.
Therefore, \eqref{11-1} implies \eqref{11-8}.

Now it remains to prove that \eqref{11-8} implies \eqref{11-10}.
By Lemma \ref{smsp}, it suffices to prove that $\V(\bigcap_{n \ge 1}I_n)$ is contained in $\bigcup_{n \ge 1}\V(I_n)$ for any sequence $I_1, I_2, \dots$ of ideals of $R$.
For each $n \ge 1$, fix a system of generators $\xx(n)$ of $I_n$.
Set $C= \bigoplus_{n \ge 1}\k(\xx(n), R)[n]$; note that this is defined in $\dm(R)$.
Then $\supp C = \bigcup_{n \ge 1} \supp \k(\xx(n), R) = \bigcup_{n \ge 1} \V(I_n)$ by Proposition \ref{annihi}(3).
The radical closure $\E$ of $\langle \bigcup_{n \ge 1}\V(I_n) \rangle$ is tame by assumption.
Lemma \ref{clint} implies $\supp\E=\bigcup_{n \ge 1} \V(I_n)=\supp C$.
Thus $C$ is in $\supp^{-1}\supp\E=\E$ by Proposition \ref{tss}, and $C^{\otimes r} \in \langle \bigcup_{n \ge 1}\V(I_n) \rangle$ for some $r>0$.
Using \cite[Proposition 1.6.21]{BH}, we have
\begin{align}
C^{\otimes r}
&\label{nonn}= \textstyle\bigoplus_{n \ge 1}(\bigoplus_{i_1 + \cdots + i_r=n} \k(\xx(i_1), R) \ltensor_R \cdots \ltensor_R \k(\xx(i_r), R))[n] \\
&\gtrdot \textstyle\bigoplus_{n \ge 1} \k(\xx(n), R)^{\otimes r}[nr] 
=\textstyle\bigoplus_{n \ge 1} \k(\underbrace{\xx(n), \dots, \xx(n)}_{r}, R)[nr] 
\gtrdot \textstyle\bigoplus_{n \ge 1} \k(\xx(n), R)[nr]=:B.\nonumber
\end{align}
Thus $B$ is in $\langle \bigcup_{n \ge 1}\V(I_n) \rangle$, and Corollary \ref{strcpt}(3) implies $\V(\ann B) \subseteq \bigcup_{n \ge 1}\V(I_n)$.
We have $\ann B = \bigcap_{n \ge 1} \ann \k(\xx(n), R)= \bigcap_{n \ge 1} I_n$ by Proposition \ref{annihi}(3).
It follows that $\V(\bigcap_{n \ge 1}I_n) \subseteq \bigcup_{n \ge 1}\V(I_n)$.
\end{proof}

Our second main result in this section deals with the difference between the radical and tame closures.

\begin{thm}\label{diff}
Let $W$ be a specialization-closed subset of $\spec R$.
Set $\X=\langle W\rangle$ and $\Y=\supp^{-1}W$.
\begin{enumerate}[\rm(1)]
\item
The subcategory $\X$ is compact, and satisfies $\X^\rrad=\sqrt\X$ and $\X^\tame=\Y$.
\item
The subcategory $\X$ (resp. $\Y$) is the smallest (resp. largest) thick $\otimes$-ideal of $\dm(R)$ whose support is $W$.
In particular, one has $\X\subseteq\sqrt\X\subseteq\Y$.
\item 
Assume that $R$ is either a domain or a local ring, and that $W$ is nonempty and proper.
Then one has $\sqrt\X\subsetneq\Y$.
Hence $\Y$ is not compact, and $\X^\rrad\subsetneq\X^\tame$.
\end{enumerate}
\end{thm}

\begin{proof}
(1) The first statement is evident.
The equalities follows from Lemma \ref{clint} and Proposition \ref{si}.

(2) Let $\ZZ$ be a thick $\otimes$-ideal of $\dm(R)$ whose support is $W$.
Then it is clear that $\ZZ$ is contained in $\Y$.
Proposition \ref{key} implies that $R/\p$ belongs to $\ZZ$ for each $\p\in W$, which shows that $\ZZ$ contains $\X$.

(3) Since $W$ is nonempty, there is a prime ideal $\p\in W$.
Let $\xx=x_1,\dots,x_r$ be a system of generators of $\p$, and put $C=\bigoplus_{i\ge0}\k(\xx^{i+1},R)[i]$, which is an object of $\dm(R)$.
The support of $C$ is equal to $\V(\p)$ by Proposition \ref{annihi}(3), which is contained in $W$ as it is specialization-closed.
Hence $C$ is in $\supp^{-1}W=\Y$.

Suppose that $\sqrt\X$ coincides with $\Y$, and let us derive a contradiction.
There exists an integer $n>0$ such that the $n$-fold tensor product $D:=C\ltensor_R\cdots\ltensor_RC$ belongs to $\X$.
An analogous argument to \eqref{nonn} yields that $D$ contains $E:=\bigoplus_{k\ge0}\k(\xx^{k+1},R)[nk]$ as a direct summand, whence $E$ belongs to $\X$.
We use a similar technique to the one in the latter half of the proof of Proposition \ref{no}.
By Lemma \ref{ann}, there are prime ideals $\p_1,\dots,\p_m\in W$ and integers $e_1,\dots,e_m>0$ such that $\ann E$ contains $\p_1^{e_1}\cdots\p_m^{e_m}$.
We have
\begin{equation}
\textstyle\ann E=\bigcap_{k\ge0}\ann\k(\xx^{k+1},R)=\bigcap_{k\ge0}\xx^{k+1}R=0
\end{equation}
by Proposition \ref{annihi}(3) and Krull's intersection theorem.
This yields $\p_1^{e_1}\cdots\p_m^{e_m}=0$, which says that each prime ideal of $R$ contains $\p_i$ for some $1\le i\le m$.
As $W$ is specialization-closed, we observe that $W=\spec R$, which is contrary to the assumption.
Consequently, $\sqrt\X$ is strictly contained in $\Y$.

If $\Y$ is compact, then we have $\Y=\langle\supp\Y\rangle=\langle W\rangle=\X\subseteq\sqrt\X$ by Proposition \ref{cs} and Proposition \ref{si}(1), which is a contradiction.
Hence $\Y$ is not compact.
\end{proof}

\begin{rem}\label{repl}
\begin{enumerate}[(1)]
\item
Let $\p,C$ be as in the proof of Theorem \ref{diff}(3).
Then
\begin{enumerate}[\qquad(a)]
\item
$\supp C$ is contained in $\supp R/\p$, but $C$ does not belong to $\tthick R/\p$.
\item
$\V(\ann R)$ is contained in $\V(\ann C)$, but $R$ does not belong to $\tthick C$.
\end{enumerate}
This guarantees in Proposition \ref{key} one cannot replace $\V(\ann X)$ by $\supp X$, or $\supp\Y$ by $\V(\ann\Y)$.

Indeed, we have $\supp C=\supp R/\p=\V(\p)\subseteq W\ne\spec R$ and $\ann C=\bigcap_{i\ge0}\xx^{i+1}R=0$.
The former together with Proposition \ref{whole} shows $R\notin\tthick C$, while the latter implies $\V(\ann R)=\V(0)=\V(\ann C)$.
Assume $C$ is in $\tthick R/\p$.
Then $\ann C=0$ contains some power of $\ann R/\p=\p$ by Lemma \ref{ann}.
Hence $\V(\p)=\spec R$, which is a contradiction.
Therefore $C$ is not in $\tthick R/\p$.
\item
The assumption in Theorem \ref{diff}(3) that $R$ is either domain or local is indispensable.
In fact, let $R = A \times B$ be a direct product of two commutative noetherian rings.
Then $\spec R= \spec A \sqcup \spec B$ and $\dm(R)\cong\dm(A) \times \dm(B)$, which imply that $\supp_{\dm(R)}^{-1}(\spec A)=\dm(A)= {\langle\spec A\rangle}_{\dm(R)}$.
\item
Recall that we have the following first section-retraction pair (Proposition \ref{cr}), while Corollary \ref{rtsr} gives rise to the following second section-retraction pair.
$$
()^\rrad:\Cpt\rightleftarrows\Rad:()_\cpt,\qquad
\inc:\Tame\rightleftarrows\Rad:()^\tame.
$$
Corollary \ref{d3} implies that the left diagram below commutes.
Therefore, it is natural to ask whether the right diagram below also commutes.
$$
\xymatrix{
& \Rad\ar[ld]_{()^\cpt}\ar[rd]^{()^\tame}\\
\Cpt\ar@{=}[rr]^\sim && \Tame
}
\qquad
\xymatrix{
& \Rad\\
\Cpt\ar[ru]^{()^\rrad}\ar@{=}[rr]^\sim && \Tame\ar[lu]_\inc
}
$$
This is equivalent to asking if $(\X_\cpt)^\rrad=\X$ for all $\X\in\Tame$, and to asking if $\Y^\tame=\Y^\rrad$ for all $\Y\in\Cpt$.
Theorem \ref{diff} gives rise to a negative answer to this question.
\end{enumerate}
\end{rem}

Finally, we consider a conjecture of Balmer.
Let $\T$ be an arbitrary essentially small tensor triangulated category.
Balmer \cite{B3} constructs a continuous map
$$
\rho_\T^\bullet:\spc\T\to\hspec\RR_\T^\bullet
$$
given by $\rho_\T^\bullet(\P)=(f\in\RR_\T^\bullet\mid\cone f\notin\P)$, where $\RR_\T^\bullet=\Hom_\T(\one,\sus^\bullet\one)$ is a graded-commutative ring.
(The ideal generated by a subset $S$ of a ring $A$ is denoted by $(S)$.)
Recall that a triangulated category is called {\em algebraic} if it arises as the stable category of some Frobenius exact category.
Balmer \cite[Conjecture 72]{B4} conjectures the following.

\begin{conj}[Balmer]\label{balm}
The map $\rho_\T^\bullet$ is (locally) injective when $\T$ is algebraic.
\end{conj}

\noindent
Here, recall that a continuous map $f:X\to Y$ of topological spaces is called {\em locally injective at $x\in X$} if there exists a neighborhood $N$ of $x$ such that the restriction $f|_N:N\to Y$ is an injective map.
We say that $f$ is {\em locally injective} if it is locally injective at every point in $X$.
If for any $x\in X$ there exists a neighborhood $E$ of $f(x)$ such that the induced map $f^{-1}(E)\to E$ is injective, then $f$ is locally injective.

Let us consider the above conjecture for our tensor triangulated category $\dm(R)$.
It turns out that for $\T=\dm(R)$, Balmer's constructed map $\rho_\T^\bullet$ coincides with our constructed map $\pp:\spc\dm(R)\to\spec R$.

\begin{prop}\label{ato}
Let $\P$ be a prime thick $\otimes$-ideal of $\dm(R)$.
One then has the following.

{\rm(1)} $\pp(\P)=(a\in R\mid R/a\notin\P)=\{a\in R\mid R/a\notin\P\}$.\quad\qquad
{\rm(2)} $\pp(\P)=\rho_{\dm(R)}^\bullet(\P)$.
\end{prop}

\begin{proof}
Corollary \ref{3'} and (1) imply (2).
Let us show (1).
Set $J=(a\in R\mid R/a\notin\P)$.
As $R$ is noetherian, we find a finite number of elements $x_1,\dots,x_n$ with $R/x_1,\dots,R/x_n\notin\P$ and $J=(x_1,\dots,x_n)$.
Therefore $\k(x_1,\dots,x_n,R)=\k(x_1,R)\ltensor_R\cdots\ltensor_R\k(x_n,R)$ is not in $\P$ by Corollary \ref{3'} and the fact that $\P$ is prime.
Using Corollary \ref{3'} again shows $J\in\ii(\P)$, whence $J$ is contained in $\pp(\P)$.
Next, take any $a\in\pp(\P)$.
Since $\V(\pp(\P))$ is not contained in $\supp\P$, neither is $\V(a)$.
This implies $R/a\notin\P$ by Corollary \ref{3'}.
\end{proof}

As an application of our Theorem \ref{diff}, we confirm that Conjecture \ref{balm} is not true in general; our $\dm(R)$ is an algebraic triangulated category, but does not satisfy Conjecture \ref{balm} under quite mild assumptions:

\begin{cor}\label{cntex}
Assume that $R$ has positive dimension, and that $R$ is either a domain or a local ring.
Then the map $\pp:\spc\dm(R)\to\spec R$ is not locally injective.
Hence Conjecture \ref{balm} does not hold.
\end{cor}

\begin{proof}
We can choose a nonunit $x\in R$ such that the ideal $xR$ of $R$ has positive height (hence it has height $1$).
Put $\X=\langle\V(x)\rangle$.
Using Theorem \ref{diff}(3) and Lemma \ref{ti}, we find a prime thick $\otimes$-ideal $\P$ such that $\X\subseteq\P\subsetneq\P^\tame$.
Suppose that $\pp$ is locally injective at $\P$.
Then there exists a complex $M\in\dm(R)$ with $\P\in\u(M)$ such that the restriction $\pp|_{\u(M)}:\u(M)\to\spec R$ is injective.
Since $M$ is in $\P$, it is also in $\P^\tame$.
Hence both $\P$ and $\P^\tame$ belong to $\u(M)$.
However, these two prime thick $\otimes$-ideals are sent by $\pp$ to the same point; see Theorem \ref{s}.
This contradicts the injectivity of $\pp|_{\u(M)}$, and we conclude that $\pp$ is not locally injective at $\P$.
The last assertion of the corollary follows from Proposition \ref{ato}(2).
\end{proof}

\begin{rem}
The reader may think that Corollary \ref{cntex} can also be obtained by showing that the map
$$
f:\spc\dm(R)\to\spc\kb(\proj R),\quad \P\mapsto\P\cap\kb(\proj R)
$$
is not injective.
We are not sure whether the non-injectivity of the map $f$ implies Corollary \ref{cntex}, but at least showing the non-injectivity of $f$ is equivalent to our approach:
Using Proposition \ref{key}, we see that $\P\cap\kb(\proj R)$ contains the Koszul complex of a system of generators of each prime ideal belonging to $\supp\P$.
Hence $\supp(\P\cap\kb(\proj R))=\supp\P$, and the Hopkins--Neeman theorem implies $\P\cap\kb(\proj R)=\supp_{\kb(\proj R)}^{-1}\supp\P$.
Therefore, for $\P,\Q\in\spc\dm(R)$ it holds that
$$
f(\P)=f(\Q)\ \Longleftrightarrow\ \supp\P=\supp\Q,
$$
which says that the map $f$ is injective if and only if all the prime thick tensor ideals of $\dm(R)$ are tame.
In the end, even if we intend to prove Corollary \ref{cntex} by showing the non-injectivity of the map $f$, we must find a non-tame prime thick tensor ideal of $\dm(R)$, which is what we have done in this section.
\end{rem}

\section{Thick tensor ideals over discrete valuation rings}\label{sect:dvr}

In this section, we concentrate on handling the case where $R$ is a discrete valuation ring.
Several properties that are specific to this case are found out in this section.
Just for convenience, we write complexes as chain complexes, rather than as cochain complexes.
We start by studying complexes with zero differentials.

\begin{prop}\label{kasane}
Let $X=\bigoplus_{i\ge0}X_i[i]=(\cdots\zs X_3\zs X_2\zs X_1\zs X_0\to0)$ be a complex in $\dm(R)$.
Then it holds that $\tthick X=\tthick Y$ in $\dm(R)$, where
$$
Y=\textstyle\bigoplus_{i\ge0}(\bigoplus_{j=0}^iX_j)[i]=(\cdots\zs X_3\oplus X_2\oplus X_1\oplus X_0\zs X_2\oplus X_1\oplus X_0\zs X_1\oplus X_0\zs X_0\to0).
$$
\end{prop}

\begin{proof}
Putting $F=\bigoplus_{j\ge0}R[j]$, we have $X\ltensor_RF=(\bigoplus_{i\ge0}X_i[i])\ltensor_R(\bigoplus_{j\ge0}R[j])=\bigoplus_{i,j\ge0}X_i[i+j]=Y$.
Hence $\tthick X$ contains $\tthick Y$.
The opposite inclusion also holds as $X$ is a direct summand of $Y$.
\end{proof}

\begin{prop}\label{zero}
Let $X=\bigoplus_{i\ge0}X_i[i]=(\cdots\zs X_3\zs X_2\zs X_1\zs X_0\to0)$ be a complex in $\dm(R)$.
Then for all integers $a_i\ge0$, the thick $\otimes$-ideal closure $\tthick X$ in $\dm(R)$ contains
$$
\textstyle\bigoplus_{i\ge0}X_i^{\oplus a_i}[2i]=(\cdots\to X_3^{\oplus a_3}\to0\to X_2^{\oplus a_2}\to0\to X_1^{\oplus a_1}\to0\to X_0^{\oplus a_0}\to0).
$$
\end{prop}

\begin{proof}
In the category $\dm(R)$ the complex $\bigoplus_{i\ge0}X_i^{\oplus a_i}[2i]=\bigoplus_{i\ge0}(X_i\ltensor_RR^{\oplus a_i})[2i]$ is a direct summand of $\bigoplus_{i,j\ge0}(X_i\ltensor_RR^{\oplus a_j})[i+j]=(\bigoplus_{i\ge0}X_i[i])\ltensor_R(\bigoplus_{j\ge0}R^{\oplus a_j}[j])=X\ltensor_RY$, where $Y=\bigoplus_{j\ge0}R^{\oplus a_j}[j]=(\cdots\zs R^{\oplus a_2}\zs R^{\oplus a_1}\zs R^{\oplus a_0}\to0)$ is a complex in $\dm(R)$.
Thus the assertion follows.
\end{proof}

\begin{cor}\label{kougo}
Let $X=\bigoplus_{i\ge0}X_i[i]=(\cdots\zs X_3\zs X_2\zs X_1\zs X_0\to0)$ be a complex in $\dm(R)$.
Then for any integers $a_i\ge0$ the complex
$$
Y=\textstyle\bigoplus_{i\ge0}X_i^{\oplus a_i}[i]=(\cdots\zs X_3^{\oplus a_3}\zs X_2^{\oplus a_2}\zs X_1^{\oplus a_1}\zs X_0^{\oplus a_0}\to0)
$$
is in $\tthick\{X_\even,X_\odd\}$, where $X_\even=\bigoplus_{i\ge0}X_{2i}[i]=(\cdots\zs X_6\zs X_4\zs X_2\zs X_0\to0)$ and $X_\odd=\bigoplus_{i\ge0}X_{2i+1}[i]=(\cdots\zs X_7\zs X_5\zs X_3\zs X_1\to0)$.
\end{cor}

\begin{proof}
The complex $Y$ is the direct sum of $A=(\cdots\to0\to X_4^{\oplus a_4}\to0\to X_2^{\oplus a_2}\to0\to X_0^{\oplus a_0}\to0)$ and $B=(\cdots\to X_5^{\oplus a_5}\to0\to X_3^{\oplus a_3}\to0\to X_1^{\oplus a_1}\to0\to0)$.
Proposition \ref{zero} shows that $A$ is in $\tthick X_\even$ and $B$ is in $\tthick X_\odd$.
Therefore $Y$ belongs to $\tthick\{X_\even,X_\odd\}$.
\end{proof}

A natural question arises from Proposition \ref{zero} and Corollary \ref{kougo}:

\begin{ques}
Does $\tthick(\cdots\to0\to X_2\to0\to X_1\to0\to X_0\to0)$ contain $(\cdots\zs X_2\zs X_1\zs X_0\to0)$?
Does $\tthick(\cdots\zs X_1\zs X_0\to0)$ contain $(\cdots\zs X_1^{\oplus a_1}\zs X_0^{\oplus a_0}\to0)$ for all integers $a_i\ge0$?
\end{ques}

We do not know the general answer to this question.
The following example gives an affirmative answer.

\begin{ex}
Let $(R,xR)$ be a discrete valuation ring.
Then
$$
\tthick(\cdots\zs R/x^3\zs R/x^2\zs R/x\to0)
=\tthick(\cdots\to0\to R/x^3\to0\to R/x^2\to0\to R/x\to0).
$$
\end{ex}

\begin{proof}
In fact, the inclusion $(\supseteq)$ follows from Proposition \ref{zero}.
To check the inclusion $(\subseteq)$, set $A=(\cdots\zs R/x^3\zs R/x^2\zs R/x\to0)$ and $B=(\cdots\to0\to R/x^3\to0\to R/x^2\to0\to R/x\to0)$.
Note that for each integer $n\ge0$ there is an exact sequence $0\to R/x^n\xrightarrow{x^{n+1}}R/x^{2n+1}\to R/x^{n+1}\to0$ of $R$-modules.
This induces an exact sequence $0\to C \to A\to B\to0$ of complexes of $R$-modules, where
$$
C=(\cdots\zs R/x^n\zs R/x^{2n}\zs R/x^{n-1}\zs R/x^{2(n-1)}\zs\cdots\zs R/x^2\zs R/x^4\zs R/x\zs R/x^2\to0).
$$
We see that $C=B[2]\oplus D$, where $D=(\cdots\to0\to R/x^{2n}\to0\to\cdots\to0\to R/x^4\to0\to R/x^2\to0)$, and have an exact sequence $0 \to B[1]\to D\to B[1]\to0$ of complexes.
The assertion now follows.
\end{proof}

The {\em Loewy length} of a finitely generated $R$-module $M$, denoted by $\ell\ell_R(M)$, is by definition the infimum of integers $i$ such that the ideal $(\rad R)^i$ kills $M$.
Let us consider thick $\otimes$-ideals defined by Loewy lengths.

\begin{nota}\label{dl}
Let $R$ be a local ring with maximal ideal $\m$.
Let $c\ge0$ be an integer.
\begin{enumerate}[(1)]
\item
Let $\L_c$ be the subcategory of $\df(R)$ consisting of complexes $X$ such that there exists an integer $t\ge0$ with $\ell\ell(\h_i X)\le ti^{c-1}$ for all $i\gg0$.
\item
When $c\ge1$, let $G_c$ be the complex $\bigoplus_{i>0}(R/\m^{i^{c-1}})[i]=(\cdots\zs R/\m^{3^{c-1}}\zs R/\m^{2^{c-1}}\zs R/\m\to0)$.
\end{enumerate}
\end{nota}

\begin{prop}\label{L_0}
Let $(R,\m,k)$ be local.
One has $\L_0\subsetneq\L_1\subsetneq\L_2\subsetneq\cdots$ and $\L_0=\dbf(R)=\thick_{\dm(R)}k$.
\end{prop}

\begin{proof}
Fix an integer $n\ge0$.
It is clear that $\L_n$ is contained in $\L_{n+1}$.
We have $\ell\ell(\h_iG_{n+1})=i^n$ for each $i\ge0$, which shows $\L_n\ne\L_{n+1}$.
Hence the chain $\L_0\subsetneq\L_1\subsetneq\L_2\subsetneq\cdots$ is obtained.
Let $X$ be a complex in $\dm(R)$.
Suppose that there exists an integer $t\ge0$ such that $\ell\ell(\h_i X)\le ti^{-1}$ for $i\gg0$.
Then we have to have $\ell\ell(\h_i X)=0$ for $i\gg0$, which says that $\h_j X=0$ for $j\gg0$.
Thus we obtain $\L_0=\dbf(R)=\thick_{\dm(R)}k$, where the second equality is shown in Proposition \ref{si2}.
\end{proof}

Recall that an abelian category $\A$ is called {\em hereditary} if it has global dimension at most one, that is, if $\Ext_\A^2(\A,\A)=0$.
Recall also that a ring $R$ is called {\em hereditary} if $R$ has global dimension at most one.

From now on, we study thick $\otimes$-ideals of $\dm(R)$ when $R$ is local and hereditary.
In this case, $R$ is either a field or a discrete valuation ring.
If $R$ is a field, then by Corollary \ref{art} there are only trivial thick $\otimes$-ideals.
So, we mainly consider the case of a discrete valuation ring.
First, we mention a well-known fact, saying that each complex in the derived category of a hereditary abelian category has zero differentials.

\begin{lem}\label{hered}\cite[1.6]{K}
Let $\A$ be a hereditary abelian category.
Then for each object $M\in\d(\A)$ there exists an isomorphism $M\cong\h(M)=\bigoplus_{i\in\Z}\h_i(M)[i]$ in $\d(\A)$.
\end{lem}

The lemma below is part of our first main result in this section.

\begin{lem}\label{lc}
Let $R$ be a discrete valuation ring.
Then $\L_c$ is a thick $\otimes$-ideal of $\dm(R)$ for every $c\ge1$.
\end{lem}

\begin{proof}
By Proposition \ref{si}(3), it suffices to show $\L_c$ is a thick $\otimes$-ideal of $\df(R)$.
We do this step by step.\\
(1) Take any complex $X$ in $\L_c$.
There exist integers $t,u\ge0$ such that $\ell\ell(\h_i X)\le ti^{c-1}$ for all $i\ge u$.
Let $Y$ be a direct summand of $X$ in $\df(R)$.
Then $\h_i Y$ is a direct summand of $\h_i X$, and we have $\ell\ell(\h_i Y)\le\ell\ell(\h_i X)\le ti^{c-1}$ for all $i\ge u$.
Hence $Y$ belongs to $\L_c$.\\
(2) Let $X\to Y\to Z\rightsquigarrow$ be an exact triangle in $\df(R)$.
Suppose that both $X$ and $Z$ belong to $\L_c$.
Then there exist integers $t,u,a,b\ge0$ such that $\ell\ell(\h_i X)\le ti^{c-1}$ and $\ell\ell(\h_j Z)\le uj^{c-1}$ for all $i\ge a$ and $j\ge b$.
An exact sequence $\cdots\to\h_k X\to\h_k Y\to\h_k Z\to\cdots$ is induced, and from this we see that $\ell\ell(\h_k Y)\le\ell\ell(\h_k X)+\ell\ell(\h_k Z)\le(t+u)k^{c-1}$ for all $k\ge\max\{a,b\}$.
Therefore, $Y$ belongs to $\L_c$.\\
(3) Let $X$ be a complex in $\L_c$.
Then there exist integers $t,u\ge0$ such that $\ell\ell(\h_i X)\le ti^{c-1}$ for all $i\ge u$.
It holds that $\ell\ell(\h_i (X[1]))=\ell\ell(\h_{i-1} X)\le t(i-1)^{c-1}\le ti^{c-1}$ for all $i\ge u+1$ for all $i\ge u+1$, where the second inequality holds as $c\ge1$.
Also, $\ell\ell(\h_i (X[-1]))=\ell\ell(\h_{i+1}X)\le t(i+1)^{c-1}\le t(i+i)^{c-1}=(2^{c-1}t)\cdot i^{c-1}$ for all $i\ge\max\{1,u-1\}$, where the first inequality holds as $i\ge u-1$, and the second one holds since $i\ge1$ and $c\ge1$.
Thus the complexes $X[1]$ and $X[-1]$ belong to $\L_c$.\\
(4) Let $X,Y$ be complexes in $\df(R)$.
Suppose that $X$ belongs to $\L_c$.
We want to show that $X\ltensor_RY$ also belongs to $\L_c$.
Taking into account (3) and Lemma \ref{hered}, we may assume that $X=\bigoplus_{i\ge1}X_i[i]$ and $Y=\bigoplus_{j\ge0}Y_j[j]$ with $X_i,Y_j$ being $R$-modules, and that there exist $s\ge1,t\ge0$ such that $\ell\ell(X_i)\le ti^{c-1}$ for all $i\ge s$.
Set $u=\max\{\ell\ell(X_i)\mid1\le i\le s-1\}$; note that each $X_i$ has finite length, whence has finite Loewy length.
We have $X\ltensor_RY=\bigoplus_{i\ge1,\,j\ge0}(X_i\ltensor_RY_j)[i+j]$, and from this we get $\h_k (X\ltensor_RY)=\bigoplus_{i\ge1,\,j\ge0,\,i+j\le k}\Tor_{k-i-j}^R(X_i,Y_j)$ for all integers $k$.
Note here that $\Tor_{k-i-j}^R(X_i,Y_j)=0$ for $i+j>k$.

We claim that $\ell\ell(X_i)\le(t+u)i^{c-1}$ for all $i\ge1$.
In fact, recall $c\ge1$ and $t,u\ge0$.
If $i\ge s$, then $\ell\ell(X_i)\le ti^{c-1}\le(t+u)i^{c-1}$.
If $1\le i\le s-1$, then $\ell\ell(X_i)\le u\le t+u\le(t+u)i^{c-1}$.
The claim follows.

Fix three integers $i,j,k$ with $i\ge1$, $j\ge0$ and $i+j\le k$.
Then $(t+u)k^{c-1}\ge(t+u)i^{c-1}$ since $k\ge i$ and $c\ge1$.
The claim shows that $X_i$ is killed by $\m^{(t+u)k^{c-1}}$, and so is $\Tor_{k-i-j}^R(X_i,Y_j)$, where $\m$ stands for the maximal ideal of $R$.
Hence $\ell\ell(\h_k (X\ltensor_RY))\le(t+u)k^{c-1}$ for all $k\in\Z$, which implies $X\ltensor_RY\in\L_c$.

It follows from the above arguments (1)--(4) that $\L_c$ is a thick $\otimes$-ideal of $\df(R)$.
\end{proof}

\begin{rem}
Let $(R,\m,k)$ be a local ring.
When $c=0$, the subcategory $\L_c$ is never a thick $\otimes$-ideal of $\dm(R)$.
Indeed, by Proposition \ref{L_0} we have $\L_0=\dbf(R)$.
The module $k$ is in $\L_0$, but the complex $(\cdots\zs k\zs k\to0)=k\ltensor_R(\cdots\zs R\zs R\to0)$ is not in $\L_0$. 
\end{rem}

Now we have our first theorem concerning the subcategories $\L_c$ of $\dm(R)$ for a discrete valuation ring $R$.
This especially says that the equality of Proposition \ref{pr}(2) does not necessarily hold.

\begin{thm}\label{nonnoeth}
Let $R$ be a discrete valuation ring.
Then $\L_c$ is a prime thick $\otimes$-ideal of $\dm(R)$ for all integers $c\ge1$.
In particular, one has
$$
\dim(\spc\dm(R))=\infty>1=\dim R.
$$
\end{thm}

\begin{proof}
Lemma \ref{lc} says that $\L_c$ is a thick $\otimes$-ideal of $\dm(R)$.
Proposition \ref{L_0} especially says $\L_c\ne\dm(R)$.
Let $X,Y$ be complexes in $\dm(R)$ with $X\ltensor_RY\in\L_c$, and we shall prove that either $X$ or $Y$ is in $\L_c$.
Applying Lemma \ref{hered} and taking shifts if necessary, we may assume $X=\bigoplus_{i\ge0}X_i[i]$ and $Y=\bigoplus_{j\ge0}Y_j[j]$, where $X_i,Y_j$ are finitely generated $R$-modules.
Assume that $X$ is not in $\df(R)$.
Then $X_a$ has infinite length for some $a\ge0$.
As $R$ is a discrete valuation ring, $X_a$ has a nonzero free direct summand.
Hence $R[a]$ is a direct summand of $X$, and $Y[a]=R[a]\ltensor_RY$ is a direct summand of $X\ltensor_RY$.
As $X\ltensor_RY$ is in $\L_c$, so is $Y$.
Similarly, if $Y\notin\df(R)$, then $X\in\L_c$.
This argument shows that we may assume that both $X$ and $Y$ belong to $\df(R)$, or equivalently, that all $X_i$ and $Y_j$ have finite length as $R$-modules.
Since $X\ltensor_RY$ belongs to $\L_c$, there exist integers $t,u\ge0$ such that $\h_n(X\ltensor_RY)$ has Loewy length at most $tn^{c-1}$ for all $n\ge u$.
Assume that $X$ is not in $\L_c$.
Then we can find an integer $e\ge u$ such that $\ell\ell(X_e)>te^{c-1}$.
We have $X\ltensor_RY=\bigoplus_{i,j\ge0}(X_i\ltensor_RY_j)[i+j]$, which gives rise to $\h_n (X\ltensor_RY)=\bigoplus_{i,j\ge0}\Tor_{n-i-j}(X_i,Y_j)$ for all integers $n$.
Setting $a_i=\ell\ell(X_i)$ and $b_j=\ell\ell(Y_j)$ for $i,j\ge0$, we obtain for every integer $n\ge e$:
$$
\h_n (X\ltensor_RY)\gtrdot\Tor_{n-e-(n-e)}(X_e,Y_{n-e})=X_e\otimes_RY_{n-e}\gtrdot R/x^{a_e}\otimes_RR/x^{b_{n-e}}=R/x^{\min\{a_e,b_{n-e}\}}
$$
It is seen that $\min\{a_e,b_{n-e}\}\le tn^{c-1}$ for all $n\ge e$.
As $a_e>te^{c-1}$, we must have $a_e>b_{n-e}$, and $b_{n-e}\le tn^{c-1}$ for all $n\ge e$.
Hence $\ell\ell(\h_n (Y[e]))=\ell\ell(Y_{n-e})=b_{n-e}\le tn^{c-1}$ for $n\ge e$, which implies that $Y[e]$ is in $\L_c$, and so is $Y$.
Similarly, if $Y$ is not in $\L_c$, then $X$ is in $\L_c$.
Thus $\L_c$ is a prime thick $\otimes$-ideal of $\dm(R)$.
Now $\L_1\subsetneq\L_2\subsetneq\L_3\subsetneq\cdots$ from Lemma \ref{lc} is an ascending chain of prime thick $\otimes$-ideals with infinite length, which shows the inequality in the proposition; see Proposition \ref{pr}(1).
\end{proof}

To make an application of the above theorem, we state and prove a lemma.

\begin{lem}\label{essur}
For each prime ideal $\p$ of $R$, one has $\dim \spc \dm(R_\p) \le \dim \spc \dm(R)$.
\end{lem}

\begin{proof}
We first show that the localization functor $L: \dm(R) \to \dm(R_\p)$ is an essentially surjective.
Let $X =(\cdots \xrightarrow{d_2} X_1 \xrightarrow{d_1} X_0 \to 0)$ be a complex in $\dm(R_\p)$.
What we want is a complex $Y\in\dm(R)$ such that $X \cong L(Y)$.
For each integer $i \ge 0$, choose a finitely generated $R$-module $Y_i$ with $(Y_i)_\p=X_i$, and $R$-linear maps $d^Y_i:Y_i \to Y_{i-1}$ and $s_i \in R \setminus \p$ such that $d^X_i=\frac{d^Y_i}{s_i}$ in $\Hom_{R_\p}(X_i, X_{i-1})=\Hom(Y_i, Y_{i-1})_\p$. 
Then $\frac{d^Y_{i-1}d^Y_i}{s_{i-1} s_i}= d^X_{i-1}d^X_i=0$, and there is an element $t_i \in R \setminus \p$ such that $t_i d^Y_{i-1} d^Y_i=0$.
Define a complex $Y=(\cdots \xrightarrow{t_{i+1} d^Y_{i-1}} Y_i \xrightarrow{t_i d^Y_i} \cdots \xrightarrow{t_2 d^Y_2} Y_1  \xrightarrow{t_1 d^Y_1} Y_0 \to 0)$ in $\dm(R)$.
Then there is an isomorphism
$$
\xymatrix{
Y_\p\ar[d] & = & (\cdots\ar[r] & (Y_i)_\p\ar[r]^{\frac{t_id^Y_i}{1}}\ar[d]^{u_i}_\cong & (Y_{i-1})_\p\ar[r]\ar[d]^{u_{i-1}}_\cong & \cdots\ar[r] & (Y_2)_\p\ar[r]^{\frac{t_2d^Y_2}{1}}\ar[d]^{u_2}_\cong & (Y_1)_\p\ar[r]^{\frac{t_1d^Y_1}{1}}\ar[d]^{u_1}_\cong & (Y_0)_\p\ar[r]\ar@{=}[d] & 0) \\
X & = & (\cdots\ar[r] & X_i\ar[r]^{d^X_i} & X_{i-1}\ar[r] & \cdots\ar[r] & X_2\ar[r]^{d^X_2} & X_1\ar[r]^{d^X_1} & X_0\ar[r] & 0),
}
$$
of complexes, where $u_i := t_1 \cdots t_i s_1 \cdots s_i$.  
Thus, we obtain $L(Y) =Y_\p \cong X$.

The essentially surjective tensor triangulated functor $L$ induces an injective continuous map $\spc L: \spc \dm(R_\p) \to \spc \dm(R)$ given by $\P \mapsto L^{-1}(\P)$; see \cite[Corollary 3.8]{B}.
This map sends a chain $\P_0 \subsetneq \cdots \subsetneq \P_n$ of prime thick $\otimes$-ideals of $\dm(R_\p)$ to the chain $L^{-1}(\P_0)\subsetneq\cdots\subsetneq L^{-1}(\P_n)$ of prime thick $\otimes$-ideals of $\dm(R)$.
The lemma now follows.
\end{proof}

The following corollary of Theorem \ref{nonnoeth} provides a class of rings $R$ such that the Balmer spectrum of $\dm(R)$ has infinite Krull dimension.
This class includes normal local domains for instance.

\begin{cor}\label{mugen}
If $R_\p$ is regular for some $\p$ with $\height\p>0$, then $\dim\spc\dm(R)=\infty$.
\end{cor}

\begin{proof}
We may assume $\height\p=1$.
We have $\dim\spc\dm(R)\ge\dim\spc\dm(R_\p)=\infty$, where the inequality follows from Lemma \ref{essur}, and the equality is shown in Theorem \ref{nonnoeth}.
\end{proof}

Next we study generation of the thick tensor ideals $\L_c$.
In fact each of them possesses a single generator.

\begin{thm}\label{poly}
Let $(R,xR,k)$ be a discrete valuation ring, and let $c\ge1$ be an integer.
It then holds that $\L_c=\ttthick_{\dm(R)}G_c$.
In particular, one has $\L_1=\ttthick_{\dm(R)}k$.
\end{thm}

\begin{proof}
Clearly, $G_c$ is in $\L_c$.
Lemma \ref{lc} implies that $\tthick G_c$ is contained in $\L_c$.
We establish a claim.
\begin{claim*}
Let $0\le n\le c-1$ be an integer.
Let $X\in\df(R)$ be a complex.
Suppose that there exists an integer $t\ge0$ such that $\ell\ell(\h_i X)\le ti^n$ for all $i\gg0$.
Then $X$ belongs to $\tthick G_c$.
\end{claim*}
\noindent
Once we show this claim, it will follow that $\L_c$ is contained in $\tthick G_c$, and we will be done.

First of all, note that $k$ is a direct summand of $G_c$.
Combining this with Proposition \ref{si2}, we have
\begin{equation}\label{tdbf}
\tthick G_c\supseteq\tthick k\supseteq\thick k=\dbf(R).
\end{equation}
Let $X$ be a complex as in the claim.
Using Lemma \ref{hered}, we may assume $X=\bigoplus_{i\ge s}X_i[i]$ for some integer $s$ and $R$-modules $X_i$ of finite length.
There is an integer $u\ge s$ with $\ell\ell(X_i)\le ti^n$ for all $i\ge u$.
We have $X=(\bigoplus_{i\ge u}X_i[i])\oplus(\bigoplus_{i=s}^{u-1}X_i[i])$, whose latter summand is in $\dbf(R)$.
In view of \eqref{tdbf}, replacing $X$ with the former summand, we may assume $u=s$.
When $s\ge0$, we set $X_i=0$ for $0\le i\le s-1$.
When $s<0$, we have $X=(\bigoplus_{i\ge0}X_i[i])\oplus(\bigoplus_{i=s}^{-1}X_i[i])$, whose latter summand is in $\dbf(R)$.
By similar replacement as above, we may assume $s=0$.
Thus, $X=\bigoplus_{i\ge0}X_i[i]$ and $\ell\ell(X_i)\le ti^n$ for all $i\ge0$.

Since $R$ is a discrete valuation ring with maximal ideal $xR$, for every $i\ge1$ there is an integer $a_{ij}\ge0$ such that $X_i$ is isomorphic to $\bigoplus_{j=1}^{ti^n}(R/x^j)^{\oplus a_{ij}}$.
Therefore it holds that
\begin{align*}
\textstyle X
&\textstyle\cong\bigoplus_{i\ge0}(\bigoplus_{j=1}^{ti^n}(R/x^j)^{\oplus a_{ij}})[i]
\textstyle\lessdot\bigoplus_{i\ge0}(\bigoplus_{j=1}^{ti^n}R/x^j)^{\oplus a_i}[i]\\
&\textstyle\in\tthick\left\{\bigoplus_{i\ge0}(\bigoplus_{j=1}^{t(2i)^n}R/x^j)[i],\,\bigoplus_{i\ge0}(\bigoplus_{j=1}^{t(2i+1)^n}R/x^j)[i]\right\}
=\tthick\left\{A_1,A_2\oplus(\bigoplus_{j=1}^tR/x^j)\right\},
\end{align*}
where $a_i:=\max\{a_{ij}\mid 1\le j\le ti^n\}$ and $A_l:=\bigoplus_{i\ge1}(\bigoplus_{j=t(2i-l)^n+1}^{t(2i-l+2)^n}R/x^j)[i]$ for $l=1,2$.
The relations ``$\in$'' and ``$=$'' follow from Corollary \ref{kougo} and Proposition \ref{kasane}, respectively.
Since $\bigoplus_{j=1}^tR/x^j$ is in $\tthick G_c$ by \eqref{tdbf}, it suffices to show that $A_l$ belongs to $\tthick G_c$ for $l=1,2$.

We prove this by induction on $n$.
When $n=0$, we have $A_1=A_2=0\in\tthick G_c$, and are done.
Let $n\ge1$.
Fix $l=1,2$.
The exact sequences
$$
0\to R/x^{t(2i-l)^n}\xrightarrow{x^j}R/x^{j+t(2i-l)^n}\to R/x^j\to0\qquad(i\ge1,\,1\le j\le tb_{il})
$$
with $b_{il}=(2i-l+2)^n-(2i-l)^n$ induce exact sequences
$$
\textstyle0 \to (R/x^{t(2i-l)^n})^{\oplus tb_{il}}\to \bigoplus_{j=t(2i-l)^n+1}^{t(2i-l+2)^n}R/x^j \to \bigoplus_{j=1}^{tb_{il}}R/x^j \to 0\qquad(i\ge1),
$$
which induce an exact triangle $B_l \to A_l \to C_l\rightsquigarrow$ in $\df(R)$, where we set $B_l=\bigoplus_{i\ge1}(R/x^{t(2i-l)^n})^{\oplus tb_{il}}[i]$ and $C_l=\bigoplus_{i\ge1}(\bigoplus_{j=1}^{tb_{il}}R/x^j)[i]$.
Since $\ell\ell(\h_i C_l)=tb_{il}$ has degree at most $n-1$ as a polynomial in $i$, the induction hypothesis implies that $C_l$ is in $\tthick G_c$.
By Corollary \ref{kougo}, $B_l$ belongs to
$$
\textstyle\tthick\{\bigoplus_{i\ge0}(R/x^{t(4i+r)^n})[i]\mid 0\le r\le 3\}.
$$

Let $f(i)$ be a polynomial in $i$ over $\NN$ with leading term $ei^n$.
The exact sequences
$$
0 \to R/x^{(t-1)f(i)} \xrightarrow{x^{f(i)}} R/x^{tf(i)} \to R/x^{f(i)} \to 0\qquad(i\ge0)
$$
induce an exact triangle $D_{t-1} \to D_t \to D_1 \rightsquigarrow$ in $\df(R)$, where we put $D_t=\bigoplus_{i\ge0}R/x^{tf(i)}[i]$.
An inductive argument on $t$ shows that $D_t$ belongs to the thick closure of $D_1$.
The exact sequences
$$
0 \to R/x^{f(i)-(m+1)i^n}\xrightarrow{x^{i^n}}R/x^{f(i)-mi^n}\to R/x^{i^n}\to0\qquad(i\ge0)
$$
induce an exact triangle $E_{m+1}\to E_m\to G_c\rightsquigarrow$, where we set $E_m=\bigoplus_{i\ge0}(R/x^{f(i)-mi^n})[i]$ for $0\le m\le e$.
Hence $E_0$ is in the thick closure of $G_c$ and $E_e$.
Since $\ell\ell(\h_i E_e)=f(i)-ei^n$ has degree at most $n-1$ as a polynomial in $i$, the induction hypothesis shows that $E_e$ is in $\tthick G_c$.
Hence $D_1=E_0$ is also in $\tthick G_c$, and so is $D_t$.
Therefore $B_l$ is in $\tthick G_c$.
Thus $A_l$ belongs to $\tthick G_c$ for $l=1,2$.
\end{proof}

\begin{rem}\label{repl2}
Let $(R,xR,k)$ be a discrete valuation ring, and let $c\ge2$ be an integer.
Then $\supp G_c=\{xR\}=\supp k$.
In particular, we have $\supp G_c\ne\spec R$, so that $R$ is not in $\tthick G_c$ by Proposition \ref{whole}.
Krull's intersection theorem implies $\ann G_c=0=\ann R$.
Proposition \ref{L_0} and Theorem \ref{poly} imply that $G_c$ is not in $\L_1=\tthick k$.
In summary:

\begin{enumerate}[\quad(1)]
\item
$\supp G_c$ is contained in $\supp k$, but $G_c$ does not belong to $\tthick k$.
\item
$\V(\ann R)$ is contained in $\V(\ann G_c)$, but $R$ does not belong to $\tthick G_c$.
\end{enumerate}
This guarantees that in Proposition \ref{key} one cannot replace $\V(\ann X)$ by $\supp X$, or $\supp\Y$ by $\V(\ann\Y)$.
\end{rem}

\begin{ex}
Let us deduce the conclusion of Proposition \ref{no}(1) directly in the case where $(R,\m,k)$ is a discrete valuation ring.
In this case, we have $\Spcl(\spec)=\{\emptyset, \{\m\}, \spec R\}$.
Using Proposition \ref{cs}, we obtain $\Cpt=\{\zero, \tthick k,\dm(R)\}$.
Example \ref{zerop} and Theorems \ref{poly}, \ref{nonnoeth} say that $\zero$ and $\tthick k$ are prime.
Thus the compact prime thick $\otimes$-ideals of $\dm(R)$ are $\zero$ and $\tthick k$.
It follows from Corollary \ref{ato2} that $\pp(\tthick k)$ does not contain $\m$, which implies $\pp(\tthick k)=0$.
Hence $\Cpt\cap\pp^{-1}(\m)=\{\zero\}$.
\end{ex}

Let us consider for a discrete valuation ring $R$ the tameness and compactness of the thick $\otimes$-ideals $\L_c$.

\begin{prop}\label{a}
Let $R$ be a discrete valuation ring, and let $c\ge1$ be an integer.
Then $\L_c$ is a non-tame prime thick $\otimes$-ideal of $\dm(R)$.
If $c\ge2$, then $\L_c$ is non-compact.
\end{prop}

\begin{proof}
It is shown in Theorem \ref{nonnoeth} that $\L_c$ is a prime thick $\otimes$-ideal of $\dm(R)$.
Denote by $xR$ the maximal ideal of $R$.
Using Proposition \ref{L_0} and Theorem \ref{poly}, we easily see that $\supp\L_c=\V(x)=\{xR\}$.

Suppose that $\L_c$ is tame.
Then $\L_c=\supp^{-1}\{xR\}$ by Proposition \ref{tss}.
For example, consider the complex $E=\bigoplus_{i\ge0}(R/x^{i!})[i]$.
We have $\supp E=\{xR\}$, which shows $E\in\L_c$.
Hence there exists an integer $t\ge0$ such that $i!=\ell\ell(\h_i E)\le ti^{c-1}$ for all $i\gg0$.
This contradiction shows that $\L_c$ is not tame.

Suppose that $\L_c$ is compact.
Then $\L_c=\langle\supp\L_c\rangle=\tthick k=\L_1$ by Proposition \ref{cs} and Theorem \ref{poly}.
This gives a contradiction when $c\ge2$; see Proposition \ref{L_0}.
Thus $\L_c$ is not compact for all $c\ge2$.
\end{proof}

\begin{rem}
Theorem \ref{poly} implies that $\L_c$ is generated by the complex $G_c$, whose support is the closed subset $\{\m\}$ of $\spec R$.
Corollary \ref{a} says that $\L_c$ is not compact for $c\ge2$.
This gives an example of a non-compact thick $\otimes$-ideal which is generated by objects with closed supports.
\end{rem}

In the proof of Proposition \ref{a}, a complex defined by using factorials of integers played an essential role.
In relation to this, a natural question arises.

\begin{ques}
Let $(R,xR)$ be a discrete valuation ring.
Consider the complex
$$
\textstyle
E=\bigoplus_{i\ge0}(R/x^{i!})[i]=(\cdots\zs R/x^{120}\zs R/x^{24}\zs R/x^{6}\zs R/x^{2}\zs R/x\zs R/x\to0)
$$
in $\dm(R)$.
Is it possible to establish a similar result to Theorem \ref{poly} for $\tthick E$?
For example, can one characterize the objects of $\tthick E$ in terms of the Loewy lengths of their homologies?
\end{ques}

We have no idea to answer this question.
In relation to it, in the next example we will consider complexes defined by using not factorials but polynomials.
To do this, we provide a lemma.

\begin{lem}\label{lin}
Let $x$ be a non-zerodivisor of $R$.
Then the complex $\bigoplus_{i\ge0}(R/x^{a_i+b_i})[i]$ belongs to the thick closure of $\bigoplus_{i\ge0}(R/x^{a_i})[i]$ and $\bigoplus_{i\ge0}(R/x^{b_i})[i]$ for all integers $a_i,b_i\ge0$.
In particular, the complex $\bigoplus_{i\ge0}(R/x^{ca_i})[i]$ is in the thick closure of $\bigoplus_{i\ge0}(R/x^{a_i})[i]$ for all integers $c,a_i\ge0$.
\end{lem}

\begin{proof}
For each $i\ge0$ there is an exact sequence $0 \to R/x^{a_i} \xrightarrow{x^{b_i}} R/x^{a_i+b_i} \to R/x^{b_i} \to 0$.
From this an exact sequence $0 \to \bigoplus_{i\ge0}(R/x^{a_i})[i] \to \bigoplus_{i\ge0}(R/x^{a_i+b_i})[i] \to \bigoplus_{i\ge0}(R/x^{b_i})[i] \to 0$ is induced.
The first assertion follows from this.
The second assertion is shown by induction and the first assertion.
\end{proof}

\begin{ex}\label{abc}
Let $x\in R$ be a non-zerodivisor.
For integers $a,b,c\ge0$, define a complex
$$
X(a,b,c)=\textstyle\bigoplus_{i\ge0}(R/f_i)[i]=(\cdots\zs R/f_2\zs R/f_1\zs R/f_0\to0),
$$
where $f_i=x^{ai^2+bi+c}\in R$.
Then it holds that $\tthick\{X(a,b,c)\mid a,b,c\ge0\}=\tthick\{X(1,0,0)\}$.
\end{ex}

\begin{proof}
It is obvious that the left-hand side contains the right-hand side.
In view of Lemma \ref{lin}, the opposite inclusion will follow if we show that $X(1,0,0),X(0,1,0),X(0,0,1)$ are in $\tthick\{X(1,0,0)\}$, whose first containment is evident.
The complex $X(1,0,0)$ has the direct summand $(R/x)[1]$, so the module $R/x$ belongs to $\tthick\{X(1,0,0)\}$.
We have $X(0,0,1)=R/x\ltensor_R(\cdots\zs R\zs R\to0)$, which is in $\tthick\{X(1,0,0)\}$.
The exact sequences $0 \to R/x^{i^2} \xrightarrow{x^{2i+1}} R/x^{(i+1)^2} \to R/x^{2i+1} \to 0$ and $0 \to R/x^{2i+1} \xrightarrow{x} R/x^{2i+2} \to R/x \to 0$ with $i>0$ induce exact sequences $0 \to X(1,0,0) \to X(1,0,0)[-1] \to X(0,2,1) \to 0$ and $0 \to X(0,2,1) \to X(0,2,2) \to X(0,0,1) \to 0$, which shows that $\tthick\{X(1,0,0)\}$ contains $X(0,2,1)=(\cdots\zs R/x^5\zs R/x^3\zs R/x\to0)$ and $X(0,2,2)=(\cdots\zs R/x^6\zs R/x^4\zs R/x^2\to0)$.
Applying Corollary \ref{kougo}, we see that $X(0,1,0)$ belongs to $\tthick\{X(1,0,0)\}$.
\end{proof}

\begin{rem}
One can consider a general statement of Example \ref{abc} by defining $f_i=x^{a_0i^d+a_1i^{d-1}+\cdots+a_d}$, so that it is nothing but the example for $d=2$.
We do not know if it holds for $d\ge3$.
\end{rem}


\end{document}